\documentclass{amsart}
 \usepackage{amsaddr}
 
\usepackage{graphicx, nicefrac}

\usepackage[utf8]{inputenc}
\usepackage[T1]{fontenc}
\usepackage[a4paper,left=2.5cm,right=2.5cm,top=2.5cm,bottom=2.5cm]{geometry}
\usepackage{libertine}
\usepackage{graphicx}
\usepackage{float}
\usepackage{amsmath}
\usepackage{amsfonts}
\usepackage{multicol}
\usepackage{array}
\usepackage{multirow}
\usepackage{listings}
\usepackage{appendix}
\usepackage{subfigure}
\usepackage{mathtools}
\usepackage{bigints}
\usepackage{caption}
\usepackage{amsthm}
\usepackage{bbold}
\usepackage{verbatim}
\usepackage[shortlabels]{enumitem}
\usepackage{tikz-cd}

\makeatletter
\newcommand{\stars}{}% just for safety
\DeclareRobustCommand{\stars}[1]{\stars@{#1}}
\newcommand{\stars@}[1]{%
  \ifcase#1\relax\or\stars@one\or\stars@two\or\stars@three\or\stars@four
  \else ??\fi
}
\newcommand{\stars@char}{$\scriptstyle*$}
\newcommand{\stars@base}[1]{%
  $\m@th\vcenter{\offinterlineskip\ialign{\hfil##\hfil\cr#1\crcr}}$%
}
\newcommand{\stars@one}{%
  \stars@base{\stars@char}%
}
\newcommand{\stars@two}{%
  \stars@base{\stars@char\cr\stars@char}%
}
\newcommand{\stars@three}{%
  \stars@base{\stars@char\cr\stars@char\stars@char}%
}
\newcommand{\stars@four}{%
  \stars@base{\stars@char\stars@char\cr\stars@char\stars@char}%
}
\makeatother

\usepackage{amsmath}
\usepackage{cancel}
\usepackage{amssymb}
\usepackage[dvips]{epsfig}
\usepackage{graphicx}
\usepackage{latexsym}
\usepackage{amsthm}
\usepackage{amssymb}
\usepackage{ulem}

\usepackage{amsmath,amsthm,amsfonts,amssymb,amscd,amsbsy,dsfont,hyperref}

\usepackage{palatino}

\usepackage{eucal}
\usepackage{float}
\usepackage{xcolor}

\usepackage{xurl}

\usepackage{color}

\usepackage{multirow}

\newcommand*\dif{\mathop{}\!\mathrm{d}}

\newcommand{\Prob}{\mathbb{P}}

\newcommand{\R}{\mathbb{R}}

\newcommand{\segN}[1]{ [\![ #1 ]\!] }

\newcommand{\diag}{\text{Diag}}

\newcommand{\LogLik}[1]{ \mathcal{L}( #1 ) }

\DeclareMathOperator*{\argmin}{arg\,min}

\newtheorem{remark}{Remark}
\newtheorem{corollary}{Corollary}
\newtheorem{definition}{Definition}
\newtheorem{lemma}{Lemma}
\newtheorem{proposition}{Proposition}
\newtheorem{theorem}{Theorem}

\newtheorem{thm}{Theorem}
\newtheorem{Proposition}{Proposition}

\newtheorem{assumption}[thm]{Assumption}

\newtheorem{defi}[thm]{Definition}

\newtheorem*{assumption*}{Assumption}

\begin{document}

\title{ A sparsity test for multivariate Hawkes processes \vspace{-2em}}
\author[]{\footnotesize	 antoine lotz \vspace{-0.5em}}
\address{\footnotesize	 \textsc{Paris-Dauphine University}, \textsc{PSL}, Paris, France
\\\vspace{-0.2em}
\textsc{EDF}-\textsc{Lab} , \textsc{EDF} \textsc{R}\&\textsc{D}, Palaiseau, France}
\email{antoine.lotz@dauphine.eu}
\begin{abstract}
    Multivariate Hawkes processes (\textsc{mhp}) are a class of point processes in which events at different coordinates interact through mutual excitation.  The weighted adjacency matrix of the \textsc{mhp} encodes the strength of the relations, and shares its support with the causal graph of interactions of the process. We consider the problem of testing for causal relationships across the dimensions of a marked \textsc{mhp}. The null hypothesis is that a joint group of adjacency coefficients are null, corresponding to the absence of interactions. The alternative is that they are positive, and the associated interactions do exist. To this end, we introduce a novel estimation procedure in the context of a large sample of independent event sequences. We construct the associated likelihood ratio test and derive the asymptotic distribution of the test statistic as a mixture of $\chi^2$ laws. We offer two applications on financial datasets to illustrate the performance of our method.  In the first one, our test reveals a deviation from a static equilibrium in bidders' strategies on retail online auctions. In the second one, we uncover some factors at play in the dynamics of German intraday power prices.  \\

\\
   
\end{abstract}

\maketitle

\begin{center}
    
\end{center}

\section{Introduction}

\subsection{Motivation} Hawkes processes are self-exciting point processes in which the  probability of an event appearing increases in the aftermath of past occurrences. Multivariate Hawkes processes (\textsc{mhp}) extend this property to interdependent sequences, wherein events in a coordinate of the process are also contingent upon the past of other coordinates. They provide an ideal model for the dynamics of interacting systems and have demonstrated this capacity in diverse applications ranging from neuroscience~\cite{Reynaud}~\cite{Martinez}, to high-frequency financial data~\cite{bacryBompaireGaiffas},  or  information propagation on media networks~\cite{HawkesTwitter}. In this work, we consider the problem of testing for temporal influence between the coordinates of a marked \textsc{mhp}. The hypotheses for our test express naturally in terms of characteristics of the process, and we informally introduce the relevant features of the model hereafter. \\

Suppose one observes a $K$-th dimensional counting process $(\boldsymbol{N}_t)=(N_{k,t})$ over a fixed time-frame $[0,T]$.  The jumps of the $k$-th coordinate occur at distinct times $(t^k_i)$ with $1 \leq i \leq N^k_T$, and to each jump time $t_i^k$ is associated a mark $X_i^k$ with values in $\R^m$. The law of $(\boldsymbol{N}_t)$ is characterised by its predictable intensity process $(\lambda_{1,t}, \cdots, \lambda_{K,t})$, where each $\lambda_{k,t}$ can be thought of as the instantaneous probability of an event arriving at coordinate $k$ at time $t$.  In this paper, we are interested in the class of linear marked \textsc{mhp}s, which describes processes with intensity of the form

\begin{equation}
    \lambda_{k,t} 
    =
    \mu_k \Big( \frac{t}{T}\Big)
    +
    \sum_{k=1}^K \sum_{t^l_i < t}
    \varphi_{kl}(t-t^l_i,X^l_i),
    \label{equ:intensity}
\end{equation}

where  the $\varphi_{kl} \colon [0,\infty) \times \R^m \mapsto [0,\infty)$ are integrable functions referred to as kernels of the Hawkes process. The intensity process~\eqref{equ:intensity} is constructed so that an event at coordinate $l$ at time $s$ increases the intensity at coordinate $k$ at time $t$ by $\varphi_{kl}(t-s)$.  The influence between different coordinates of the process are thus encoded into $\varphi=(\varphi_{kl})$, and selecting interactions simply amounts to selecting null kernels with the question typically phrased in terms of the $L^1$ norms of the kernels. To this aim, a significant effort has been dedicated  to penalization-based approaches, with a variety of criteria being successfully developed for least-square (Hansen \textit{et al.}~\cite{Hansen}, Bacry \textit{et al.}~\cite{bacryBompaireGaiffas}),  and maximum likelihood estimation ( Zhou \textit{et al.}~\cite{ZhouZha}, Xu \textit{et al.}~\cite{Xu}, Salehi~\cite{Salehi}, Sulem \textit{et al.}~\cite{sulem2023bayesian}  and Goda~\cite{Goda}). Testing for causal influence is a natural adjacent question to variable selection for (possibly marked) \textsc{mhp}s. Yet, the topic  seems to have attracted a more modest interest. The only available test is a multiple comparison procedure introduced by Bonnet \textit{et al.}~\cite{Martinez} as a by-product of their estimation method for inhibiting \textsc{mhp}s. Such methods however incur a change in statistical power and the full theoretical background for their results remains an open question.

\subsection{Contribution and organisation of the paper}  Given a subset $\mathcal{J} \subset [\![1,K]\!]^2$ of coordinate pairs, we construct the likelihood-ratio test for the null hypothesis
\begin{equation}
    \int_0^\infty \varphi_{kl}(s) \dif s=0, \hspace{0.1cm} \text{for every  } (k,l) \in \mathcal{J} , \label{equ:what_we_testH0}
\end{equation}

against the one-sided alternative
\begin{equation}
    \int_0^\infty \varphi_{kl}(s) \dif s>0, \hspace{0.1cm} \text{for every  } (k,l) \in \mathcal{J} . \label{equ:what_we_testH1}
\end{equation}

Our configuration differs from the usual situation where a single long time trajectory of the \textsc{mhp} is recorded. We observe instead $n$ independent realisations of $(\boldsymbol{N}_t)$ over $[0,T]$, where $T$ is fixed and $n$ large.  Such data structures are commonly found in biology  (see for example Reynaud-Bouret~\cite{Reynaud}) and they are well fitted to the financial applications we have in mind, which involve short-maturity illiquid products (as will be detailed further down this section). Following an aggregation procedure, this setting results in the large baseline regime introduced by Chen \& Hall~\cite{FengInferenceForNonStationarySEPP} in the univariate case. More precisely, we work in the parametric setting where $\mu$ and $\varphi$ depend on some $\theta \in \Theta \subset \R^p$ and we are concerned with processes which intensities take the form
\begin{equation*}
    \lambda_{k,t}^{(n)}(\theta)
    =
    n \mu_k\Big(\frac{t}{T}, \theta \Big) + \sum_{k=1}^K \sum_{t_i^l < t} \varphi_{kl}(t - t_i^l,X_i^l,\theta).
\end{equation*}

The necessary assumptions are detailed in section~\ref{section:model_assumptions} along a rigorous definition of the \textsc{mhp}.\\

Our main result is Theorem~\ref{thm:distribution_of_lambda_as_a_projection} in section~\ref{section:large_n_main_results}, stating the convergence of the likelihood ratio toward a difference of squared distances from a normal variable to two sub-spaces depending on~\eqref{equ:what_we_testH0} and~\eqref{equ:what_we_testH1}.  Going from the univariate to the multivariate incurs some specific new difficulties, as the kernels possibly being null translate into the true parameter lying at a boundary of $\Theta$. In turn, this gives rise to singularities in the statistics of interest. The asymptotic distribution of the likelihood ratio is in particular defined as a chi-bar ( $\Bar{\chi}^2$) law, as introduced by Kudo~\cite{Kudo} and Shapiro~\cite{Shapiro}. The $\Bar{\chi}^2$ arises naturally from constraints on a parameter space and is a recurrent occurrence in statistics of random processes, see for instance Andrews~\cite{AndrewsDF}~\cite{AndrewsGarch} for an application to \textsc{garch} models.  We refer also to Silvapulle \& Sen~\cite{Silvapulle} for a comprehensive review on constrained statistical inference. In our case, the $\Bar{\chi}^2$ expresses as a mixture of standard $\chi^2$ laws. Though the weights are generally intractable, we show they have a simple expression in most of the low-dimensional use cases. In particular, our procedure yields a straightforward test for the Poisson \textit{versus} self-exciting problem introduced by Dachian \& Kutoyants~\cite{Kuto} in a different setting. In the general case, the evaluation of the weights can be entirely foregone in favour of a conditional test based on Susko's~\cite{Susko} method for $\Bar{\chi}^2$-tests, with negligible power loss. As a corollary of our results and still in section~\ref{section:large_n_main_results}, we recover the consistency and asymptotic distribution of the maximum likelihood estimator (\textsc{mle}), answering a question of Chen \& Hall~\cite{FengInferenceForNonStationarySEPP} regarding the extension of their work to \textsc{mhp}s.  Hypotheses of the type~\eqref{equ:what_we_testH0} describe the most general forms for the parameter space, and are in many cases a desirable property. The non-regular distribution of the \textsc{mle} is in this sense a generic feature of the linear \textsc{mhp}. While the boundary parameter issue has been raised in other works (see Goda~\cite{Goda} for the particular example of an exponential \textsc{mhp} in a long time regime), the derivation of the entire asymptotic distribution of the \textsc{mle} and likelihood ratio is a seemingly new result.  \\

In section~\ref{section:numerical} we provide ample numerical illustrations of our procedure at work, including both simulation studies and applications.  In a first example,  we show how our test may be used to infer agents' behaviour in online auctions. Specifically, we provide evidence that interactions between agents are a significant factor in bid arrivals, therefore shifting the transaction price away from the static Vickrey equilibrium. In a second example, we study the price formation of hourly power futures on the German intraday market. We find that in addition to the reversion classically accounted for by \textsc{mhp}-based models in finance, the price mechanics display a form of inertia. We offer an explanation to our finding in terms of liquidity constraints.   Our two real-world examples pertain to illiquid assets. The configuration of sparsely populated event sequences has been the subject of a relatively limited body of work (see Salehi~\cite{Salehi}), and this type of application had consequently remained at the margin of the range of \textsc{mhp}s-based model. Our examples demonstrate the good performance of our approach in this context.\\

In section~\ref{section:preparation_for_proofs} are the proofs. Before moving to our assumptions, we provide a brief overview of our strategy. We follow a classical \textsc{m}-estimator approach. The initial steps of the proof involve showing the convergence of the log-likelihood after appropriate normalisation and centering.  This requires proving the law of large number
\begin{equation}\label{equ:pres_LLN}
    \frac{1}{n} \big\{ \lambda_t^{(n)}(\theta) - \mathbb{E}[ \lambda_t^{(n)}(\theta)] \big\}
    =
    o_{\Prob}(1)
\end{equation}

holds uniformly in $t$. Up to some technicalities, the proof of~\eqref{equ:pres_LLN} does not require specific adjustments to deal with the specificities of the multivariate case. The main complication stems instead from satisfying the conditions for Geyer's~\cite{Geyer}[Theorem 4.4] constrained \textsc{m}-estimator master theorem, which is the key to the weak convergence of the likelihood ratio. In particular, we require a stochastic equicontinuity condition for the second order remainder of the log-likelihood, as classically introduced by Pollard~\cite{Pollard}[Chapter VII, section I].  We show that verifying this assumption can be reduced to precising the rate of convergence in~\eqref{equ:pres_LLN}, in the sense that its suffices to have 
\begin{equation}\label{equ:pres_TCL}
    \frac{1}{\sqrt{n}}
    \big\{  \lambda_t^{(n)}(\theta) - \mathbb{E}[ \lambda_t^{(n)}(\theta)] \big\}
    =
    O_{\Prob}(1)
\end{equation}

uniformly in $t$ and $\theta$. The proof for~\eqref{equ:pres_TCL} is slightly more involved than that of~\eqref{equ:pres_LLN}. We proceed in the manner of Pollard~\cite{Pollard}[Chapter VII, section II] and use a chaining technique to achieve the uniformity in $\theta$. To this end, we need a sub-exponential bound on the probability that the left-hand side of~\eqref{equ:pres_TCL} deviates as $\theta$ moves. A fitting bound is achieved using a martingale concentration inequality due to van de Geer~\cite{vandeGeerconcentration}[Lemma 2.1] among others. We can then apply Geyer's~\cite{Geyer} result, yielding theorem~\ref{thm:distribution_of_lambda_as_a_projection} and the rest of our results.\\

\section{Model and assumptions}\label{section:model_assumptions}

\subsection{Setting and notation.} Consider a rich enough probability space on which, for each $k \in \segN{1,K}$ , the event times $(t^k_i)$ introduced above are defined as a sequence of increasing positive random variables and required to be distinct, meaning that jumps as different coordinates may not occur simultaneously. To each event time $t^k_i$ is associated a  mark $(X^k_i)$ with values in $\R^m$ and law $F^k(\dif x)$. The $(X_i^k)$ form an i.i.d collection of random variables.  The multivariate counting process  $(\boldsymbol{N}_t)=( N_{1,t}, \cdots, N_{K,t})$  for the $(t^k_i)$ is then defined for any $t \geq 0$ and $k \in \segN{1,K}$ as
\begin{equation*}
    N_{k,t} = \sum_{i \in \mathbb{N}} \mathbb{1}_{ \{ t^k_i \leq t \} },
\end{equation*}

and the counting measure $N_k(\dif s, \dif x)$ is the random Borel measure from $\mathcal{B}(\R^+) \otimes \mathcal{B}(\R)^{\otimes m} $ to $\R^+ $
\begin{equation*}
    N_k \colon A \times B  
    \mapsto
    \sum_{i \in \mathbb{N}} 
    \mathbb{1}_{ \{t^k_i \in A \} \cap \{  X^k_i \in B \}}.
\end{equation*}

Denote by $\mathbb{F} = (\mathcal{F}_t)$ the natural filtration of $(\boldsymbol{N}_t)$ and consider a $\mathbb{F}$-predictable coordinate-wise positive process $\lambda_t= (\lambda_{1,t}, \cdots, \lambda_{K,t})$ . From Jacod~\cite{Jacod1975MultivariatePP}, under mild technical assumptions we may choose our probability space such that $(\boldsymbol{N}_t)$ admits $ (\boldsymbol{\Lambda}_t) = ( \int_0^t \lambda_s \dif s)$ as its compensator, that is, as the unique predictable increasing process such that $( \boldsymbol{N}_t - \boldsymbol{\Lambda}_t)$ is a $\mathbb{F}$-local martingale.  We refer to the comprehensive textbook of Daley \& Vere-Jones~\cite{DVJ} for more details on processes with stochastic intensity.  This equips us with a characterisation of the (marked) \textsc{mhp}.

\begin{definition}[multivariate marked Hawkes process]\label{def:marked_hawkes_process} \phantom \\
Let $T>0$. The process $(\boldsymbol{N}_t)_{t \in [0,T]}$ is said to be a multivariate marked Hawkes process if, for any $t \in [0,T]$, its intensity takes the form
    \begin{align*}
    (\lambda_{k,t})_k
    =
    \Big(
    \mu_k \Big( \frac{t}{T} \Big)
    +
    \sum_{l =1}^K
    \int_{[0,t) \times \R^m} \varphi_{kl}(t-s,x) N_l( \dif s, \dif x) 
    \Big)_k,
    \end{align*}
    concisely written,
    \begin{align*}
    \lambda_t
    =
    \mu \Big( \frac{t}{T} \Big)
    +
    \int_{[0,t) \times \R^m} \varphi(t-s,x) \boldsymbol{N}( \dif s, \dif x),
    \end{align*}
    where $(\mu_k) \colon [0,1] \mapsto (0,\infty)^K$ is the baseline, and  $(\varphi_{kl}) \colon \R^+ \times \R^m \mapsto \mathcal{M}_K(\R^+)$ is  the kernel of the process.
\end{definition}

Now denote by $\boldsymbol{M} (\dif s, \dif x)= (M_k(\dif s, \dif x))$ the compensated measure
\begin{equation*}
    \boldsymbol{M}(\dif s, \dif x) \vcentcolon = N_k(\dif s, \dif x) - \lambda_{k,s} \dif s F_k(\dif x).
\end{equation*}

 Define also $\dif \boldsymbol{N}_s =  \boldsymbol{N}(\dif s, \R^m) $,  and $\dif \boldsymbol{M}_s= \boldsymbol{M}(\dif s, \R^m) $. Then, from Definition~\ref{def:marked_hawkes_process},
\begin{equation*}
    (\boldsymbol{M}_t) = (\int_0^t \dif \boldsymbol{M}_s),
    \hspace{0.2cm} t \in [0,T], 
\end{equation*}
    
is a local martingale. While Jacod's~\cite{Jacod1975MultivariatePP} theorem suffices to guarantee the existence of $(\boldsymbol{N}_t)$, note that one can always explicitly construct the \textsc{mhp} such that the no simultaneous jump condition we required above is satisfied. See for instance Hawkes \& Oakes~\cite{Oakes} and Delattre \textit{et al.}~\cite{DelattreHawkesLargeNetworks} for two equivalent representations, respectively as as cluster process and a Poisson \textsc{sde} solution. The jumps being distinct entails some useful simplification for the variation of the local martingale $(\boldsymbol{M}_t)$, as transcribed in remark~\ref{remark:diagonal_covariation} below.

\begin{remark}\label{remark:diagonal_covariation}
    The quadratic variation $[\boldsymbol{M},\boldsymbol{M}]$ and predictable variation $\langle \boldsymbol{M},\boldsymbol{M}\rangle$ of $(\boldsymbol{M}_t)$ are both diagonal and respectively given by $\diag( N_{k,t} )$ and $\diag( \int_0^t \lambda_k(s) \dif s)$.
\end{remark}

We need more notation. Throughout, we denote by $\lVert \cdot \lVert_p$ the $\ell^p$-norm on $\R^K$ and by $\lVert \cdot \lVert_{L^1}$ the $L^1(\R^+)$-norm  on the positive integrable functions $\lVert \phi \lVert_{L^1} = \int_0^\infty \phi(s) \dif s$. The set of squared real $K \times K$ matrices is labelled $\mathcal{M}_K(\R)$ and for any $M \in \mathcal{M}_K(\R)$, $\rho(M)$ is the spectral radius of $M$. For any open set $U \subset \R$, the set $\mathcal{C}^0(U)$ designates the continuous functions from $U$ to $\R$, and likewise $\mathcal{C}^1(U)$ designates the differentiable functions from $U$ to $\R$. Finally, for any $f=(f_{kl})  \colon  \R^m \mapsto  \mathcal{M}_K(\R^+)$, we employ the notation
\begin{align*}
   F(\dif x)=\diag(F^j(\dif x)) \hspace{0.5cm} \text{and} \hspace{0.5cm}
   \langle 
   F,f
   \rangle 
   =
   \int_{\R^m} f(\cdot,x) F(\dif x),
\end{align*}
 and write $f \in L^2(F(\dif x))$ when %cette formulation existe bien
 $ \forall i,j \in \segN{1,K}^2 \int_\R f_{ij}^2(x)  F^j(\dif x) < \infty$.  

 \subsection{Estimation procedure and assumptions}\label{section:assumptions}

 As mentioned in the introduction, our setting differs from the usual observation of a single Hawkes process. We suppose instead that $n$ independent realisations $( \boldsymbol{N}^1_t )$, ..., $(\boldsymbol{N}^n_t)$ of the same Hawkes process with baseline $\mu$ and kernel $\varphi$ are recorded on a fixed time interval $[0,T]$. Now, by linearity of the intensity the sum process 
 \begin{equation}\label{equ:aggregation}
     \big( \boldsymbol{N}^{(n)}_t \big)
     \vcentcolon =
     \Big( \sum_{i=1}^n \boldsymbol{N}^i_t \Big),  \hspace{0.1cm} t \in [0,T],
 \end{equation}

 is also a Hawkes process and admits $n \mu$ and $\varphi$ as a its baseline and kernel.
\begin{remark}\label{remark:clusters_addition}
     This specificity of linear \textsc{mhp}s can be linked to the cluster representation of Hawkes \& Oakes~\cite{Oakes}. Summing up linear Hawkes processes is equivalent to aggregating  on the same timeline independent clusters from different realisations. Since cluster arrivals are Poissonian and thus additive, this leads to the baseline of the resulting process being the sum of the baseline of each realisation.
 \end{remark}
 
 We therefore recover Chen \& Hall~\cite{FengInferenceForNonStationarySEPP}'s framework of a baseline tending to infinity.  Though  their work does not appear to take advantage of the previous observation, the univariate results they introduce extend immediately to our configuration through the aggregation device ~\eqref{equ:aggregation}, and we construct our assumptions as minimal extension of their own. Following the preceding remarks, we work in a parametric setting and we consider a collection $\boldsymbol{N}^{(n)}$ of marked \textsc{mhp}s indexed by $n \in \mathbb{N}^\star$, with intensities
\begin{equation}\label{equ:parameterised_intensity}
     \lambda^{(n)}_t(\theta)
     =
     n \mu \left( \frac{t}{T}, \theta \right)
     +
     \int_0^t  \varphi(t-s,x, \theta ) 
     \boldsymbol{N}^{(n)}(\dif s, \dif x),
     \hspace{0.2cm}
     t \in [0,T].
 \end{equation}
under the probability measure $\Prob(\theta)$ where $\theta$ lies in a compact subset $\Theta \subset \R^d$ with non empty interior. To simplify the expression for our test, we assume $\theta$ separates into $\theta=(\gamma, \alpha, \beta)$ where $\mu$ depends only on $\gamma \in \R^K$ and $\varphi$ expresses as  \begin{equation}\label{equ:adjacency}
     \varphi_{kl}( \cdot, \cdot, \theta) = \alpha_{kl} \Phi_{kl}(\cdot  , \cdot,  \beta),
 \end{equation}

 where $\alpha \in \mathcal{M}_K(\R^+)$ and for $k,l$ and any value of $\beta$
\begin{equation}\label{equ:phi_is_a_density}
     \int_{[0,\infty) \times \R^m} \Phi_{kl}(s, x,\beta ) F_l(\dif x) \dif s =1. 
 \end{equation}
 
  Expression~\eqref{equ:adjacency} and~\eqref{equ:phi_is_a_density} are consistent with parameterisations found in the literature (see for instance Ogata~\cite{OgataETAS} or Hardiman~\cite{HardimanBouchaud}).  The support of the matrix $\alpha$, that is, the matrix $h = (h_{kl}) \in \mathcal{M}_K(\{0,1\})$ with $h_{kl} = 0 \iff \alpha_{kl}=0$, is  then known as the adjacency of the \textsc{mhp} (see for instance Sulem~\cite{sulem2023bayesian}). Accordingly, $\alpha$ is known as its weighted adjacency (see Bacry \textit{et al.}~\cite{bacryBompaireGaiffas}). The null hypothesis conveniently re-phrases in terms of the matrix parameter $\alpha$ as
\begin{equation}
    \label{equ:what_we_testH0a}
     \alpha_{kl}=0, \hspace{0.1cm} \text{for every} (k,l)  \hspace{0.1cm}  \in J,
\end{equation}

and the alternative as
\begin{equation}
    \label{equ:what_we_testH1a}
     \alpha_{kl}>0, \hspace{0.1cm} \text{for every} (k,l)  \hspace{0.1cm}  \in J.
\end{equation}

 We will use formulations~\eqref{equ:what_we_testH0}, ~\eqref{equ:what_we_testH1} and  \eqref{equ:what_we_testH0a},\eqref{equ:what_we_testH1a} interchangeably.  
 
 \begin{remark}
     The notion of "\textit{adjacency}" is chosen to reflect a salient property of \textsc{mhp}s in terms of graphical models. Using the definition of Granger-causality for point processes introduced by Didelez~\cite{Didelez}, Xu \textit{et al.}~\cite{Xu} show that an arrow exists from coordinate $l$ to coordinate $k$ in the causal graph of the \textsc{mhp} if and only if $\alpha_{kl} >0$.  In this regard, our test can be understood as a causality test. 
 \end{remark}
 
 Additionally to the parametric specification above, the baseline and kernel are subject to the following assumptions.
 
 \begin{assumption}\label{ass:baseline}
\leavevmode
 \begin{itemize}
    \item[(i)] For any $\theta \in \Theta$, $k \in \segN{1,K}$ and $t \in [0,1]$, $\mu_k(t,\theta)>0.$

    \item[(ii)] For any $\theta$ $\in \Theta$ and $k \in \segN{1,K}$,  $\mu_k(\cdot,\theta) \in \mathcal{C}^0([0,1])$.
\end{itemize}
 \end{assumption}

  \begin{assumption}\label{ass: kernel_separability}
\leavevmode
There exist two unique functions $f,g$ on $\R^+$ with values in $\mathcal{M}_K(\R^+)$ such that, for any $s \in \R^+$, $x \in \R^m$ and $\theta \in \Theta$,
    \begin{equation}\label{kernel_separation}
        \varphi (s,x, \theta) =   g(x,\theta) \odot f(s,\theta),
    \end{equation}
    where $\odot$ denotes the Hadamard (or coefficient-wise) product over $\mathcal{M}_K(\R^+)$.
\end{assumption}

\begin{remark}
    The uniqueness of decomposition~\eqref{kernel_separation} is not a benign assumption, it \textit{de facto} excludes non identifiable choices of the type $g (\gamma,x) \varphi(s,\alpha,\beta) = \alpha \gamma x \phi(s,\beta)$. Conversely,  the Hadamard product is only meant to reproduce commonly found parameterisations (see for instance Bacry \textit{et al.}~\cite{bacryBompaireGaiffas}). Other choices would not alter our results, including the usual matrix product.  
\end{remark}

 \begin{assumption}\label{ass:f_continuity}
 \leavevmode
\begin{itemize} 
    \item[(i)] For any $\theta \in \Theta$, $ f(\cdot,\theta) 
     \in \mathcal{C}^1(\R^+)$. 
     \item[(ii)] The function $f$ is uniformly equi-continuous in $\theta$. 
\end{itemize}
 \end{assumption}

  \begin{assumption}\label{ass:g_continuity}
   \leavevmode
\begin{itemize} 
    \item[(i)] For any $\theta \in \Theta$, $ g(\cdot,\theta) 
     \in L^2(F(\dif x))$.
     \item[(ii)] The function $g$ is uniformly equi-continuous in $\theta$. 
\end{itemize}
 \end{assumption}

  \begin{remark}
    Any kernel of the type

    \begin{equation}
        \varphi(s,x,\theta) = g(x) \odot f(s,\theta)
        \label{equ:applied_setting}
    \end{equation}

    where $g \in L^2(F(\dif x))$ and $f$ is a $\mathcal{C}^1$ function over $\R^+ \times \Theta$ will satisfy assumptions~\ref{ass:f_continuity} and~\ref{ass:g_continuity}. In particular, fixing $g \equiv 1$ retrieves Chen \& Hall's~\cite{FengInferenceForNonStationarySEPP} original setup, and our assumptions can be regarded as a natural extension of the univariate case's conditions in this respect. While theoretically restrictive, the splitting Assumption~\eqref{ass: kernel_separability} encompasses most of the existing literature regarding applications of marked Hawkes process; with use cases including seismology~\cite{OgataETAS}, information diffusion~\cite{HawkesTwitter}, epidemiology~\cite{chiang}, and price dynamics on financial markets~\cite{LeeVolatility}\cite{DeschatreGruet}~\cite{Chen23}. The choice of i.i.d marks is a more limiting Assumption. Embrecht \textit{et al.}'s model~\cite{embrechts_liniger_lin_2011}  features a split kernel but remains out of the scope of this article since its marks exhibit a dynamic distribution. 
\end{remark}

We also require some identifiability conditions on $\mu$ and $\varphi$.

\begin{assumption}
    \label{ass:identifiability}
    If $\vartheta \neq \theta$, then there exist a sub-interval $ J \subset [0,T]$ and a subset $S \subset \R^m$ of non null-measures on which, for every $(t,x) \in J \times S$, either $\mu(t,\vartheta)>\mu(t,\theta)$ or $\mu(t,\vartheta)<\mu(t,\theta)$ or $g(x,\vartheta)f(t,\vartheta)<g(x,\theta)f(t,\theta)$ or $g(x,\vartheta)f(t,\vartheta)>g(x,\theta)f(t,\theta)$.
\end{assumption}
 
As we move to the weak convergence of the \textsc{lrs}, we will be concerned with derivatives of the likelihood, requiring some higher order regularity conditions on the kernel $\varphi$.

 \begin{assumption}\label{ass:regularity_of_mu}
  \leavevmode
    \begin{itemize}
    \item[(i)] For any $t \in [0,T]$,  $\theta \mapsto \mu(t,\theta)$ is twice continuously differentiable over $\Theta$.
    \item[(ii)] For any $\theta \in \theta$, and any $p \in \{1,2\}$, $t \mapsto \partial_\theta^p \mu(t,\theta) \in  \mathcal{C}^1(\R^+)$.
\end{itemize}
 \end{assumption}

  \begin{assumption}\label{ass:regularity_of_kernel}
  \leavevmode
    \begin{itemize}
    \item[(i)] For any $t \in [0,T]$,  $\theta \mapsto f(t,\theta)$ is twice continuously differentiable over $\Theta$.
    \item[(ii)] For any $\theta \in \theta$, and any $p \in \{1,2\}$, $t \mapsto \partial_\theta^p f(t,\theta) \in  \mathcal{C}^1(\R^+)$.
\end{itemize}
 \end{assumption}

 \begin{remark}
     Take note that the differentiability of $f$ is not restricted to the interior of $\Theta$, and we therefore require its right-differentiability at boundaries where some $\alpha_{kl}$ are null. 
 \end{remark}

   \begin{assumption}\label{ass:regularity_of_weight}
  \leavevmode
    \begin{itemize}
    \item[(i)] for any $x \in \R^m$, $\theta \mapsto g(x,\theta)$ is continuously differentiable over $\Theta$.
    \item[(ii)] For any $q \in \{1,2 \}$, $\partial_\theta^q g$ is uniformly $p$-Hölder in $\theta$ for some $p \in (0,1]$:  there is some $C>0$ such that, for any $x \in \R^m$ and any $\theta, \nu \in \Theta$ 
    \begin{equation*}
        \lVert \partial_\theta^p g(x,\theta) - \partial_\theta^p g(x,\nu) \lVert_1
        \leq 
        C 
        \lVert \theta - \nu \lVert_2^p
    \end{equation*}
\end{itemize}
 \end{assumption}

\begin{remark}
    Hypothesis~\ref{ass:regularity_of_weight} drastically restricts the theoretical class of  functions for $g$, but it is inconsequential in practice. In other works, it is often the case that $g$ does not depend on $\theta$, and the conditions above are then immediately met (see Deschatre and Gruet~\cite{DeschatreGruet} and Goda~\cite{Goda} for two examples). When $g$ does depend on $\theta$, recalling the instantaneous probability interpretation of the intensity, a fitting choice is the logistic function
    \begin{equation*}
        g(x,\theta) 
        =
        \frac{1}{1 + e ^{ - \theta x}},
    \end{equation*}

    which derivative of order $p$ in $\theta$ is $\lvert x \lvert^p$-Lipschitz. Assumption~\ref{ass:regularity_of_weight} is thus easily satisfied in practice by truncating the marks, and the remark extends to most reasonable choices for $g$.
\end{remark}

The next hypothesis is standard in the study of linear Hawkes processes, and notably suffices to guarantee the non-explosion of $\boldsymbol{N}_t$ (see for instance Bacry \textit{et al.} \cite{bacrylimit}[Lemma 1]).  

 \begin{assumption}\label{ass: stability} 
 The  process is stable in that for any $\theta$ in $\Theta$\\
 \begin{equation*}
     \rho \Big( 
     \int_{(0,\infty] \times  \R^m} \varphi(s,x,\theta) \dif s F(\dif x)
     \Big)< 1.
 \end{equation*}
 where $\rho$ denotes the spectral radius.
 \end{assumption}

\begin{remark}\label{remark:introduce_the_branching_ratio}
    In Hawkes \& Oakes'~\cite{Oakes} cluster representation, the univariate Hawkes process is an immigration-birth process in which migrants arrive at rate $\mu$ and have descendants according to a Galton-Watson process with branching ratio $\int_0^\infty \varphi (s) \dif s$. Assumption~\ref{ass: stability} then has the intuitive interpretation that the ratio of non-immigrants in the accumulated population cannot exceed $1$. A popular choice for $\varphi$ is the exponential parameterisation $t \mapsto \alpha \exp(- \beta t)$. The condition then expresses straightforwardly as $\nicefrac{\alpha}{\beta}<1$, and is comparably simple to verify for most kernel choices.  
\end{remark}

In our context, Assumption~\ref{ass: stability} guarantees a semi-analytic form of the normalized mean intensity, as exposed in lemma~\ref{lemma:normalized_intensity}, which proof is straightforward and delayed to section \ref{section:preparation_for_proofs}

 \begin{lemma}\label{lemma:normalized_intensity}
Under Assumption~\ref{ass: stability}, for any $ \theta_0 \in \Theta$ and any $t \in [0,T]$ 
    \begin{equation*}
    \frac{1}{n} \mathbb{E}^{\theta_0} \big[ \lambda^{(n)}(t,\theta_0) \big]
    =
    h(t, \theta_0)
    :=
    \mu\Big(
        \frac{t}{T}
        ,
        \theta_0
    \Big)
    + 
    \int_0^t 
        \langle F, \Psi \rangle (t-s,\theta_0) \mu(s,\theta_0) 
    \dif s,
\end{equation*}

where $\mathbb{E}^{\theta}$ is the expectation under $\Prob(\theta)$ and $
    \langle F,  \Psi \rangle
    =
    \sum_{k \geq 1}
    \langle F , \varphi \rangle ^{\star k}
$,
with $\star k$ the $k$-fold convolution.
\end{lemma}

Throughout, we denote by $H(\cdot,\cdot)$ the function
\begin{equation}\label{equ:H_function}
    H \colon t,\theta \mapsto \int_0^t h(s, \theta) \dif s.
\end{equation} 

The definition of $h$ being independent from the value of $n$ allows to express the asymptotic information $I(\theta)$ of the model in a convenient form, as will be made clear in section \ref{section:large_n_main_results}. In the following, for any vector $u = (u_i)$ in some $\R^d$, we use the tensor involution notation $u^{\otimes 2}$ to denote the $ d \times d$ matrix $u^{\otimes 2} = ( u_i u_j) = u u^\textsc{t}$.

\begin{assumption}(Non degeneracy of the observed information)\label{ass: Fisher_information_is_positive}
    \begin{itemize}
     \item[] The matrix valued function
     \begin{equation}\label{equ:Fisher_information}
         I(\theta)
         =
         \int_0^T
         \sum_{k=1}^K
             \frac
            {
            \{
            \partial_\theta \mu_k(s,\theta) 
            +
            \int_0^s 
            \sum_{l=1}^K
            \partial_\theta 
                \langle F,\varphi\rangle_{kl}(s-u,\theta) h_l(u,\theta) \dif u
            \}^{\otimes 2}
            }
            {
            \mu(s,\theta) 
            +
            \int_0^s 
            \sum_{l=1}^K 
                \langle F,\varphi\rangle_{kl}(s-u,\theta) h_l(u,\theta) \dif u
            }
         \dif s 
     \end{equation}
      takes non singular values, where $h(t,\theta)=(h_k(t,\theta))$ is defined in Lemma~\ref{lemma:normalized_intensity}.
     \label{hyp:non_singular}
 \end{itemize}
\end{assumption}

\begin{remark}\label{remark:indentifiability}
    Verifying assumption~\ref{ass: Fisher_information_is_positive} relates to the parameter identifiability problem, which is a non trivial question for \textsc{mhp}s. consider for instance the exponential kernel
    \begin{equation*}
        \varphi_{kl} \colon t \mapsto \alpha_{kl} \exp( -\beta_{kl} t).
    \end{equation*}
    
    for $(k,l) \in \segN{1,K}^2$. It suffices that $\alpha_{kl}=0$ in the true parameter to prevent $\beta_{kl}$ from being identified. Under the non degeneracy assumption~\ref{ass: Fisher_information_is_positive}, the size and structure of the decay parameter are hence constrained by those of the weighted adjacency.  We refer to Bonnet \textit{et al.}~\cite{Martinez} for a simple identifiability criterion in the context of exponential \textsc{mhp}s.  In general, we provide a consistent estimator for $I(\theta)$ in Lemma~\ref{lemma:convergence_of_likelihood_derivatives} on which assumption~\ref{ass: Fisher_information_is_positive} can be numerically verified.
\end{remark}

\begin{comment}
    \begin{remark}
    While there is no simple general criterion for Assumption~\ref{ass: Fisher_information_is_positive}, our results provide a consistent estimator from which the non-singularity of $I(\theta)$ can be numerically verified.
\end{remark}
\end{comment}

 \subsection{Likelihood-based inference}

The log-likelihood of $\theta \in \Theta$ for the (marked) \textsc{mhp} is given by 
 \begin{equation}\label{equ:log_lkl_def}
    \LogLik{n,\theta} 
    =
    \sum_{k=1}^K 
    \int_0^T
        \ln \lambda^{(n)}_{k,s} (\theta)
    \dif N^{(n)}_{k,s}
    -
    \int_0^T
        \lambda^{(n)}_{k,s}(\theta)
    \dif s,
\end{equation}

where $( \boldsymbol{N}^{(n)}_t)$ is the process defined by~\eqref{equ:aggregation}. We refer to Daley \& Vere-Jones~\cite{DVJ}[Proposition 7.2.3 p. 232] for a proof of~\eqref{equ:log_lkl_def}. The maximum likelihood estimator $\hat{\theta}_{n}$ (\textsc{mle}) for $\theta$ is then defined as the global argmax of the likelihood $\mathcal{L}(n,\cdot)$ over $\Theta$. For any sub-model $\Theta_0$, we define similarly $\hat{\theta}_n^0$ as the global argmax of $\mathcal{L}(n,\cdot)$ over $\Theta_0$.  Note that $\hat{\theta}_n^0$ writes as a random vector of $\R^d$ with $p$ null coordinates. Our core problem is to derive the asymptotic distribution of the (log) likelihood ratio statistic (\textsc{lrs})
\begin{align*}
    \Lambda_n 
    =
    2 \Big(
    \sup_{\theta \in \Theta} \mathcal{L}(n, \theta)
    -
    \sup_{\theta \in \Theta_0} \mathcal{L}(n, \theta)
    \Big)
    =
    2 \big( \mathcal{L}(n,\hat{\theta}_n) - \mathcal{L}(n,\hat{\theta}^0_n) \big) 
\end{align*}

for sub-models of the type 
\begin{equation}\label{equ:sub_model}
    \Theta_0 = \big\{ \theta=(\gamma, \alpha, \beta)
    \in \Theta \lvert \hspace{0.2cm} 
    \alpha_{kl}=0; \hspace{0.2cm} (k,l) \in \mathcal{J} \big\} 
\end{equation}

corresponding to the set of hypotheses~\eqref{equ:what_we_testH0a} and~\eqref{equ:what_we_testH1a}.  As is usual for likelihood ratio tests the supremum over the whole parameter space is used as a proxy for the alternative. The distribution of the \textsc{mle} is of course also of great interest and we cover the problem as a by-product of our general results on $\Lambda_n$. Finally, we stress that the class of test we introduce is quite flexible, and $\Theta_0$ may be chosen to match a larger set of problems.

\section{Main results}\label{section:large_n_main_results}

\subsection{Consistency of the \textsc{mle}}
Before getting to the likelihood ratio, a preliminary consistency result for the \textsc{mle} is stated.

\begin{proposition}[]\label{prop:large_sample_consitency}
    Under Assumptions~\ref{ass:baseline} to~\ref{ass:identifiability} and Assumption~\ref{ass: stability} the \textsc{mle} $\hat{\theta}$ is consistent: for any $\theta_0 \in \Theta$,
    \begin{align}
        \hat{\theta}_n 
        & \xrightarrow[n \to \infty]{\Prob(\theta_0)}
        \theta_0.
        \label{equ:MLE_is_consistent}
    \end{align}
\end{proposition}

Observe that we have not introduced any restriction locating $\theta_0$ in the interior $\mathring{\Theta}$ of its parameter space $\Theta$. When the true parameter belongs to $\mathring{\Theta}$, the consistency of the \textsc{mle} guarantees it eventually positions itself away from the border with probability tending to $1$. We may then expect the same asymptotic behaviour as in the univariate case, namely, the asymptotic normality of the \textsc{mle}. In the general case where this restriction is lifted, the \textsc{mle} may find itself pinned along the boundary with strictly positive probability, giving rise to singularities in the distribution of the \textsc{lrs} $\Lambda_n$. Note in particular that in the univariate case, the main result of Chen \& Hall~\cite{FengInferenceForNonStationarySEPP} is thus non applicable to our purpose of testing for self-excitation as it breaks down under the Poisson null hypothesis $\varphi=0$. The general statement for the convergence of the \textsc{lrs} however remains similar to the standard case, as still expresses in terms of quadratic forms in a Gaussian variable, see Van den Vaart~\cite{VanDenVaartAsymptoticStatistics}[section 16.3].

\subsection{Asymptotic of the likelihood ratio test}Let $\Theta_0$ be the sub-model~\eqref{equ:sub_model} associated to a null hypothesis of the type~\eqref{equ:what_we_testH0}. Without loss of generality we may assume that the null adjacencies correspond to the first $p$ parameter coordinates $\theta_1, \cdots \theta_p$. Then, $\sqrt{n} ( \Theta_0 - \{ \theta_0 \})$ converges \footnote{In the sense of Painlevé-Kuratowski, see the introduction of Geyer~\cite{Geyer}.} to $H_0 = \{0 \}^p \times \R^{d-p}$ and $\sqrt{n} ( \Theta - \{ \theta_0 \})$ to $H=(0, \infty)^p \times \R^{d-p}$. The asymptotic distribution of the \textsc{lrs} $\Lambda_n$ is then given as the difference between the squared distances of a certain Gaussian variable to $H_0$ and $H$.

\begin{theorem}\label{thm:distribution_of_lambda_as_a_projection}
 Under Assumptions~\ref{ass:baseline} to~\ref{ass: Fisher_information_is_positive}, for any $\theta_0 \in \Theta_0$

 \begin{equation}\label{equ:chibar_as_a_projection}
     \Lambda_n
     \xrightarrow[n \to \infty]{\mathcal{L}(\Prob(\theta_0))}
      \lVert I(\theta_0)^{\nicefrac{1}{2}} X - I(\theta_0)^{\nicefrac{1}{2}} H_0 \lVert^2_2
      -
      \lVert I(\theta_0)^{\nicefrac{1}{2}} X - I(\theta_0)^{\nicefrac{1}{2}} H\lVert^2_2
 \end{equation}
    where $X \sim \mathcal{N}(0,I(\theta_0)^{-1})$
\end{theorem}

In the standard case, $H$ contains the whole parameter space and the second term vanishes. Here the term subsists as the singularities in the distribution of the \textsc{mle} echo into that of the \textsc{lrs}. This gives rise to a $\bar{\chi}^2$ distribution.

\begin{proposition}\label{proposition:the_chi_bar_distribution}
     Under Assumptions~\ref{ass:baseline} to~\ref{ass: Fisher_information_is_positive}, the likelihood ratio for the set of hypotheses~\eqref{equ:what_we_testH0a} and~\eqref{equ:what_we_testH1a} is asymptotically distributed under any $\Prob(\theta_0)$ with $\theta_0 \in \Theta_0$ as a chi-bar law:
     \begin{equation}
              \Lambda_n
     \xrightarrow[n \to \infty]{\mathcal{L}(\Prob(\theta_0))}
      \Bar{\chi}^2(p)
     \end{equation}
\end{proposition}

The chi-bar distribution can be conceived as the one of a $\chi^2$ law with random degrees of freedom depending on the number of zeros in the projection on $H$ of its underlying Gaussian variable $X$ in Theorem~\ref{thm:distribution_of_lambda_as_a_projection}. The $\Bar{\chi}^2$ thus expresses as a mixture of standard $\chi^2$
\begin{equation*}
    \Bar{\chi}^2(p) = \sum_{k=0}^p \omega_k \chi^2(k).
\end{equation*}

We refer to Annex~\ref{annex:chi_bar_def} for a rigorous derivation of the existence and expressions of the weights in our case, and to the works of Shapiro~\cite{Shapiro}, Kudo~\cite{Kudo}, Self \& Liang~\cite{Self-Liang}, and Susko~\cite{Susko} for further details on the links between the $\Bar{\chi}^2$ and the standard $\chi^2$ distributions. The weights $(\omega_k)$ have the unfortunate property of generally depending non-explicitly on the value of $\theta_0$. This problematic dependency can however easily be circumvented. In some specific albeit useful configurations -- including the test of an univariate Poisson against the self exciting alternative -- the weights are in fact known and independent from the true parameter. In the general case, Susko~\cite{Susko} suggests exploiting the relation between the effective degree of the chi-bar and its underlying normal variable $X$ by approximating the conditional distribution of $\Lambda_n$ by a $\chi^2(p-k)$ law when $k$ zeros are estimated in the \textsc{mle} $\hat{\theta}_n$. This yields the \textit{conditional} chi-bar test. The p-values for our test are therefore given by the following rule:

\begin{itemize}
    \item When the $(\omega_k)_{k \in \segN{1,p}}$ are known, by Proposition~\ref{proposition:the_chi_bar_distribution}, for any $\theta_0 \in \Theta_0$
    \begin{equation*}
        \Prob_{\theta_0}( \Lambda_n \leq z^\omega_a )
        \xrightarrow[n \to \infty]{}
        a,
    \end{equation*}
    where $z^\omega_a$ is the quantile at level $a$ for the $\sum_k \omega_k \chi^2(k)$ law .\\
    \item When the $(\omega_k)_{k \in \segN{1,p}}$ are unknown, by Theorem~\ref{thm:distribution_of_lambda_as_a_projection} and Susko's Theorem~\cite{Susko}, denoting $\hat{p}$ the estimated number of zeros in $\hat{\theta}_n$, for any $k \in \segN{1,p}$,

    \begin{equation*}
        \Prob_{\theta_0}( \Lambda_n \leq z^{p-k}_a 
        \hspace{0.1cm}
        \lvert
        \hspace{0.1cm}
        \hat{p}=k )
        \xrightarrow[n \to \infty]{}
        a,
    \end{equation*}
    where $z_a^{p-k}$ is the quantile at level $a$ for the $\chi^2(p-k)$ law .
\end{itemize}

As the test statistic goes to $\infty$ as $n \to \infty$ under any $\Prob(\theta)$ with $\theta$ in the alternative, the test is consistent. As a by-product of Theorem~\ref{thm:distribution_of_lambda_as_a_projection}, we recover the asymptotic distribution of the \textsc{mle}.

\begin{proposition}[Central limit theorem]\label{thmTCL}
    Under Assumptions~\ref{ass:baseline} to~\ref{ass: Fisher_information_is_positive}, for any $\theta_0 \in \Theta_0$, the distribution of $\sqrt{n}(\hat{\theta}_n - \theta_0)$ converges under $\Prob(\theta_0)$ to the distribution of  

    \begin{equation*}
        \argmin_{Y \in H}
        \lVert I(\theta_0)^{\nicefrac{1}{2}} X - I(\theta_0)^{\nicefrac{1}{2}} Y \lVert^2_2 
    \end{equation*}

    where $X \sim \mathcal{N}(0, I(\theta_0)^{-1})$.  Furthermore, denoting by $\Tilde{\theta}^0_n$ the non-null coordinates of the sparse \textsc{mle} $\hat{\theta}_n^0$ and by $\Tilde{\theta}_0$ the non-null coordinates of the true parameter, the distribution of $\sqrt{n}(\Tilde{\theta}^0_n - \Tilde{\theta}_0)$ converges to a centered normal distribution with variance given by the inverse of the principal submatrix of $I(\theta_0)$ corresponding to the null hypothesis parameters. 
\end{proposition}

\section{Numerical study}\label{section:numerical}

In this section, we illustrate the preceding theoretical results with examples from simulated and real data. While the present work mainly deals with model selection for \textsc{mhp}s, allowing for a null adjacency opens interesting perspective for the univariate Hawkes process. Our first example is concerned with testing for self-excitation in such setting. This allows for an application on Ebay auctions, on which we test for potential deviations from a naive equilibrium. The second  example is concerned with the bivariate Hawkes, which is of frequent use in financial applications. We show how our test may be used to differentiate between different price feedback regimes on intraday power markets. In these first two examples, the asymptotic distribution of the \textsc{lrs} is explicit and independent from the value of the true parameter. Finally, in our third and last example, we show how the test may be conducted in the general case, and demonstrate how Susko's~\cite{Susko} methodology applies in our setting.

\subsection{The Poisson \textit{versus} Self-exciting test }\label{sec:testting_for_excitation}
Self-excitation is the prominent feature of the univariate Hawkes process.   A natural problem is to determine whether the observed event sequence does feature such a property at all. Recall that in spite of the univariate setting of this test, the results of Chen \& Hall are non applicable under a Poisson null and we must have recourse to Theorem~\ref{thm:distribution_of_lambda_as_a_projection} instead.  The Poisson \textit{versus} Hawkes problem was first raised by Dachian \& Kutoyants~\cite{Kuto}, whom propose a  test in the long time asymptotic for the Poisson null hypothesis against the one sided alternative. It has since arisen again in the context of applications. In Kramer \& Kiesel~\cite{KRAMER2021105186} for instance,  the authors identify the $\Bar{\chi}^2$ family of distribution as a the correct asymptotic for the likelihood ratio, but they are unable to conclude.

\subsubsection{Setting and model}

We first work in the setting of Dachian \& Kutoyants~\cite{Kuto}, where the kernel $\varphi \colon \R^+ \mapsto \R^+$ of the Hawkes process depends on a single parameter $\alpha \in [0,1)$
\begin{equation}\label{equ:univariate_exponential_model}
     \varphi \colon t, \alpha   \mapsto
     \alpha \Phi(t),
 \end{equation}

 and is otherwise completely specified by a known positive function $\Phi: \R^+ \mapsto \R^+$. We also let the baseline fluctuate in time according to $
     \mu \colon t, \mu_0, \kappa \mapsto 
     \mu_0 \exp ( \kappa \frac{t}{T})$. Our question amounts to testing \begin{equation}\label{equ:Poisson_null}
         \alpha= 0 \hspace{0.5cm} \text{(the process is a Poisson process),}
     \end{equation} 
     
     against the one sided alternative \begin{equation}\label{equ:Hawkes_alternative}
         \alpha > 0 \hspace{0.5cm} \text{(the process is a true Hawkes process).}
     \end{equation}
     
     The asymptotic distribution under the null of $\Lambda_n$~\eqref{equ:Poisson_null}  is given by Theorem~\ref{thm:distribution_of_lambda_as_a_projection}    as the one of $\mathbb{1}_{X>0}X^2$ where $X \sim \mathcal{N}(0,1)$ (see annex~\ref{annex:chi_bar_def}) and simplifies as \begin{equation}\label{equ:dim_one_mixture}
         \Bar{\chi}^2 = \frac{1}{2} \chi^2_0 + \frac{1}{2} \chi^2_1.
     \end{equation}
     The test provides a staighforward decision rule as it rejects~\eqref{equ:Poisson_null} at the significance level  $a \in (0,\nicefrac{1}{2})$ when $\Lambda_n$ is above the $(1-2a)$--quantile of the standard $\chi^2(1)$ distribution. The framework of Dachian \& Kutoyants~\cite{Kuto} is however too restrictive in practice. The function $\Phi$ usually depends on some additional decay parameter $\beta$, which is not known. We then face the issue that two different values for $\beta$ yield the same distribution for the process under the null~\eqref{equ:Poisson_null}: a non identifiable nuisance parameter must  be dealt with in addition to the boundary problem. The exact same configuration can be encountered in Ritz~\cite{Ritz}[Remark 4] for instance, where the author remarks that the distribution of the \textsc{lrs} then differs from~\eqref{equ:dim_one_mixture}.\\
     
     A naive answer to the issue consists in falling back on the simpler case~\eqref{equ:univariate_exponential_model} by setting $\beta$ at some pre-determined value $\beta^*$ to conduct the test. The distribution of $\Lambda_n$ under the null~\eqref{equ:Poisson_null} is still given by~\eqref{equ:dim_one_mixture} and the level of the test is therefore unchanged. The resulting mis-specification of $\beta$ will however decrease the value of the \textsc{lrs} under the alternative, as underlined by Hansen~\cite{HansenNuisance}. In our case, we find numerical evidence that the subsequent power loss is very modest even for gross mis-specifications of $\beta$. In light of these results, we leave potential improvements on the naive test to further study (see for instance Andrews~\cite{AndrewsNuisance} for sup-likelihood ratio tests). Note also that the question of the test's robustness to errors in $\beta$ goes beyond identifiability questions as the estimation of the decay parameter is subject to notoriously pervasive numerical instability issues.

\subsubsection{Simulation study.} We  report the results for $10000$ simulations over $T=10$s of the point process $N^{(n)}$ with intensity $n \mu$, corresponding to the sum of $n=2000$ non-homogeneous Poisson process with intensity $\mu$. The parameters are set to $\mu_0=2$ and $\kappa=2$, and the simulation performed using a standard thinning approach. At each step we compute the value of the \textsc{lrs} $\Lambda_n$ for the full model against the Poissonian one by numerically maximising the likelihood over the full ($\alpha >0$) and sparse ($\alpha=0$) parameter set corresponding to each hypothesis. Three sets of such simulations are performed, with the kernel of the process in the full model successively taken as a Pareto kernel, an exponential kernel, and a gamma kernel, all with decay $\beta=10$.
\begin{equation}\label{equ:gamma_kernel}
    \Phi \colon t,\beta \mapsto \nicefrac{1}{\beta}(1+t)^{-(1+\beta)}
    \hspace{0.5cm} 
    \Phi \colon t,\beta \mapsto \beta \exp(-\beta t ) 
    \hspace{0.5cm} \text{and} \hspace{0.5cm}
    \Phi \colon t, \beta  \mapsto \beta^2 t \exp( - \beta t).
\end{equation}

 We provide in figure~\ref{fig:qqplotPoissonvHawkes} the \textsc{qq}-plot for the empirical distribution of $\Lambda_n$ conditionally on its being positive, which is expected to follow a $\chi^2(1)$ distribution. In practice, we use a threshold of $\varepsilon=10^{-4}$ above which the \textsc{lrs} is considered positive.\\

\begin{figure}[ht]
\includegraphics[width=.283\textwidth]{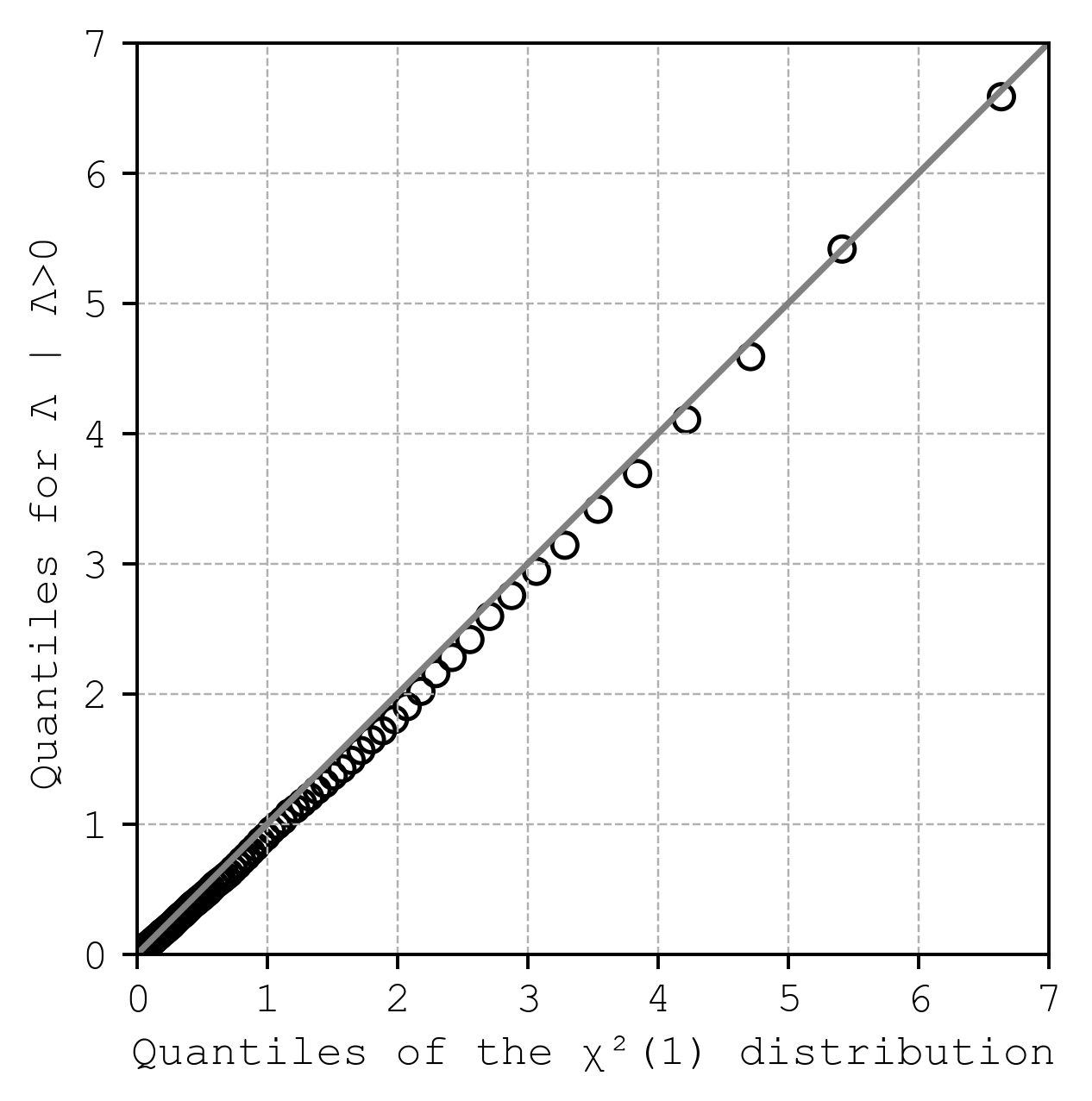}
\includegraphics[width=.27\textwidth]{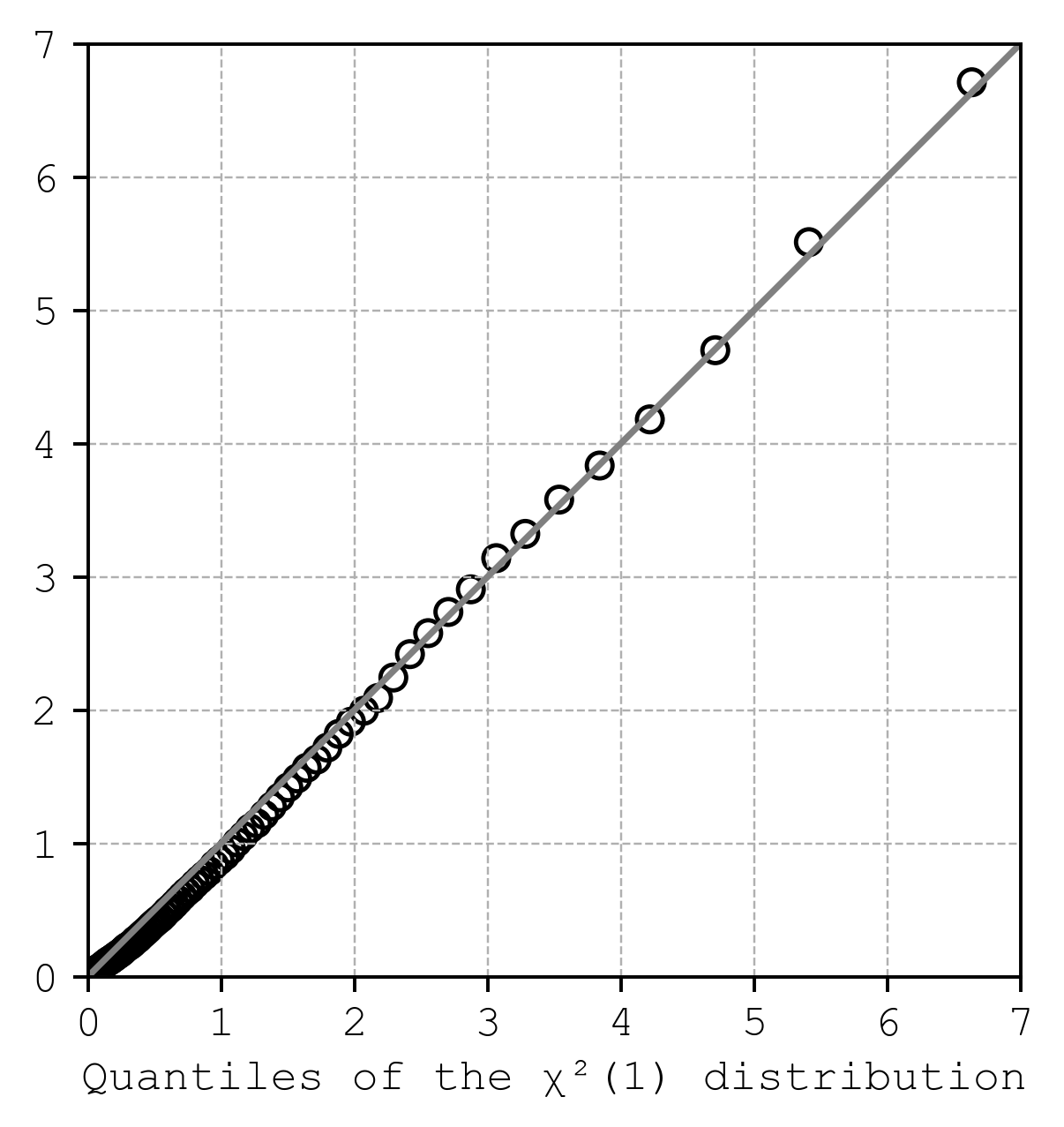}
\includegraphics[width=.27\textwidth]{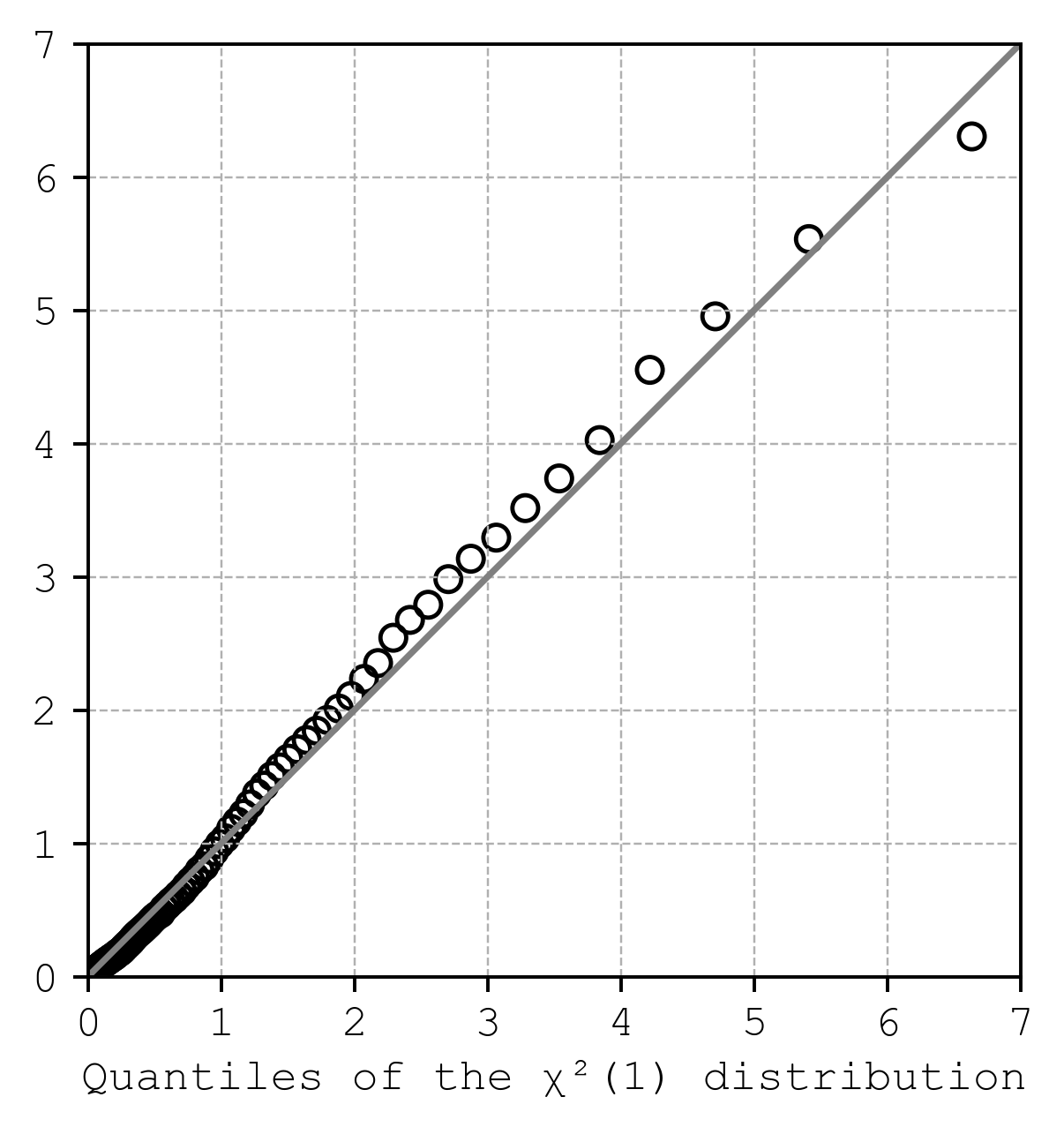}
\caption{\textsc{qq}-plots for $\Lambda_n$ conditional on positivity against its theoretical asymptotic distribution for $n=2000$ trajectories. From left to right: Power-law, exponential and gamma kernels with decay $\beta=10$.}
\label{fig:qqplotPoissonvHawkes}
\end{figure}

Note that under our chosen parameter values, the simulated process averages about $20$ jumps per realisation, which would prevent using standard estimation procedure based on long-time realisations of the process. The likelihood ratio behaves as expected under the null and we turn our attention to the distribution of the \textsc{lrs} under the alternative. Since this supposes the simulation of a large number of Hawkes processes, we limit our numerical study to the case of an exponential Hawkes process to keep the computation time under a reasonable length.\\

A collection of $N=2000$ sets of $n$ trajectories of a Hawkes process with constant baseline $\mu=4$, ecay $\beta=3$, and kernel $\varphi \colon t \mapsto \alpha \beta  \exp( - \beta t)$ is simulated for $12$ different values of the weighted adjacency (or endogeneity) parameter $\alpha$, and $n$ successively set at $n=500$ and $n=2000$, totalling $60$ million simulated trajectories. We compute in figure~\ref{fig:power} the empirical power of the test at the $5\%$ level for two specifications of the decay, one at the true value $\beta_0=3$, the other at $\beta =30$ corresponding to a $900\%$ misspecification.\\

\begin{figure}[ht!]
\includegraphics[width=.35\textwidth]{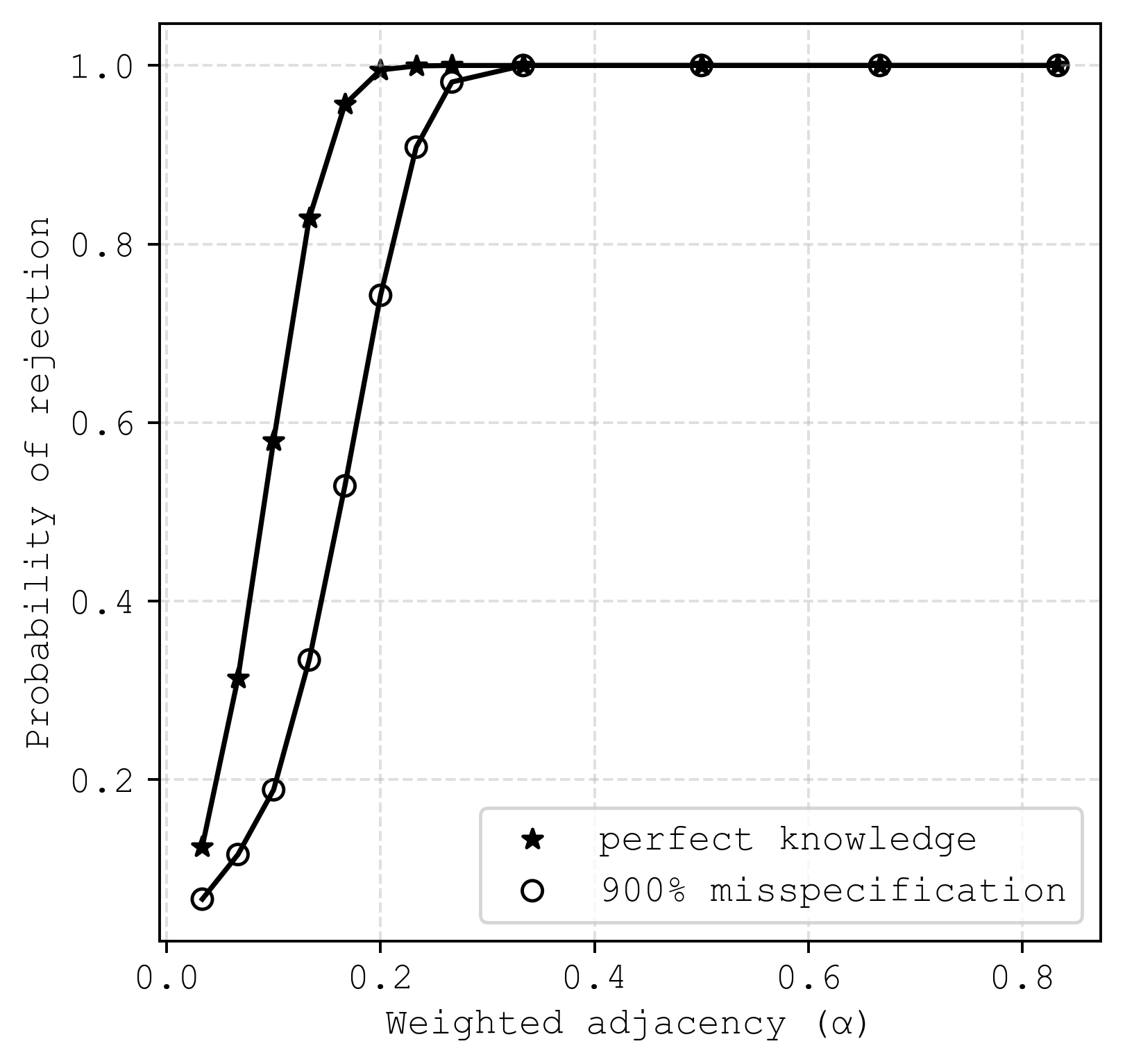}
\includegraphics[width=.35\textwidth]{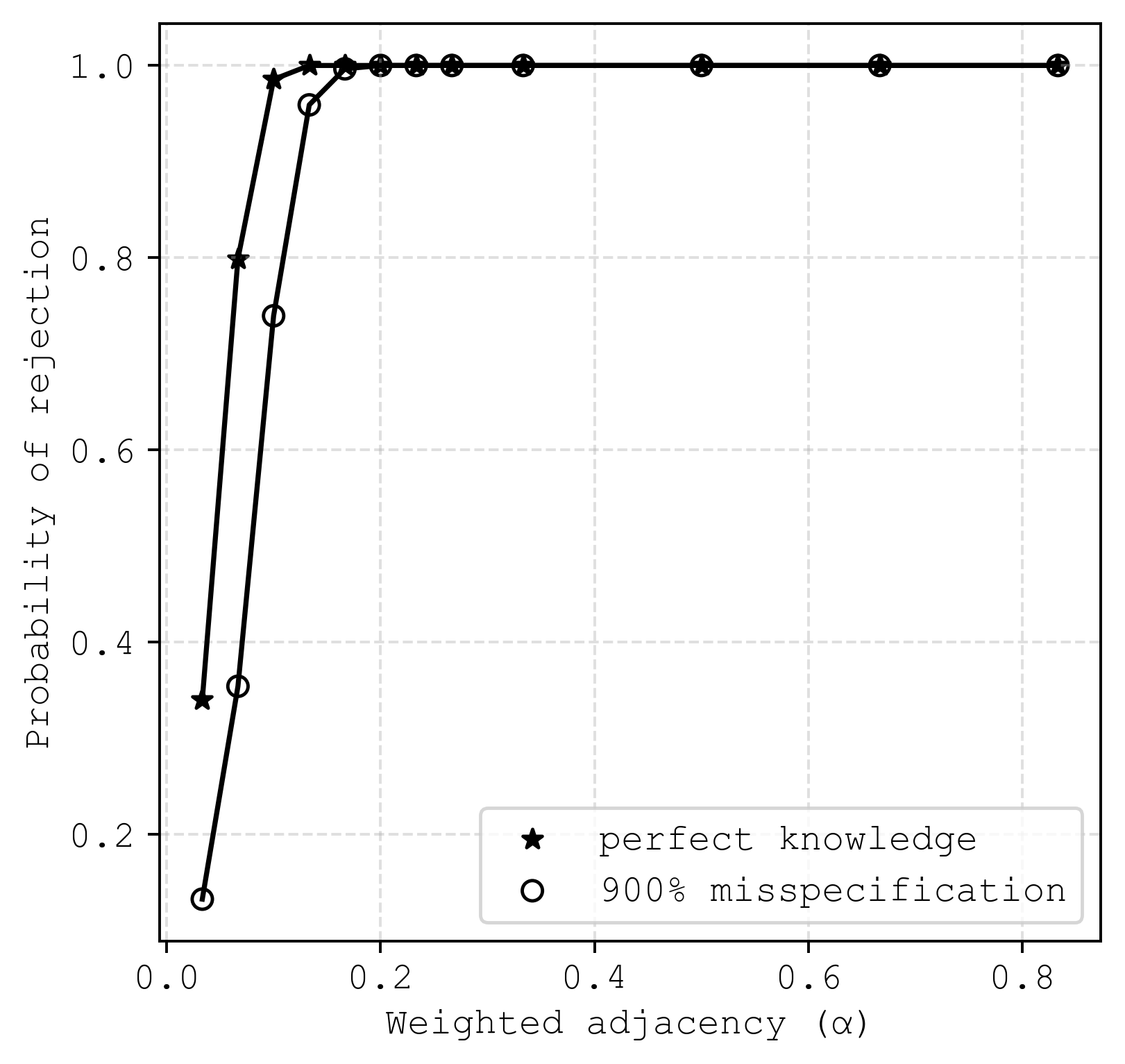}
\caption{Empirical power of the test for $N=2000$ sets of $n$ trajectories. Left: $n=500$ trajectories. Right $n=2000$ trajectories.  }
\label{fig:power}
\end{figure}

The numerical results suggest that a tenfold error in the choice of the decay incurs only a limited power loss, with most of the shift in the power function being contained around small local alternatives. In particular, the test remains consistent. 

\subsubsection{Strategic bidding in online auctions} We turn our attention to the real world setting of bid arrivals on online auction. The Ebay.com website runs second best price open auctions, meaning that the bids are public and the transaction settled at the second highest recorded bid. Were the bids sealed and static instead of public and dynamic, this would give rise to a Vickrey auction (see Krishna~\cite{krishna}[p. 181]), where a Nash equilibrium is reached with participants bidding their full capacity. If Ebay bidders are insensitive to each other's behaviour, the observed equilibrium should resemble the one of the Vickrey auction. One  would then expect bid arrivals to occur independently consistently with an homogeneous Poisson model.  This is of course not what is observed in practice. Shmueli \textit{et al.}~\cite{shmueli} remark bid arrivals depend on complex factors, and use a three-stages non-homogeneous Poisson process to account for the changing nature of the process. Participants engage for instance in last minute bidding, a behaviour known as \textit{sniping}. We propose an explanation for the deviation from the static equilibrium based on the categorisation of bidders into two cohorts:\\ 

\begin{itemize}
    \item \textit{Price sensitive} bidders, whose estimation is dynamic and depends on other bidder's behaviour.\\  
    \item \textit{Price insensitive} bidders, who place a single bid at full capacity, either at an entirely random time, or at the last few moments in an effort to squeeze out price-sensitive bidders. \\
\end{itemize}

We aim to prove the existence of the first category of agents by fitting a Hawkes model on bid arrival times. Ebay auctions participation is mostly human and we use a gamma kernel~\eqref{equ:gamma_kernel} to reflect the additional reaction time this incurs. We also retain the exponential baseline of model~\eqref{equ:univariate_exponential_model} in an effort to match the increasing intensity pattern.  If bidders were to all belong to the first category, the increasing Poisson-like exponential intensity should account entirely for the sniping phenomenon. If it results instead from agents competitively interacting, the Hawkes model is more appropriate. Our null and alternative hypotheses are therefore the same as for the preceding simulation study, that is a Poisson null $\alpha=0$ against the one-sided positive alternative. Should the test reject the null hypothesis, Ebay auctions would fall in the framework of interdependent values. The dynamic nature of equilibrium bidding strategies is well documented in such models (Krishna~\cite{krishna}[Proposition 6.2 p. 91]), consistently with our proposed interpretation. \\

The data consists in bid arrival times of $628$ Ebay auctions opened and settled in $2020$ for luxury watches, digital assistance devices and game consoles. Auctions may run over $7,5$ or $3$ days from which we keep only the last $2.8$ days. This allows for a simple synchronous model, and avoids the non stationarity due to novelty-induced bids at the start of $3$-days auctions, which are of little relevance to our question.\\

We set the parameter space for the decay as the closed real interval from $1$ to $10000$ days$^{-1}$ and repeat the test with the two values $(\beta_1,\beta_2)=(5,1000)$ to reduce the potential misspecification. Using the Bonferroni correction
\begin{equation}\label{equ:bonferroni}
\Prob \Big[ 
\bigcup_{i=1}^m \big\{ \Lambda_n(\beta_i) \geq q_{1-\frac{a}{m}}\big\}
\Big]
\leq 
\sum_{i=1}^m
\Prob \big[ \Lambda_n(\beta_i) \geq q_{1-\frac{a}{m}} \big]
\leq m \frac{a}{m}= a,
\end{equation}

we obtain a level $a$ test from the repeated procedure by rejecting the null as soon as one of the \textsc{lrs} is above the quantile of order $1 - \nicefrac{a}{m}$ of the $\Bar{\chi}^2_1$ distribution. When the value set for $\beta$ is large, the repetition allows to recoup some of the lost power as long as $m$ is kept reasonably small (here $m$=2). We report in table~\ref{tab:ebay_auctions} the \textsc{lrs} and p-value for each $\beta_i$, with $\beta_0=0$ representing the Poisson null.\\

\begin{table}[h!]
  \begin{center}
    \begin{tabular}{  l  c c c c c  }
    \hline
    $\beta$ & $\hat{m}$ &  $\hat{\kappa}$ & $\hat{\alpha}$ & $\textsc{lrs}$ & p \\
    \hline
    $0$ & $4.95$ & $3.43$ &  \textsc{n.a} & \textsc{n.a} & \textsc{n.a}  \\
    $5$ &  $0.24$ & $3.7$ & $0.96$ &  $2.26\text{e}2$& $7\text {e}-51$ \\
    $1000$ & $0.31$ & $1.50$ & $0.89$ & $4.88\text{e}3$ & $\sim 0$  \\
    \hline
    
    \end{tabular}
    \caption{Estimated parameters and likelihood ratio statistic against the Poisson null of a gamma Hawkes process on bid arrival times of $627$ ebay auctions.}
    \label{tab:ebay_auctions}
  \end{center}
\end{table}

\begin{comment}
            x: [ 3.062e-01  1.496e+00  8.986e-01]

                 x: [ 2.385e-01  3.703e+00  9.590e-01]
\end{comment}

Both decay choices yield p-values far below $0.005$, the  Poisson model is rejected at the level $0.01$,  matching our supposition of the price being formed through competitive price discovery. Though the second-price mechanism of Ebay auctions does not reward late bidding, their dynamic nature favorises interdependent strategies,  in which bidders await and react to the information flow from competing bids.    Providing a complete model for the dynamic of the auction would require a more flexible baseline, mimicking the varying intensity of Shmueli \textit{et al.}~\cite{shmueli}. Such ambition is out of the scope of this exploratory example which is intended to showcase the general relevance of our inference procedure in the context of dynamic auctions. In doing so, we have uncovered a feedback mechanism between bidders' valuations and bid values. Finally, we acknowledge our conclusions are perhaps dataset-specific. The auctions set we estimated the model upon suffers from a high prevalence of luxury goods, the demand for which features a characteristic positive price elasticity. Accessing additional data would here allow for more general conclusions. 

\subsection{Testing for cross-excitation in the bivariate Hawkes process}

Suppose one wishes to test for the joint significance of two (weighted) adjacency coefficients $\alpha_{ij}$ and $\alpha_{kl}$  belonging to different coordinates $k$ and $l$ of a multivariate Hawkes process.  From the expression of the Fisher information~\eqref{equ:Fisher_information}, that $k \neq l$ entails the asymptotic independence of the $\hat{\alpha}_{ij}$ and  $\hat{\alpha}_{kl}$. One may then derive the limit distribution of the likelihood ratio as
\begin{equation}\label{equ:chi_bar_id}
    \Bar{\chi}^2 
    =
    \frac{1}{4} \chi^2_0
    + \frac{1}{2} \chi^2_1
    +\frac{1}{4} \chi^2_2
\end{equation}

by remarking the weights of the mixture express either in terms of the probability of two independent Gaussian variables being positive or of a single Gaussian being positive.  Alternatively, setting the antidiagonal term to $0$ in example $7$ of Self \& Liang \cite{Self-Liang} yields the same result. While somewhat specific, the present example has a practical use in the context of financial statistics, and we shape our simulations and application around the microstructure model of Bacry \textit{et al}.

The price of an asset $(P_t)$ over a timeframe $t \in [0,T]$ is modelled as an accumulation of positive and negative jumps across the discrete price grid, respectively corresponding to pure jump processes $(N^+_t)$ and $(N^-_t)$ with null initial conditions:
\begin{equation*}
    P_t - P_0 =  N^+_t - N^-_t.
\end{equation*}

For liquid enough assets, the processes are simple as price jumps rarely exceed the mesh size of the grid (know as the \textit{tick} size). In the converse, the $(N^{\pm}_t)$ are marked with price increment sizes. The two jump processes are chosen as the coordinates of a bivariate Hawkes process with baseline $\mu$ and kernel $\varphi$. This model, originally intended for high frequency prices at the level of the tick grid, has also become an instrument of choice for power prices, which retain at large time scales some features reminiscent of microstructure phenomena -- see Deschatre \& Gruet~\cite{DeschatreGruet} for a price model an see Kramer \&  Kiesel~\cite{KRAMER2021105186} for order book dynamics.  From now on we suppose the baseline and kernel take the shapes
\begin{equation}\label{equ:exmaple2_model}
\mu = \exp( \kappa t) \begin{bmatrix}
    m & m
\end{bmatrix}^\textsc{t} 
\hspace{0.5cm}
    \varphi \colon t,x \mapsto 
    \beta x
    \begin{bmatrix}
         \gamma_1 & \alpha \\
        \alpha& \gamma_2
    \end{bmatrix}
    \exp( - \beta t).
\end{equation}

In the classical setting of Bacry \textit{et al.}, the price trajectories are guided purely by reversion effects, corresponding to the $\gamma_1=\gamma_2=0$ regime. One may wonder whether this kernel shape, which was introduced to account for the negative auto-correlation of price increments at high frequency remains reasonable at larger scale. This corresponds to testing the following hypotheses:
\begin{align}\label{equ:hyp_for_testing_bracy_et_al}
        \mathcal{H}_0 &: \gamma_1  =0, \:  \gamma_2 = 0   &&\text{  (Prices are purely reverting),} \nonumber \\
    \mathcal{H}_1 &: \gamma_1>0, \: \gamma_2 > 0 &&\text{ (Prices have a momentum component),}
\end{align} 
where the term \textit{momentum} is borrowed from the financial vernacular to designate a self-sustained directional dynamic in prices.

\subsubsection{Simulation study.} The simulated process has the specification~\eqref{equ:exmaple2_model} with $m=5$, $\alpha=5$, $\kappa=2$Hz, $\beta=10$Hz and lives under the null $\gamma_1=\gamma_2=0$. The distribution of the marks is the one of $\lvert X \rvert + \delta$ where $X \sim \mathcal{N}(0,1)$ and $\delta=10^{-2}$ is an offset parameter meant to prevent numerical error in the simulation. We report the results of repeated estimations for a set of $10000$ simulations over a timeframe of $T=1s$ and $n=800$. Figure~\ref{fig:qqplot} displays the \textsc{qq}-plot for the values obtained for $\Lambda_n$, conditionally on its being strictly positive, with a non-nullity threshold of $\varepsilon=10^{-5}$.\\

\begin{figure}[h!]
\centering
\includegraphics[width=.3\textwidth]{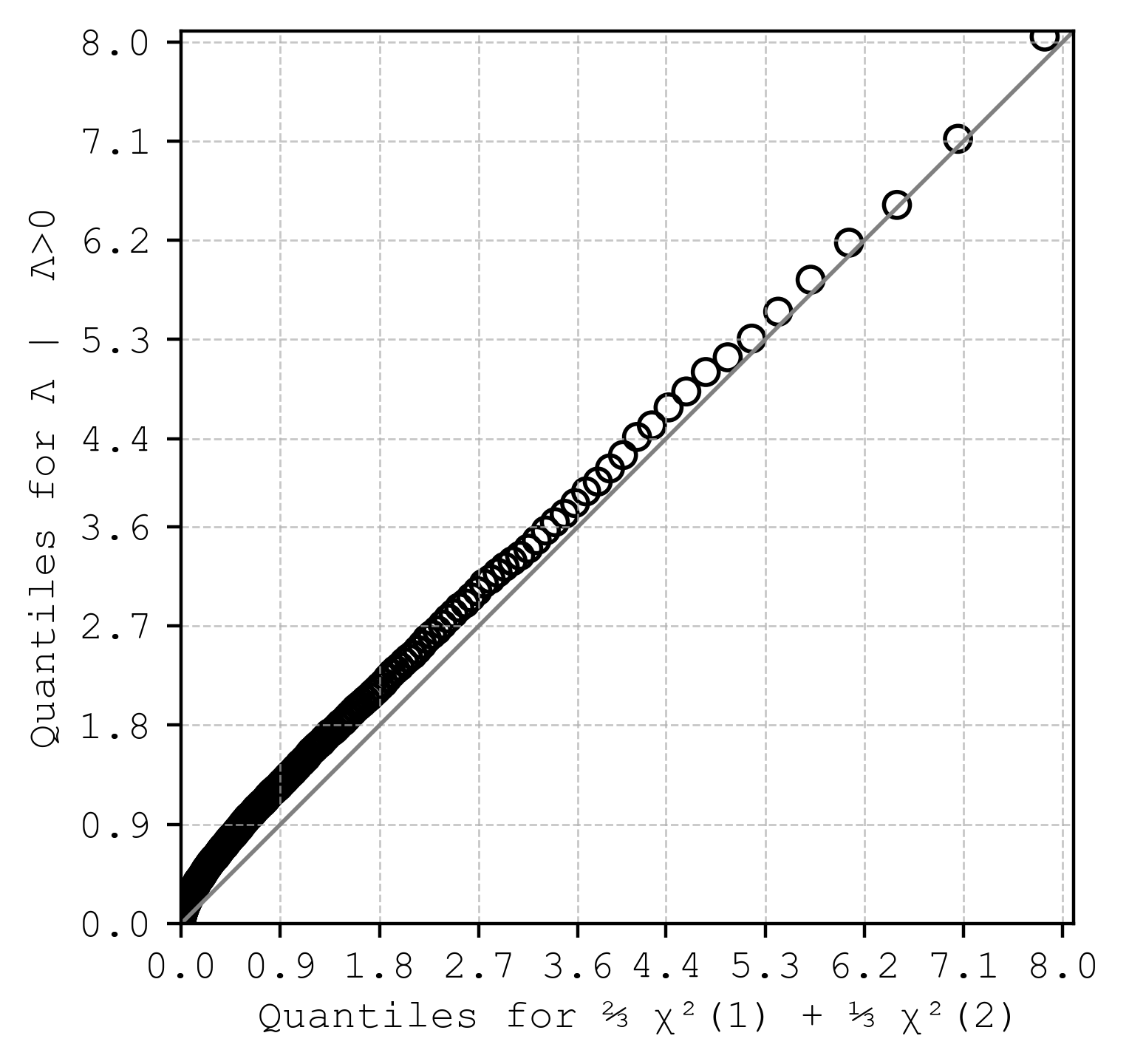}
\caption{\textsc{qq}-plots for $\Lambda_n$ conditionally on its positivity for $N=10000$ simulations of  $n=800$ trajectories of a marked \textsc{mhp} with characteristics~\eqref{equ:exmaple2_model} and absolute normal marks. }
\label{fig:qqplot}
\end{figure}

Conditional on positivity, a chi-bar variable with weights~\eqref{equ:chi_bar_id} has a $\nicefrac{2}{3} \chi^2(1) + \nicefrac{1}{3} \chi^2(2)$ law, and one observes $\max(\Lambda_n,\varepsilon) \simeq \max(\Lambda_n,0)$ to converge well enough to such a distribution .

\subsubsection{Single asset price dynamics of late night delivery power futures}\label{example:example2} We proceed to our application to intraday power prices. The dataset consists of $4$-hours long windows from $5$ to $1$ hour to delivery of hourly power future mid-prices\footnote{Arithmetic average of best bid and best ask prices.} for the $n=364$ trading sessions occuring from deliveries from January $1$, $2019$ to December $30$th, $2019$. The exclusion of the last hour of trading serves so as to avoid taking into account effects associated with the delivery zone reducing in size after that point.  In Deschatre \& Gruet~\cite{DeschatreGruet}, the authors focus on the most liquid of the hourly products of the German intraday power market. Thanks to the aggregation procedure~\eqref{equ:aggregation}, we are also able to estimate the Hawkes-based power prices model on late night and early morning deliveries future contracts, which are less intensively traded. The retained futures cover the $12$ consecutive hourly delivery periods beginning from $18$h to $5$h.
\vspace{0.2cm}
 
\begin{table}[h!]
\centering
\begin{tabular}{  c  c c c c c c c c c c c c   }
\hline
Delivery start & 18h & 19h & 20h & 21h & 22h & 23h & 0h & 1h & 2h & 3h & 4h & 5h   \\
\hline
$\alpha^\sharp$   & $2.30$ & $2.54$ & $2.11$& $2.21$ & $2.40$ & $2.11$& $2.30$ & $3.04$ & $2.90$ & $2.44$ & $2.20$ & $2.35$   \\
$\gamma_1$  & $2.03$ & $2.26$ & $2.41$& $2.25$ & $2.77$ & $1.37$ & $2.03$ & $2.41$ & $2.26$ & $2.71$ & $2.33$ & $2.35$   \\
$\gamma_2$& $1.62$ & $2.49$ & $2.30$ & $2.01$ & $2.36$ & $1.34$ & $1.62$ & $2.27$ & $2.20$ & $2.35$ &  $1.69$ & $1.52$   \\
$\beta^\sharp$  & $930$ & $516$ & $559$  & $510$ & $403$ & $369$ & $930$ & $460$& $350$ & $556$ & $488$ & $626$   \\
$\mu_0^\sharp$& $68.2$ & $95.0$ & $104$  & $99.3$ & $98.1$ & $67.6$ & $68.2$ & $60.9$ & $66.6$ & $51.1$ & $46.7$ & $93.0$   \\
$\kappa^\sharp$ & $1.14$ & $0.74$ & $0.76$ & $0.74$ & $0.60$ & $0.54$ & $1.14$ & $0.89$ & $0.91$ & $1.11$ & $1.21$ & $0.56$   \\
\hline
$\alpha^\flat$  & $3.86$ & $4.40$ & $4.44$ & $4.19$ & $4.89$& $3.27$ & $3.86$ & $5.38$ & $4.99$ & $4.70$ & $4.02$ & $3.99$   \\
$\beta^\flat$  & $692$ & $385$ & $371$  & $467$  & $302$ & $327$ & $692$ & $296$ & $253$ & $444$ & $470$ & $515$   \\
$\mu_0^\flat$& $70.2$ & $61.9$ & $105$  & $101$ & $98.9$ & $67.6$ & $70.2$ & $47.6$ & $62.3$ & $69.3$ & $77.1$ & $96.8$   \\
$\kappa^\flat$ & $1.16$ & $1.22$& $0.77$ & $0.75$ & $0.61$ & $0.58$ & $1.16$ & $0.88$ & $0.55$ & $0.93$ & $0.89$ & $0.58$   \\
\hline
$\mathbb{E}[X]$ & $0.10$ & $0.10$ & $0.10$ & $0.1°$ & $0.11$ & $0.11$ & $0.10$ & $0.09$ & $0.10$ & $0.1$ & $0.11$ & $0.11$   \\
$\eta^\sharp$ & $0.41$ & $0.47$ & $0.44$ & $0.46$ & $0.55$ & $0.46$ & $0.41$ & $0.54$ & $0.51$ & $0.50$ & $0.42$ & $0.34$   \\
$\eta^\flat$ & $0.39$ & $0.48$ & $0.45$ & $0.41$ & $0.54$ & $0.48$ & $0.39$ & $0.54$ & $0.50$ & $0.47$ & $0.43$ & $0.43$   \\
\textsc{lrs} & $647$ & $536$ & $371$ & $510$ & $644$ & $280$ & $647$ & $344$ & $249$ & $718$ & $954$ & $1379$   \\
p & $\sim 0$ & $\sim0$ & $\sim 0$ & $\sim 0$ & $\sim0$ & $\sim 0$ & $\sim 0$ & $\sim 0$ & $\sim 0$ & $\sim 0$ & $\sim 0$ & $\sim 0$   \\
\hline
\end{tabular}
\caption{Estimates for model~\eqref{equ:exmaple2_model} as applied to German intraday hourly power futures mid-prices, $5$ to $1$h to delivery on the first $364$ trading sessions of $2019$. The $^\sharp$ and $^\flat$ 
 superscripts respectively denote the full and sparse models.}
\end{table}

The null hypothesis is rejected with high certainty for each future contract, indicating the existence of a statistically significant momentum component in the dynamics of German intraday power prices. We hypothesise this stems from market participants competing to capture scarce liquidity. The cascading patterns observed in mid-prices hint at liquidity reconstituting slower than it is consumed, which would be consistent with our interpretation. Using a stochastic baseline Hawkes process, Kramer and Kiesel~\cite{KRAMER2021105186} find an influence of wind and solar forecast errors on exogenous order arrivals. Marking the kernel for similar exogenous factors could allow for testing potential relations between common stresses on power traders and episodes of directional variations in price.\\

Remarkably, our estimations situate the price endogeneity $\eta$ around the low value of $0.45$ which is consistent with the findings of Deschatre \& Gruet~\cite{DeschatreGruet}, meaning that price moves do not generate as many reactions as one would expect. This would put the intraday power markets at odds with more conventional asset classes, were values close to criticality $\eta=1$ are repeatedly observed (see for instance Hardiman \textit{et al.}~\cite{Hardiman}~\cite{HardimanBouchaud}), reflecting a no arbitrage condition (see Jusselin \& Rosenbaum~\cite{jusselin2018noarbitrage}). One  could assess the robustness of this conclusion and its potential reasons by repeating the preceding analysis with different choices for the marks, as precedently suggested. The low endogeneity would be compatible with the conjecture of liquidity constraints being an important price driver, since it suggests some sense of "missing participants". The particularities of power trading (which is typically a  capital-intensive activity and incurs specific risks) may indeed be repulsive to some of the conventional liquidity providers.  Depending on the role of such  participants, the momentum component could vanish on more accessible markets. In all generality, one would expect the additional barriers to entry to limit the diversity of market participants and favour the appearance of significant non-reverting patterns in price dynamics as a result of the alignment of traders' incentives and constraints.     \\

\subsection{Testing for an arbitrary set of adjacency coefficients}~\label{ex:general_example}

\subsubsection{Setting and procedure}Apart from the simple configurations above, the weights are usually intractably tied to the value of the true parameter. As exposed with Proposition~\ref{proposition:the_chi_bar_distribution}, a conditional test can then be conducted. Recall that, from Susko's theorem~\cite{Susko}, conditionally on $k$ zeros being observed in the \textsc{mle}, the \textsc{lrs} $\Lambda_n$ closely follows a $\chi^2(p-k)$ distribution. The resulting procedure is the following: if $\Lambda_n$ is null, the test never rejects $\mathcal{H}_0$. If $\Lambda_n$ is positive, the full and sparse \textsc{mle} should differ and a number $k<p$ of null adjacency coefficients are observed in $\hat{\theta}_n$.  The test then rejects $\mathcal{H}_0$ when $\Lambda_n$ is above the quantile of chosen level for the $\chi^2(p-k)$ distribution.

\subsubsection{Simulation study} We apply the general procedure to a $10$-dimensional Hawkes processes. The exponential kernel was again retained to keep an acceptable computation time.  Under the null, the process is a juxtaposition of two $5$-dimensional Hawkes process with cyclic interaction patterns, each parameterised by a single adjacency coefficient $\alpha_i$. Under the alternative, four coordinate pairs interact, binding the two cycles through four weights $\gamma_1 \cdots \gamma_4$.  The precise graph of interactions is encoded into Figure~\ref{fig:pattern} below.

\begin{figure}[h!]
\begin{center}
    \includegraphics[width=0.45\textwidth]{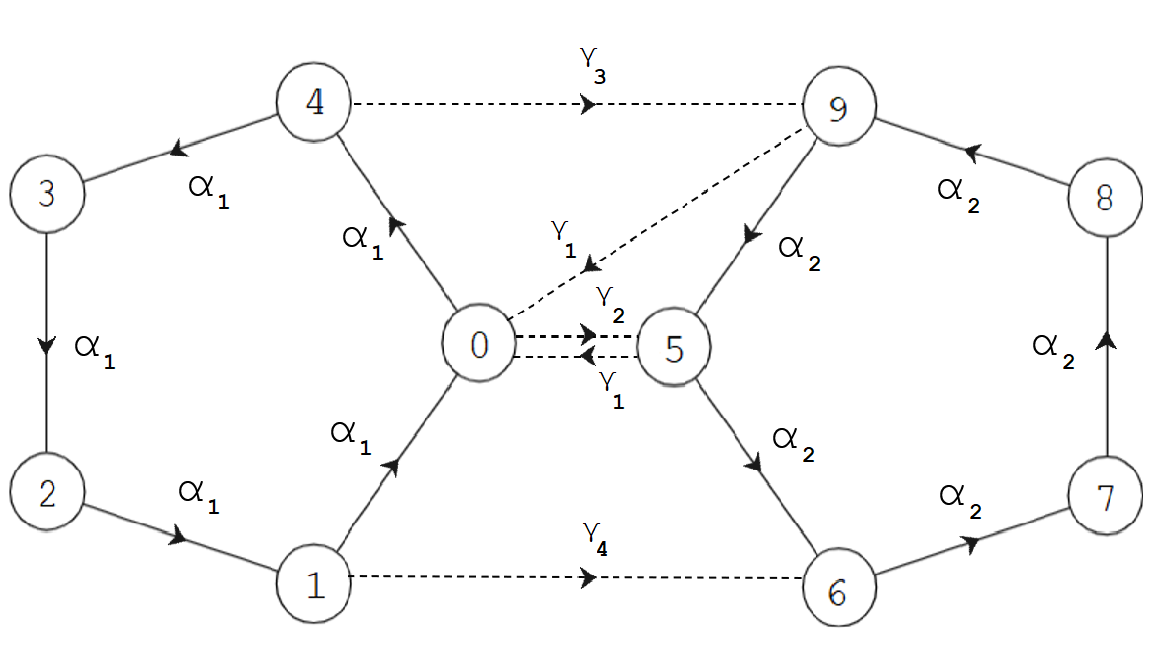}\hfill
\end{center}
\caption{Interaction graph of the example \textsc{mhp} model. Full edges: interactions under the null. Dotted edges: Additional interactions under the alternative.}
\label{fig:pattern}
\end{figure}

The  baseline and decay are constant and set to the same values $\mu=8$, $\beta=10$ for all $10$ coordinates. In figure~\ref{fig:conditional_test} are the \textsc{qq}-plots for the distribution of $\Lambda_n$ conditionally on the number of zeros being observed among the \textsc{mle}s, with a threshold at $\varepsilon=10^{-5}$. We observe the conditional distributions of the \textsc{lrs} converge well to the respective asymptotic counterparts they are given by Susko's theorem. An occasional degree mis-attribution seemingly persists, resulting in a more pronounced deviation of the distribution conditional on the effective degree of the \textsc{lrs} being $1$.   \\

\begin{figure}[h!]
\begin{center}
    \includegraphics[width=0.55\textwidth]{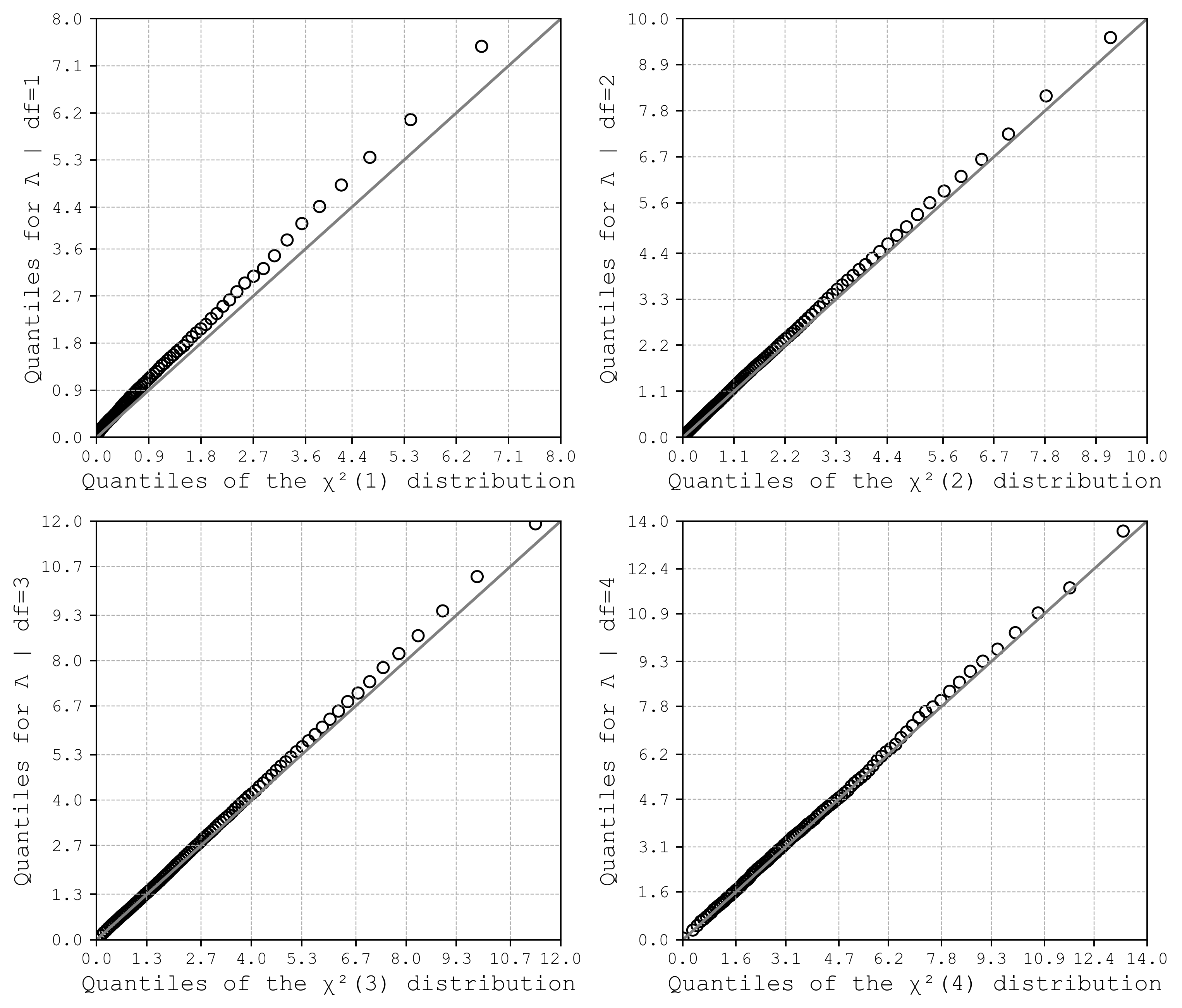}\hfill
\end{center}
\caption{Conditional \textsc{qq}-plots for the empirical distribution of $\Lambda_n$ against their theoretical $\chi^2(p-k)$ asymptotic distribution. $60 000$ simulations, $n=1000$.}
\label{fig:conditional_test}
\end{figure}

\subsubsection{Multi-asset price dynamics of late night delivery power futures}

We continue with German intraday power futures mid prices.  Our aim here is to retrieve the temporal dependency structure between multiple hourly futures. The dataset is the one of example~\ref{example:example2}. We consider the four hourly futures with deliveries starting at each hour from $22$ to $1$h on the $364$ trading sessions of $2019$. From each session is kept the $4$ hour time window from $21$h to midnight.  The model of Deschatre \& Gruet~\cite{DeschatreGruet} is extended to incorporate inter-asset excitation in addition to reversion and momentum factors.  This results in the kernel shape
\begin{equation*}
    \alpha= (\alpha_{kl})= \begin{bmatrix}
        a_1 & a_2 & 0 & 0 & 0 & 0 & 0 & 0\\
        a_2 & a_1 & 0 & 0 & 0 & 0 & 0 & 0\\
        \gamma_1 & 0 & a_1 & a_2 & 0 & 0 & 0 & 0\\
        0 & \gamma_1 & a_2 & a_1 & 0 & 0 & 0 & 0\\
        \gamma_4 & 0 & \gamma_2 & 0 & a_1 & a_2 & 0 & 0\\
        0 & \gamma_4 & 0 & \gamma_2 & a_2 & a_1 & 0 & 0\\
        \gamma_6 & 0 & \gamma_5 & 0 &  \gamma_3 & 0 & a_1 & a_2 \\
        0 & \gamma_6 & 0 & \gamma_5 &  0 & \gamma_3 & a_2 & a_1 \\
    \end{bmatrix}
\end{equation*}
where the simple cross-dependence structure is adapted from Bacry \textit{et al}~\cite{bacrylimit}[section 5.2]. The model forces a lead-lag relation wherein the less liquid late night futures react to price moves in the more liquid evening deliveries. This rules out temporal inconsistencies in which market factors materialising past the delivery period of a future may move its price by a ricocheting along the edges of the causal graph of $(\boldsymbol{N}_t)$.  We  have also taken advantage of the closeness of momentum and reversion factors estimates of example~\ref{example:example2}, keeping only two parameters for the $a_i$s. Likewise, we take only one decay parameter $\beta$ for the whole kernel. More precisely, our kernel has the shape
\begin{equation*}
    \varphi_{kl} : t,x \mapsto \alpha_{kl} x \exp( - \beta t).
\end{equation*}
A specificity of our setting is the asynchronicity of trading windows for each future.  Unrestricted trading of the asset expiring at hour $n-1$ halts one hour before it does for the future expiring at hour $n$, which is a substantial fraction of the estimation window.  Our marked Hawkes process lives on a multidimensional triangular time-frame $[0, T_1] \times \cdots \times [0,T_4]$, which only incurs an innocuous modification of the log-likelihood and does not alter our results. Accordingly, the baseline is parameterised by two $4$-dimensional parameters $(m_i)$ and $(\kappa_i)$ and by the known end times $(T_i)$:
\begin{equation*}
    \mu : t \mapsto \begin{bmatrix}
        m_1 e^{\kappa_1 \frac{t}{T_1}} & m_1 e^{\kappa_1 \frac{t}{T_1}} &
        m_2 e^{\kappa_2 \frac{t}{T_2}} & m_2 e^{\kappa_2 \frac{t}{T_2}} &
        m_3 e^{\kappa_3 \frac{t}{T_3}} & m_3 e^{\kappa_3 \frac{t}{T_3}} &
        m_4 e^{\kappa_4 \frac{t}{T_4}} & m_4 e^{\kappa_4 \frac{t}{T_4}} 
    \end{bmatrix}^\textsc{t}.
\end{equation*}

Transaction costs may incentivize a reduction in the number of simultaneously traded assets. Such efficiency gains should come naturally in a \textsc{mhp}-based approach as the model isolates direct interactions. We record in table~\ref{tab:multidim_intraday} the results of our estimations for adjacency parameters.\\

\begin{table}[h!]
  \begin{center}
    \begin{tabular}{  c  c c c c c c c c    c  }
    \hline
    Model  & $\hat{a}_1$ & $\hat{a}_2$  &  $\hat{\gamma}_1$& $\hat{\gamma}_2$& $\hat{\gamma}_3$& $\hat{\gamma}_4$& $\hat{\gamma}_5$& $\hat{\gamma}_6$ & \textsc{lrs}\\
    \hline
    every $\gamma_j>0$  & $4.26$ & $4.75$ & $1.34$ & $0.93$ & $0.95$ & $0.29$ &  $0.26$ & $0.13$ & \textsc{n.a}\\
    $\alpha_{ij}=0$ when $i>j+1$   & $3.9$ & $5.12$  &  $1.24$& $0.95$& $1.54$& \textsc{n.a}& \textsc{n.a}& \textsc{n.a}& $2361$\\
    $\alpha_{ij}=0$ when $i>j+2$   & $4.2$ & $4.89$  &  $1.19$& $1.22$& $1.47$& $0.29$& $0.9$& \textsc{n.a} & $1569$\\
    \hline
    \end{tabular}
    \caption{Inferred adjacency parameters for the first $364$ trading sessions of $2019$ on the German intraday power market. }
    \label{tab:multidim_intraday}
  \end{center}
\end{table}

\begin{comment}
    array([2.09857869e+00, 2.76468920e+00, 3.16230165e+00, 3.53515358e+00,
       1.71742183e+00, 2.08377056e+00, 2.14725090e+00, 1.42797087e+00,
       5.99995663e+02, 4.20349944e+00, 4.89027540e+00, 1.19074035e+00,
       1.12259159e+00, 1.47295776e+00, 2.98246474e-01, 9.29277956e-01])
       2.09281514,   2.76531495,   3.1300587 ,   3.52639298,
         1.72038594,   2.0884802 ,   2.10318444,   1.41963842,
       599.99668376,   3.95378948,   5.12212821,   1.24719453,
         0.9522869 ,   1.54279215])
\end{comment}

A first test is run with null hypothesis that direct dependence between futures vanishes beyond a $2$ hours delay, that is $\gamma_4=\gamma_5=\gamma_6=0$ against the alternative that they are all positive. We use a threshold of $\varepsilon=10^{-5}$ under which estimates are considered null. Since no estimates fall below $\varepsilon$, the p-value for the test is obtained by comparing the \textsc{lrs} with the quantile of appropriate level of a $\chi^2(3)$ distribution. As the null is rejected with high certainty, a second test is run with the null hypothesis that $\gamma_6=0$ against the one-sided alternative. The null is again rejected at the $0.05$ confidence level. Despite these unexpected results it is compelling to note the amplitude of interactions between two futures $i,j$ decreases fast as $T_i - T_j$ increases. In practice, the transactions costs of the intraday power market will often prevent executing on such weak signal.

\subsection{Comparison to other methods}

Our method is natural in that it leverages the specific property of linear Hawkes processes of remark~\ref{remark:clusters_addition} into a useful Poisson approximation of the Hawkes process. When $T$ is large enough to allow for credible trajectory-by-trajectory estimations, different procedures are usually preferred. In Bonnet \textit{et al.}~\cite{Martinez}, the authors estimate the \textsc{mle} on each trajectory and average the resulting sequence. In Deschatre \& Gruet~\cite{DeschatreGruet}, the likelihood for the full dataset (or the sum of each log-likelihood) is maximised. One may wonder whether the two more conventional approaches keep producing acceptable result as $T$ decreases and how they compare to our method. We apply the three options to an univariate Hawkes model with constant baseline and an exponential kernel $t \mapsto \alpha \beta \exp{-\beta t}$ .  In figure~\ref{fig:comparison} are the empirical distribution of the estimates for  $\alpha$ according to each method, the ground truth being an homogeneous Poisson process $\alpha=0$ with intensity $\mu=5$Hz. The simulations are performed on a horizon of $T=10s$. Note that this creates about the same number as in a trajectory of the real world dataset of example~\ref{sec:testting_for_excitation}.\\

\begin{figure}[ht!]
\begin{center}
    \includegraphics[width=0.55\textwidth]{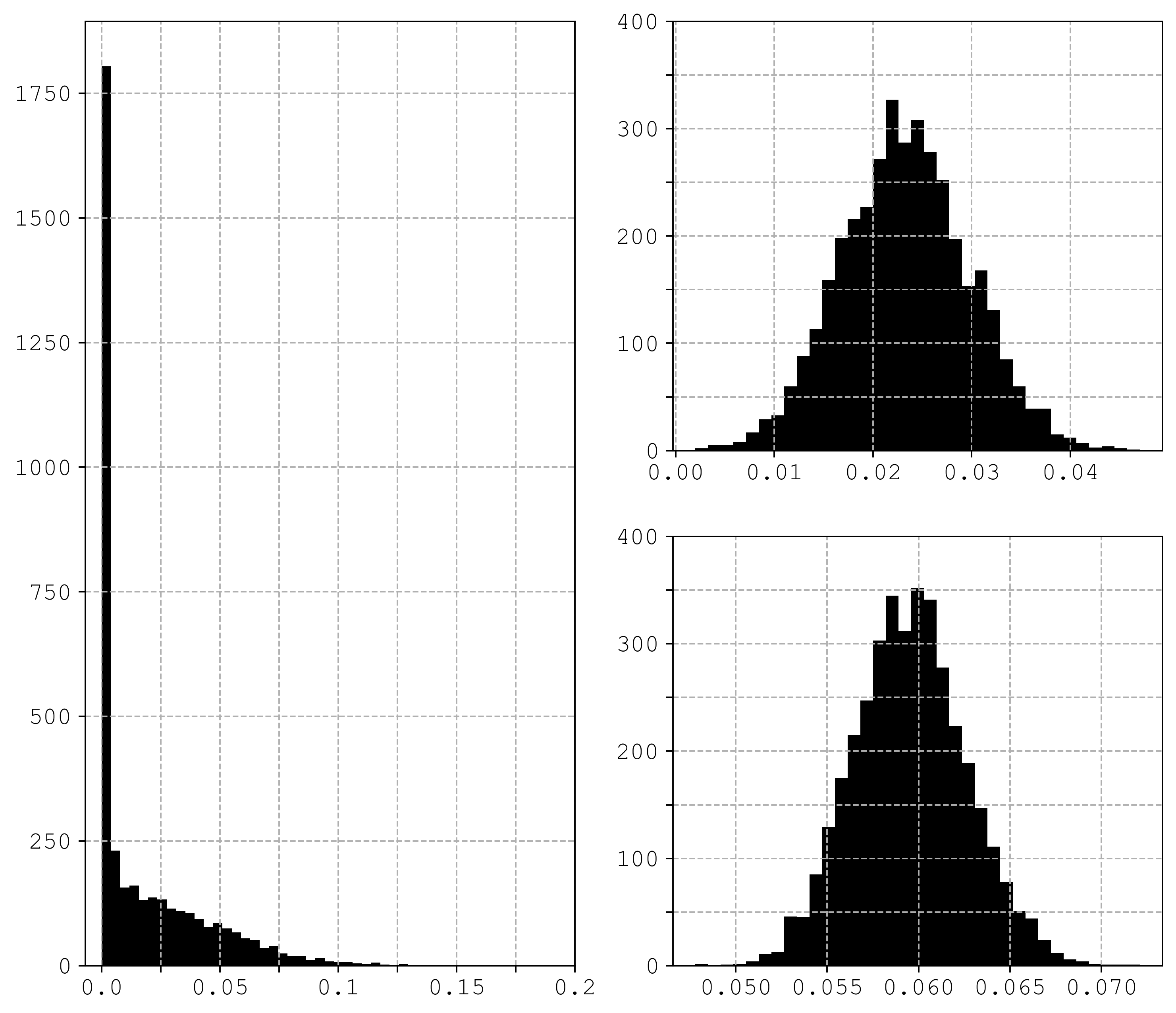}\hfill
\end{center}
\caption{Histograms for the estimators of $\alpha$ for each method, obtained for $N=4000$ simulations of $n=500$ trajectories of a Poisson process with intensity $5$Hz over a timeframe of $T=10$s. Left: our method (summing the $n$ trajectories). Top-right: maximising the sum of the $n$ log-likelihoods. Bottom-right: averaging the \textsc{mle}s over the $n$ trajectories. }
\label{fig:comparison}
\end{figure}

When applied to our setting, the two conventional approaches produce positive biases, which is especially detrimental to our main objective of testing for self-excitation. Consequently, we do not investigate further their properties at small time scales. Though the present example tells nothing of large scale properties, it is however interesting to observe our method produces the largest variance and might become a subpar solution as $T$ increases and standard estimators converge again. Note also that in the case of the average of \textsc{mle}s strategy, the bias is at least partly due to the positivity restrictions imposed upon the weighted adjacency $\alpha$, and could therefore disappear entirely in a more flexible model \textit{à la} Bonnet \textit{et al.}~\cite{Martinez} where $\alpha$ can reach negative values. 

\section{Preparation for the proofs}\label{section:preparation_for_proofs}

Throughout the proofs, $\theta_0$ is a parameter in $\Theta_0$, and $(\boldsymbol{N}^{(n)}_t)$ a point process with intensity $(\lambda^{(n)}_t(\theta_0))$. Unless specified otherwise,  any expectation $\mathbb{E}$ is taken under $\Prob(\theta_0)$. We will be concerned with the asymptotic behaviour of a variety of derivatives of the log-likelihood, which we introduce hereafter. Define the score $\boldsymbol{\mathcal{S}}(n,\cdot)
    =
    \big( \mathcal{S}(n,\cdot)_{i} \big) = \big( \partial_{\theta_i} \mathcal{L}(n,\theta) \big)$ where for any $\theta \in \Theta$ and $i \in \segN{1,d}$,
\begin{equation}\label{equ:definition_of_score}
 \mathcal{S}(n,\theta)_{i}
    =
    \sum_{k=1}^K
    \int_0^T
        \frac{\partial_{\theta_i} \lambda^{(n)}_{k,s} (\theta)}
        {\lambda^{(n)}_{k,s} (\theta)}
    \dif N^{(n)}_k
    -
    \sum_{k=1}^K
    \int_0^T
        \partial_{\theta_i}\lambda^{(n)}_{k,s} (\theta)
    \dif s.
\end{equation}

Similarly, $\partial_\theta \boldsymbol{\mathcal{S}}(n,\cdot)
    =
    (\partial_\theta \mathcal{S}(n,\cdot)_{ij})$, where for any $i,j \in \segN{1,d}$ 
\begin{align}\label{equ:score_derivative}
 \partial_\theta \mathcal{S}(n,\theta)_{ij}
    &=
    \sum_{k=1}^K
    \int_0^T
        \frac{\partial_{\theta_i} \partial_{\theta_j} \lambda^{(n)}_{k,s} (\theta)}
        {\lambda^{(n)}_{k,s} (\theta)}
        -
        \frac{\partial_{\theta_i} \lambda^{(n)}_{k,s} (\theta) \partial_{\theta_j} \lambda^{(n)}_{k,s} (\theta)}
        {(\lambda^{(n)}_{k,s} (\theta))^2}
    \dif N^{(n)}_k
    \\
    &-
    \sum_{k=1}^K
    \int_0^T \partial_{\theta_i}\partial_{\theta_j}\lambda^{(n)}_{k,s} (\theta)
    \dif s.
    \nonumber
\end{align}

Finally, define the empirical information $\boldsymbol{\mathcal{I}}(n,\cdot)= ( \mathcal{I}_{ij}(n,\cdot))_{ij}$ as
\begin{equation}\label{equ:empirical_information}
    \boldsymbol{\mathcal{I}}_{ij}(n,\theta)
    =
    \sum_{k=1}^K
    \int_0^T 
    \frac{
        \partial_{\theta_i}\lambda^{(n)}_{k,s} (\theta) 
        }
        { \lambda^{(n)}_{k,s} (\theta)}
    \frac{
        \partial_{\theta_j}\lambda^{(n)}_{k,s}(\theta)
        }
        { \lambda^{(n)}_{k,s} (\theta)}
        \dif N^{(n)}_{k,s}.
\end{equation}

  We first recall some well-known results, which make a recurrent appearance in our proofs.  
\begin{lemma}\label{lemma:pointwise_to_uniform}
   Let $( X_n(\nu) )_{n,\nu}$  be a family of random variables with values in some normed space $(E,\lVert \cdot \rVert)$, indexed by $n \in \mathbb{N}$ and $\nu$ in some compact $A\subset \R^d$. If for any   $\nu \in A$,  
    \begin{equation}\label{equ:domination_to_extend}
        X_n(\nu) = O_{\Prob}(1)
        \hspace{0.5cm}
        \big(
        \text{respectively: }
        o_\Prob(1)
        \big),
    \end{equation}

    and $(X_n(\cdot))$ is stochastically equicontinuous: for any $\eta>0$ and $\varepsilon>0$, there is  $\delta > 0$ such that
    \begin{equation}\label{equ:a_stochastic_continuity_condition_for_domination}
    \limsup_{n \to \infty}
        \Prob 
        \big[ 
        \sup_{ \lVert\nu_1 - \nu_2\lVert < \delta}
        \lVert
        X(n,\nu_1)
        -
        X(n,\nu_2)
        \lVert 
        > \eta \big]
        \leq \varepsilon,
    \end{equation}
    
    then~\eqref{equ:domination_to_extend} holds uniformly over  any compact $\Gamma \subset E$, that is 
    \begin{equation*}
        \sup_{\nu \in \Gamma} \lVert X_n(\nu) \lVert = O_\Prob(1)
        \hspace{0.5cm}
        \big(
        \text{respectively: }
        o_\Prob(1)
        \big).
    \end{equation*}
\end{lemma}

In the sequel, the conditions for the supremum of random variables to remain measurable will always be met, see appendix \textsc{c} of Pollard~\cite{Pollard} for more details. Lemma~\ref{lemma:pointwise_to_uniform} is commonly found in the context of parametric inference, and a similar argument appears in concise form in Chen \& Hall~\cite{FengInferenceForNonStationarySEPP}. We recall its proof for the sake of completeness.

\begin{proof}[Proof of Lemma~\ref{lemma:pointwise_to_uniform}]
       Let $(X_n(\nu))$ satisfy the conditions of Lemma~\ref{lemma:pointwise_to_uniform}, $\Gamma$ be a compact set, and let $\varepsilon>0$. For any $\delta >0$, we have a finite covering of $\Gamma$ with open sets included in balls with centres $(c_i)$, $i \in \segN{1,P_\delta}$, and radius strictly below $\delta$. For every $\nu \in \Gamma$, write $c(\nu)$ for the ball center closest to $\nu$.
      \begin{align}
          \Prob\big[ \sup_{ \nu \in \Gamma} \lVert 
          X_n(\nu) \lVert& 
          > M\big] \nonumber
    \\
    &\leq 
      \Prob\big[ 
        \sup_{ \nu \in \Gamma} \lVert 
          X_n(\nu) - X_n( c(\nu)) \lVert 
          >  \nicefrac{M}{2}
      \big]
      +
      \Prob\big[ \sup_{\nu \in \Gamma}
      \lVert
      X_n(c(\nu))
      \lVert > \nicefrac{M}{2}
      \big].
      \label{equ:equibound}
      \end{align}
      Now, under condition~\eqref{equ:a_stochastic_continuity_condition_for_domination}, for any sufficiently small $\delta$, the first term on the right hand side of~\eqref{equ:equibound} vanishes as $n \to \infty$. The second term is bounded by
      \begin{equation*}
          \sum_{i=1}^{P_\delta}
          \Prob \big(
      \lVert
      X_n(c_i)
      \lVert > M
      \big).
      \end{equation*}
      When $X_n(\cdot)$ is pointwise $ O_{\Prob}(1)$ over $\Gamma$ (respectively:  pointwise $o_{\Prob}(1)$), this bound may be made arbitrarily small by choosing $M$ sufficiently large  (respectively: uniformly in $M$), yielding the desired results.
\end{proof}

Stochastic equi-continuity results will most often proceed directly from the regularity of the kernel $\varphi$. In some other instances, we have recourse to a chaining Lemma. 

\begin{lemma}[Chaining, Pollard~\cite{Pollard}, Section VII.2, p. 142-146]\label{lemma:chaining}
    Let $(Z_x)$ be a stochastic process indexed by $x$ in a compact metric space $(E,d)$ with values in some normed space $(F,\lvert \cdot \lvert)$. If there exists a strictly positive constant $D$ such that for any $x,y \in F$, and any $\eta>0$,
    \begin{equation}\label{equ:chaining_decay_condition}
        \Prob\big[ 
            \lvert Z(x) - Z(y) \rvert  \geq \eta d(x,y)
        \big]
        \leq 
        2 \exp
        \Big[
            -\frac{1}{2} \frac{\eta^2}{D^2}
        \Big],
    \end{equation}

    then for any $0 < \varepsilon <1$,
    
\begin{equation}\label{equ:chaining_decay_result}
        \Prob\big[ 
        \text{There is some } x,y  \in  E\text{ such that} \hspace{0.2cm}
            \lvert Z(x) - Z(y) \rvert  
            \geq 26 D J(d(x,y))
            \hspace{0.2cm}
            \text{and } 
            \hspace{0.2cm}
         d(x,y)<\varepsilon 
        \big]
        \leq 
        2 \varepsilon,\\
    \end{equation}
    
    where $J$ is the covering integral
    \begin{equation*}
        J(\delta) 
        =
        \int_0^\delta 
        \Big(
        2 \ln \Big[ \frac{N(u)^2}{u}\Big]
        \Big)^{\frac{1}{2}}
        \dif u,
    \end{equation*}

    in which $N(u)$ is the covering number, that is the size of the smallest $u$-net on $E$ for the metric $d$. 
\end{lemma}

Observe as in Pollard~\cite{Pollard} that, in practice, one may replace $J(\cdot)$ in~\eqref{equ:chaining_decay_result} by any finite upper bound on the covering integral, and the exponential bound in~\eqref{equ:chaining_decay_condition}  by any bound decreasing in $\eta$ at the same speed. The chaining lemma requires concentration inequalities, which will arise from Lemma~\ref{lemma:concentration} below.

\begin{lemma}[Lemma 2.1 of Van de Geer~\cite{vandeGeerconcentration}]\label{lemma:concentration}
    Let $(M_t)$ be a martingale over $[0,T]$ with predictable variation $V_t= \langle M, M \rangle_t$ and such that $\lvert \Delta M_t \lvert = \lvert M_t - M_{t^-} \lvert< \kappa$ for all $t \in [0,T]$. Then, for any $a> 0$ and $b \in \R$,

    \begin{equation*}
        \Prob \big[ \text{There is some } t \in [0,T] \hspace{0.2cm} \text{such that} \hspace{0.2cm} \lvert M_t \lvert >  a  \hspace{0.2cm} \text{ and } \hspace{0.2cm} V_t \leq b^2 \big]
        \leq
        2 \exp \Big[ - \frac{1}{2} \frac{a^2}{a\kappa + b^2} \Big].
    \end{equation*}
\end{lemma}

Some more specific devices now come into play. The definition of the intensity~\ref{def:marked_hawkes_process} gives  rise to a variety of Volterra integral equations. We recall some elementary results regarding their resolvents. 
\begin{lemma}[On Volterra equations]\label{lemma:on_volterra}

Let $\kappa \colon \R^+ \mapsto \mathcal{M}_K(\R^+)$ be a locally bounded function such that
\begin{equation*}
    \rho \Big( \int_0^\infty \kappa(s) \dif s \Big)  < 1,
\end{equation*}

and $m$ a continuous function from $\R^+$ to $\R^K$. Define $\mathcal{K} = \sum_{k=1}^\infty \kappa^{\star k}$, wherein the  $\star k$-superscript denotes the $k$-fold convolution. Then, there exists a unique solution in locally bounded functions from $\R^+$ to $\R^K$ to the integral equation in $g$
\begin{equation}\label{equ:volterra_inequation}
    g(t)  =  m(t) + \int_0^t \kappa(t-s) g(s) \dif s, \hspace{0.2cm} t \in \R^+, 
\end{equation}
given by
\begin{equation*}
    g \colon  t \mapsto m(t) + \int_0^t \mathcal{K}(t-s) m(s) \dif s.
\end{equation*}
\end{lemma}
 See for instance Bacry \textit{et al.}~\cite{bacrylimit}[Lemma 3 p. 16] or Gripenberg~\cite{gripenberg}[Theorem 3.5 p. 44] for a proof. Lemma~\ref{lemma:normalized_intensity} is then an immediate consequence of Lemma~\ref{lemma:on_volterra}. Before getting to its short proof, we extend its assertion into the more general form of Corollary~\ref{coro:mean_intensity} which allows for values of $\theta$ different to the true parameter.

  \begin{corollary}[An extension of Lemma~\ref{lemma:normalized_intensity}]\label{coro:mean_intensity}
      Under assumption~\ref{ass: stability}, for any $t \in [0,T]$, $n \in \mathbb{N}^\star$ and $\theta \in \Theta$,
      \begin{equation*}
          \mathbb{E}\big[ 
          \frac{1}{n}
          \lambda^{(n)}_t(\theta) 
          \big]
          =
          \mu(t,\theta)
          +
          \int_0^t \langle F,  \varphi\rangle (t-s,\theta) h(s,\theta_0) \dif s,
      \end{equation*}
      where $s \mapsto h(s,\theta_0)$ is the unique solution in $\chi$ to the Volterra equation $ \chi = \mu(\cdot,\theta_0) + \langle F, \varphi \rangle (\cdot, \theta_0) \star \chi$. 
  \end{corollary}

  \begin{proof}
      Let $t \in [0,T]$ and $\theta \in \Theta$. Under assumptions~\ref{ass:g_continuity} and~\ref{ass: stability}, The process
      \begin{equation*}
          \Big( \int_{[0,s)\times \R^m} \mathbb{1}_{u < t} \varphi(t-u,x,\theta_0) \boldsymbol{M}^{(n)}(\dif u, \dif x)
          \Big),
          \hspace{0.2cm}
          s \in [0,T],
      \end{equation*}
      is a  martingale. In particular, at $s=t$,
      \begin{equation}\label{equ:smoothing}
      \mathbb{E} \Big[ 
          \int_{[0,t) \times \R^m}   \varphi  (t-s,x,\theta) 
          \boldsymbol{N}^{(n)}(\dif s, \dif x)
          \Big]
          =
          \int_0^t \langle F,  \varphi \rangle (t-s,\theta)
          \mathbb{E} \Big[  
          \frac{1}{n} \lambda^{(n)}_s(\theta_0)
          \Big]
          \dif s.
      \end{equation}
    Applying~\eqref{equ:smoothing} to the intensity~\eqref{equ:intensity}, one gets
    \begin{equation}\label{equ:almost_volterra}
        \mathbb{E}\Big[ 
          \frac{1}{n}
          \lambda^{(n)}_t(\theta) 
          \Big]
          =
          \mu(t,\theta_0)
          +
          \int_0^t \langle F, \varphi \rangle (t-s,\theta)
          \mathbb{E}\Big[ 
          \frac{1}{n}
          \lambda^{(n)}_s(\theta_0) 
          \Big]
          \dif s.
    \end{equation}
    Setting $\theta = \theta_0$ shows the normalised intensity at the true parameter satisfies a Volterra integral equation with baseline $\mu(\cdot, \theta_0)$ and kernel $\langle F, \varphi \rangle (\cdot, \theta_0)$. By virtue of Lemma~\ref{lemma:on_volterra}, we have Lemma~\ref{lemma:normalized_intensity}. Inserting the resulting expression for $\lambda^{(n)}_s(\theta_0)$ into the last term of~\eqref{equ:almost_volterra} yields corollary~\ref{coro:mean_intensity}.
  \end{proof}

  Following corollary~\ref{coro:mean_intensity}, define
\begin{align*}
\Bar{\lambda}(t,\theta,\theta_0)     =
    \big(\bar{\lambda}_{k} (t, \theta,\theta_0) \big)_k
    &=
    \Big(
    \mu_k(t,\theta) 
    +
    \int_0^t 
        \sum_{l=1}^K
         \langle F,  \varphi \rangle _{kl}(t-s,\theta)  h_l(s, \theta_0)
    \dif s\Big)_k. 
\end{align*}
With the same differential notation $    \partial^{\otimes 1}_\theta u 
    = (\partial_{\theta_i} u)_i
    $ 
    and
    $
    \partial^{\otimes 2}_\theta u = ( \partial_{\theta_i} \partial_{\theta_j} u)_{ij},
$ as Chen \& Hall~\cite{FengInferenceForNonStationarySEPP}, define for any $p \in \{1,2\}$
\begin{align*}
\partial_{\theta}^{\otimes p} \Bar{\lambda}(t,\theta,\theta_0) =
    \big( \partial^{  \otimes p}_\theta\bar{\lambda}_{k}(t,\theta,\theta_0)
    \big)
    &=
    \Big( \partial^{ \otimes  p}_\theta \mu_k(t,\theta) 
    +
    \int_0^t 
        \sum_{l=1}^K
        \partial^{\otimes  p}_\theta \langle F,  \varphi \rangle _{kl}(t-s,\theta)  h_l(s, \theta_0)
    \dif s\Big)_k.   
\end{align*}

 Note that $\Bar{\lambda}$ does not depend on $n$ and that $\Bar{\lambda}(\cdot, \theta_0,\theta_0)$ is simply $h(\cdot,\theta_0)$.
 As mentioned in the introduction, we will be concerned with obtaining a law of large number for the intensity. The deviation of $n^{-1} \lambda^{(n)}$ from $\Bar{\lambda}$  need be further precised. To this end, we extend a result of Jaisson \& Rosenbaum~\cite{JaissonRosenbaum}[proposition 2.1 p. 7] (see also Takeuchi~\cite{takeushi}[Proposition 1 p. 230]).

\begin{lemma}\label{lemma:intensity_decomposition}
Under Assumption~\ref{ass: stability}, for any $\theta \in \Theta$ and $t \in [0,T]$,
\begin{equation}\label{intensity_decomposition} 
    \frac{1}{n}
    \lambda^{(n)}(t,\theta) 
    =
    \Bar{\lambda}(t,\theta,\theta_0)
    + 
    \int_{[0,t) \times \R^m}
    \Psi(t-s,x,\theta_0,\theta) 
    \frac{M}{n}^{(n)}(\dif s, \dif x)
\end{equation}
where $\Psi$ is defined by
\begin{equation*}
    \Psi
    =
    \sum_{k \geq 1}
    \Tilde{\varphi}_k
    \hspace{0.5cm}
    \text{with}
    \hspace{0.5cm}
    \Tilde{\varphi}_k(s,x,\theta_0,\theta) 
    =
    \langle F,
    \varphi
    \rangle ^{\star (k-1)}(\cdot, \theta_0) \star \varphi(\cdot,x,\theta)(s).
\end{equation*}
\end{lemma}

\begin{remark}
    The function $\Psi$ defined in Lemma~\ref{lemma:intensity_decomposition} verifies for any $\theta \in \Theta$
\begin{equation*}
    \langle F, \Psi \rangle (\cdot,\theta_0,\theta)
    =
    \int_{\R^m} \Psi(\cdot,x,\theta_0,\theta) F(\dif x)
    =
    \langle F, \varphi\rangle(\cdot,\theta) + \langle F, 
    \varphi\rangle(\cdot,\theta)  \star 
    \sum_{k=1}^\infty \langle F,\varphi \rangle^{\star k}(\cdot,\theta_0).
\end{equation*}
    At $\theta=\theta_0$, this yields $
        \langle F, \Psi \rangle(\cdot,\theta_0) 
        =
        \sum_{k=1}^\infty 
        \langle F, \varphi \rangle ^{\star k}(\cdot,\theta_0)$. The notation of Lemma~\ref{lemma:intensity_decomposition} is thus consistent with that of Lemma~\ref{lemma:normalized_intensity} and Corollary~\ref{coro:mean_intensity}, which it implies.
\end{remark}

\begin{proof}[Proof of Lemma~\ref{lemma:intensity_decomposition}]
    Let $\theta \in \Theta$ and $n \in \mathbb{N}^\star$. One has for any $t \in [0,T]$
    \begin{equation}\label{equ:ready_for_volterra}
        \frac{1}{n}\lambda_t^{(n)}(\theta)
        =
        \mu(t,\theta_0) 
        +
        \int_{[0,t) \times \R^m} \varphi(t-s,x,\theta) \frac{1}{n}\boldsymbol{M}^{(n)}(\dif s, \dif x)
        +
        \int_0^t\langle F, \varphi \rangle (t-s,\theta) \frac{1}{n} \lambda_s^{(n)}(\theta_0) \dif s.
    \end{equation}
    At $\theta=\theta_0$ in particular,  $n^{-1} \lambda^{(n)}_t(\theta_0)$ verifies a Volterra equation with kernel $\langle F, \varphi\rangle (\cdot,\theta_0)$. Applying  Lemma~\ref{lemma:on_volterra}, using Fubini's theorem, and re-arranging the terms, for any $t \in [0,T]$, 
    \begin{align}
        \frac{1}{n} \lambda^{(n)}_t(\theta_0) 
        &=
        \mu(t,\theta_0)
        +
        \int_0^t \langle F, \Psi \rangle (t-s,\theta_0,\theta) \mu(s,\theta_0) \dif s
        +
        \int_{[0,t) \times \R^m}
        \varphi(t-s,x,\theta_0)
        \frac{1}{n}\boldsymbol{M}^{(n)}(\dif s,\dif x)
        \nonumber
        \\
        &+
        \int_{[0,t) \times \R^m}
        \int_0^{t-u} 
        \langle F, \Psi \rangle (t-u-s,\theta_0,\theta)
        \varphi(s,x,\theta_0)
        \dif s
        \frac{1}{n}
        \boldsymbol{M}^{(n)}(\dif u, \dif x).
        \label{equ:volterra_on_M}
    \end{align}
    
     Recalling that $\Bar{\lambda}(\cdot ,\theta_0,\theta_0) =
     \mu(\cdot ,\theta_0) + \langle F, \Psi \rangle (\cdot, \theta_0) \star \mu (\cdot, \theta_0)$ this yields expression~\eqref{intensity_decomposition} at $\theta=\theta_0$. Reinserting the obtained decomposition for $\lambda^{(n)}_s(\theta_0)$ into~\eqref{equ:ready_for_volterra} and using Fubini's theorem again ends the proof.
\end{proof}

It is useful to observe the properties of the kernel $\varphi$ extend readily to the resolvent $\Psi$. When $\varphi$ is split in the sense of Assumption~\ref{ass: kernel_separability}, $\Psi$ is split too with, for any $t,x \in \R^+ \times \R^m$ and $\theta \in \Theta$,
    \begin{equation}\label{equ:separability_psi}
        \Psi\colon (t,x,\theta_0,\theta)
        =
        g(x,\theta) \odot 
        \Big( f(\cdot,\theta) + f(\cdot,\theta) \star 
         \langle F, \Psi \rangle (\cdot, \theta_0,\theta_0) \Big)(t).
    \end{equation}
Furthermore, for any  $t,s \in [0,T]$ and $k \in \mathbb{N}^\star$
\begin{equation*}
    \lVert \langle F, \varphi \rangle' (t-s,\theta_0)  \langle F, \varphi \rangle^{\star k}(s,\theta_0) \rVert_2
    \leq 
    \sup_{u \in [0,T]} \lVert \langle F, \varphi \rangle' (u,\theta_0) \rVert_2 \lVert \langle F, \varphi \rangle^{\star k}(s,\theta_0) \rVert_2,
\end{equation*}
where  $\langle F, \varphi \rangle^{\star k}$ is integrable over $\R^+$ under Assumption~\ref{ass: stability} and \textit{a fortiori} integrable over $[0,T]$, hence $\langle F, \varphi \rangle^{\star k+1} \in \mathcal{C}^1([0,T])$ for any $k \in \mathbb{N}$. From similar elementary bounds and the dominated convergence theorem, one finds for any $\theta \in \Theta$ that $f(\cdot,\theta) + f(\cdot,\theta) \star 
         \langle F, \Psi \rangle (\cdot, \theta_0,\theta_0) \in \mathcal{C}^1([0,T])$ and $\Psi$ thus writes as the Hadamard product of $g$ with a continuously differentiable function of $t \in [0,T]$.

\section{Proof of the consistency of the \textsc{mle}}\label{section:large_sample_discussion}

 \subsection{Technical results} Our proof differs somewhat from the univariate case of Chen \& Hall~\cite{FengInferenceForNonStationarySEPP} in which the \textsc{mle} is seen as a critical point of $\mathcal{L}(n,\cdot)$.  Such characterisation falls short in the multivariate case, where $\theta_0$ is not necessarily an inner point of $\Theta$. As mentioned in the introduction, an \textsc{m}-estimator approach is preferred instead.  We are then compelled to prove the convergence in probability of the local likelihood  ratio
\begin{equation}\label{equ: likelihood_ratio}
    \frac{1}{n}
    \Lambda_n(\theta, \theta_0) 
    =
    \frac{2}{n} \big(
        \mathcal{L}(n,\theta)
        -
        \mathcal{L}(n, \theta_0)
        \big)
\end{equation}

to some limit constrast function
\begin{equation}\label{equ:limit_divergence}
    \bar{\Lambda}(\theta,\theta_0)
    =
    2 \sum_{k=1}^K
    \int_0^T
    \ln \Big\{
        \frac
        {\Bar{\lambda}_k (s,\theta,\theta_0)}
        {\Bar{\lambda}_k (s,\theta_0,\theta_0)}
    \Big\} \bar{\lambda}_k (s,\theta_0,\theta_0) 
    -
    (
        \Bar{\lambda}_k (s,\theta,\theta_0)
        -
        \Bar{\lambda}_k (s,\theta_0,\theta_0)
    )
    \dif s,
\end{equation}

 as will be precised in Lemma~\ref{lemma:sufficient_cv} (the multiplicative factor $2$  serves so as to homogenise~\eqref{equ: likelihood_ratio} with future notation). From~\eqref{equ:log_lkl_def}, the random function $\Lambda_n$~\eqref{equ: likelihood_ratio} writes as a functional of the intensity process $(\lambda_s^{(n)}(\theta))$. From the definition~\eqref{equ:aggregation} of $(N^{(n)}_t)$ as as sum process and the linearity of the intensity, for any $\theta \in \Theta$, we may write
      $\int_0^T \lambda_t^{(n)}(\theta) \dif t
     =
     \sum_{k=1}^n Y_k(\theta)$ where the $Y_k(\theta)$ are i.i.d and have the distribution of $\int_0^T \lambda_t^{(1)}(\theta) \dif t$. One may therefore already observe that, from the classical law of large numbers,
\begin{equation}\label{equ:law_of_large_numbers}
    \frac{1}{n} \int_0^T  \lambda^{(n)}_s(\theta) \dif s
    \xrightarrow[n \to \infty]{\Prob(\theta_0)\text{-a.s}} \int_0^T \Bar{\lambda}(s,\theta,\theta_0) \dif s.
\end{equation}

Corollary~\ref{coro:N_goes_to_H} below will serve to extend such results to a broad class of functionals of the intensity, encompassing $\Lambda_n$. We first prove a uniform convergence for the type of processes appearing in decomposition~\eqref{intensity_decomposition}.

\begin{lemma}\label{lemma:Doob}
    Let $ v \in L^2(F(\dif x))$. For any $\theta \in \Theta$,
    \begin{align*}
        \sup_{t \in [0,T]}
        \big\lVert 
            \int_{[0,t] \times \R^m}
                v(x)
            \frac{1}{n}
            \boldsymbol{M}^{(n)}(\dif s, \dif x)
        \big\rVert_2
        & =o_{\Prob(\theta)} (1)\\
        \sup_{t \in [0,T]}
        \big\lVert 
            \int_{[0,t] \times \R^m}
                v(x)
            \frac{1}{\sqrt{n}}
            \boldsymbol{M}^{(n)}(\dif s, \dif x)
        \big\rVert_2
         &=O_{\Prob(\theta)} (1)
    \end{align*}
\end{lemma}

\begin{proof}[Proof of Lemma~\ref{lemma:Doob}]
    Let $v=(v_{kl}) \in L^2(F(\dif x))$.  For any  $n \in \mathbb{N}^\star$. We define  the process
    \begin{equation}
    \label{equ:a_time_independant_martingale}
        \Tilde{\boldsymbol{M}}^{(n)}_t
        =
        \int_{[0,t] \times \R^m}
            v(x) 
        \boldsymbol{M}^{(n)}(\dif s, \dif x)
        =
        \big(
        \sum_{l=1}^K
        \int_{[0,t] \times \R^m} 
        v_{kl}(x)        
        M_l^{(n)}(\dif s, \dif x)
        \big)_k.
    \end{equation}
    Since $v \in L^2(F(\dif x))$, $(\Tilde{\boldsymbol{M}}_t)$ is a square integrable martingale (see Brémaud~\cite{Bremaudbook}[Theorem 5.1.33 p. 170]). Using Remark~\ref{remark:diagonal_covariation}, the covariation of $(\Tilde{\boldsymbol{M}}_t)$ simplifies as 
    \begin{equation*}
        \langle
        \Tilde{M}^{(n)}_{k,\cdot}, \Tilde{M}^{(n)}_{l,\cdot}
        \rangle_t
        =
        \sum_{p=1}^K 
        \int_0^t 
        \int_{\R^m}
            v_{kp}(x) v_{lp}(x)
        \lambda_{p,s}^{(n)}(\theta_0) \dif s F_p(\dif x).
    \end{equation*}
     
     Hence the predictable compensator of the sub-martingale $\lVert \Tilde{\boldsymbol{M}} \lVert^2_2$ simplifies too, so that
    \begin{align*}
    \big\lVert 
        \frac{1}{n} \Tilde{\boldsymbol{M}}^{(n)}
    \big\rVert^2_2
    -
    \sum_{k=1}^K
    \langle 
        \frac{1}{n} \Tilde{M}^{(n)}_{k,\cdot}
        &,
        \frac{1}{n} \Tilde{M}^{(n)}_{k,\cdot}
    \rangle\\
    &=
    \Big(
        \big\lVert 
            \frac{1}{n} \Tilde{\boldsymbol{M}}^{(n)}_t
        \big\rVert^2_2
        -
        \frac{1}{n}
        \int_0^t
            \Big\{
                \sum_{k=1}^K
                \sum_{p=1}^K
                \int_{\R^m} 
                    v^2_{kp}(x)
                F_k(\dif x)
            \Big\}
            \frac{1}{n}\lambda^{(n)}_{p,s}(\theta_0)
        \dif s 
    \Big)_{k,t}
    \end{align*}
    
   is also a martingale. From this it follows, together with Lemma~\ref{lemma:normalized_intensity} and the positivity of the $\lambda^{(n)}_{k,\cdot}$,
    \begin{equation*}
        \mathbb{E}
        \big[
            \big\lVert 
                \frac
                {1}
                {n}
                \Tilde{\boldsymbol{M}}^{(n)}_T
            \big\rVert^2_2
        \big]
        \leq 
        \frac{1}{n}
        \big(
            \sum_{k=1}^K \sum_{p=1}^K
            \int_{\R^m}   v_{kp}^2(x)  F_k(\dif x)
        \big)
        \lVert H(T,\theta) \lVert_1,
    \end{equation*}
    where $H(\cdot,\cdot)$ is defined in~\eqref{equ:H_function}, and the same bound with $n^{-1}$ omitted holds for $n^{- \nicefrac{1}{2}} \Tilde{M}$.
    Applying Doob's maximal inequality to the sub-martingale $\lVert \Tilde{\boldsymbol{M}} \lVert^2_2 $ with the bound above yields Lemma~\ref{lemma:Doob}.
\end{proof}

\begin{corollary}\label{coro:N_goes_to_H}
Under Assumptions~\ref{ass:baseline} to~\ref{ass:g_continuity} and Assumption~\ref{ass: stability}, for any $\theta \in \Theta$,
\begin{equation*}
    \sup_{t \in [0,T]}
    \big\lVert \frac{1}{n}\lambda_t^{(n)}(\theta)
    -
    \Bar{\lambda}(t,\theta,\theta_0) 
    \big\rVert_2
    =
    o_{\Prob(\theta_0)}(1).
\end{equation*}
\end{corollary}

\begin{remark}\label{remark:IPP_applies}
    The splitting Assumption~\ref{ass: kernel_separability} proves instrumental at this point.  It allows for the integral of $u(\cdot) \odot v(\cdot)$ against the counting measure $N^{(n)}(\dif s, \dif x)$ to be interpreted as an integral of $u(\cdot)$ against the Riemann-Stieltjes measure $ \int_{\R}v(x) N^{(n)}(\dif s, \dif x)$. Thus, usual integration by parts applies (see: Daley \& Vere-Jones~\cite{DVJ}[lemma 4.6.1]).
\end{remark}

\begin{proof}[Proof of Corollary~\ref{coro:N_goes_to_H}]
Let $\theta \in \Theta$. For any $n \in \mathbb{N}^\star$ and $t \in [0,T]$,
\begin{equation*}
    \frac{1}{n}\lambda^{(n)}_t(\theta) 
    -
    \Bar{\lambda}(t,\theta,\theta_0)
    =
    \int_{[0,t) \times \R^m}
    \Psi(t-s,x,\theta_0,\theta) \frac{1}{n} \boldsymbol{M}^{(n)}(\dif s, \dif x).
\end{equation*}
From~\eqref{equ:separability_psi}, we have some $u \colon  t,\vartheta,\theta \mapsto u(t,\vartheta,\theta) \in \mathcal{M}_K(\R)$ such that $u(\cdot,\theta,\theta_0) \in \mathcal{C}^1(\R^+)$ for any $\theta \in \Theta$, and $\Psi(t,x,\theta_0,\theta) = u(t,\theta,\theta_0) \odot g(x,\theta)$. Integrating by parts, for any $n \in \mathbb{N}^\star$ and $t \in [0,T]$,
\begin{align*}
    \frac{1}{n} \lambda^{(n)}_t(\theta) 
    -
    \Bar{\lambda}(t,\theta,\theta_0)
    &=
    u(0,\theta,\theta_0) \odot \int_{[0,t) \times \R^m} g(x,\theta) \frac{1}{n}\boldsymbol{M}(\dif u, \dif x)
    \\
    &+
     \int_0^t u'(t-s,\theta,\theta_0) \odot \int_{[0,s) \times \R^m} g(x,\theta) \frac{1}{n}\boldsymbol{M}(\dif u, \dif x) \dif s,
\end{align*}
hence for any $\theta \in \Theta$ and $n \in \mathbb{N}^\star$,
\begin{equation}\label{equ:intensity_bound}
    \sup_{t \in [0,T]}
    \big\lVert 
    \frac{1}{n} \lambda^{(n)}_t(\theta) 
    -
    \Bar{\lambda}(t,\theta,\theta_0)
    \big\rVert_2 
    \leq C\big( u(\cdot,\theta,\theta_0)\big)
    \sup_{ t \in [0,T]} 
    \big\lVert \int_{[0,t] \times \R^m} g(x,\theta) \frac{1}{n}\boldsymbol{M}^{(n)}(\dif u, \dif x) \big\rVert_2,
\end{equation}
where
\begin{equation*}
    C\big( u(\cdot,\theta,\theta_0)\big)
    =
    \lVert u(0,\theta,\theta_0) \rVert_2
    +
    \int_0^T 
    \lVert u'(s,\theta,\theta_0) \rVert_2 \dif s.
\end{equation*}
Corollary~\ref{coro:N_goes_to_H} then proceeds directly from Lemma~\ref{lemma:Doob}.
\end{proof}

We now extend the previous convergence to regular functions of the intensity. 

\begin{lemma}\label{lemma:clinet_like} 
        Let  $\Phi \in \mathcal{C}^0(\R^{+ \star},\R)$. Under Assumptions~\ref{ass:baseline} to~\ref{ass:g_continuity} and Assumption~\ref{ass: stability}, for any $\theta \in \Theta, k \in \segN{1,K}$,
    \begin{align*}
        \sup_{t \in [0,T]}
        \big\lvert
        \Phi \big(
            \frac{1}{n} 
            \lambda_{k,t}^{(n)}(\theta) 
        \big)
        - 
        \Phi\big( \Bar{\lambda}_k(t,\theta,\theta_0)
        \big)
        \big\rvert
        =o_{\Prob(\theta_0)}(1).
    \end{align*}
\end{lemma}

    \begin{proof}[Proof of Lemma~\ref{lemma:clinet_like}] Let $\theta \in \Theta$. Under Assumption~\ref{ass:baseline}, for any $n \in \mathbb{N}^\star$, $\lambda^{(n)}_k$ is strictly positive and $\Phi(n^{-1} \lambda^{(n)}_{k,s}(\theta))$ is indeed well defined for any $k \in \segN{1,K}$. By the inverse triangle inequality,  for any $t \in [0,T]$, and any $k \in \segN{1,K}$,
    \begin{equation}
    \label{equ:ready_for_heine}
    \inf_{s \in [0,T]}
            \Bar{
            \lambda}
            _k(s,\theta,\theta_0)
        +o_{\Prob(\theta_0)}(1)
        \leq  
            \frac{1}{n}
            \lambda^{(n)}_{k,t}(\theta)
        \leq 
        \sup_{s \in [0,T]}
            \Bar{
            \lambda}
            _k(s,\theta,\theta_0)        +o_{\Prob(\theta_0)}(1),
    \end{equation}
    where the negligible terms are uniform in time. Hence,  with probability tending to $1$ under $\Prob(\theta_0)$ as $n$ grows, $s \mapsto n^{-1} \lambda^{(n)}_s(\theta)$ eventually lies in some compact interval $J \subset \R$ in which $s \mapsto \Bar{\lambda}(s,\theta,\theta_0)$ takes values. Let $\varepsilon>0$. On account of Heine's Lemma, $\Phi$ is uniformly continuous over any compact subset of $\R$,  and we have some $\eta>0$ such that, for any two $x,y \in J$ verifying  $\lvert x - y\rvert < \eta$, one gets $\lvert \Phi(x) - \Phi(y) \rvert \leq \varepsilon$. From Corollary~\ref{coro:N_goes_to_H}, for any sufficiently large $n$, one has for any $t \in [0,T]$ that
    \begin{equation*}
        \big\lvert
            \frac{1}{n}
            \lambda^{(n)}_{k,t}(\theta) 
            -
            \Bar{\lambda}_k(t,\theta,\theta_0)
        \big\lvert 
        < 
        \eta,
    \end{equation*}
    with probability tending to $1$, thus proving Lemma~\ref{lemma:clinet_like}. 
    \end{proof}

The results above were written in the context of separable functions in the sense of Assumption~\ref{ass: kernel_separability}. They apply immediately to finite sum of such functions, provided they satisfy the required regularity conditions. Together with Assumptions~\ref{ass: kernel_separability},~\ref{ass:regularity_of_kernel} and~\ref{ass:regularity_of_weight}, this remark guarantees the convergence of corollary~\ref{coro:N_goes_to_H}  still holds for the class of functions on $\R^+ \times \R^m$ spanned by  $\partial_\theta^p \varphi(\cdot, \cdot, \theta)$, $\theta \in \Theta, p \in \{0,1,2\}$. Consequently, for any $\theta \in \Theta$ and $k \in \segN{1,K}$,

\begin{equation}
        \sup_{s \in [0,T]}
        \big\lVert
            \frac{1}{n} 
            \partial_\theta^{ \otimes p} 
            \lambda^{(n)}_{k,s}(\theta) 
        - 
            \partial_\theta^{\otimes p} 
            \bar{\lambda}_k(s,\theta,\theta_0)
        \big\lVert_1
        \xrightarrow[n \to \infty]{\Prob(\theta_0)}
        0.
        \label{equ:clinet_like_for_identity}
\end{equation}
Such results will prove useful as we tackle derivatives of $(\lambda^{(n)}_t)$ in the following sections.  Applying the arguments of the proof of Lemma~\ref{lemma:clinet_like} to~\eqref{equ:clinet_like_for_identity}, we may already state the results below before getting back to the consistency of $\hat{\theta}_n$.

\begin{lemma}\label{lemma:clinet_like_derivative}
  Under assumptions~\ref{ass:baseline} to~\ref{ass: stability}, for any $k \in \segN{1,K}$ and $i,j \in \segN{1,d}$, and $t \in [0,T]$, denote 
  \begin{equation*}
      Y^{(n)}_{k,ij}(t,\theta)= \big( \lambda^{(n)}_{k,t} (\theta), \partial_{\theta_i} \lambda^{(n)}_{k,t} (\theta), \partial_{\theta_j} \lambda^{(n)}_{k,t} (\theta), \partial_{\theta_i}\partial_{\theta_j} \lambda^{(n)}_{k,t} (\theta) \big),
  \end{equation*}
  
     and 
     \begin{equation*}
         \Bar{Y}_{k,ij}(t,\theta,\theta_0)= \big( \Bar{\lambda}_{k}(t,\theta,\theta_0), \partial_{\theta_i} \Bar{\lambda}_{k}(t,\theta,\theta_0), \partial_{\theta_j} \Bar{\lambda}_{k}(t,\theta,\theta_0), \partial_{\theta_i}\partial_{\theta_j} \Bar{\lambda}_{k}(t,\theta,\theta_0)\big).
     \end{equation*}
     
     Then, for any $\theta \in \Theta$ and  any continuous function $\Phi$ from $(0,\infty) \times \R^3$ to $\R$,
  \begin{equation*}
      \sup_{t \in [0,T]}
      \Big\lvert 
        \Phi\big( \frac{1}{n} Y^{(n)}_{k,ij}(t,\theta) \big)
        -
        \Phi\big( \Bar{Y}_{k,ij}(t,\theta,\theta_0) \big)
      \Big\lvert 
      =o_{\Prob(\theta_0)(1)}.
  \end{equation*}

\end{lemma}

\subsection{Proof of proposition~\ref{prop:large_sample_consitency} }
The proof relies essentially on Lemma~\ref{lemma:sufficient_cv} which states the convergence of the contrast function $\Lambda_n$ towards the limit function $\Bar{\Lambda}$ defined in~\eqref{equ:limit_divergence} . Before getting to its proof, remark as in Ogata~\cite{OgataMLE}[proof of Lemma 3 pp. 252--253] that $\Bar{\Lambda}$ re-expresses as 
\begin{equation*}
\theta \mapsto \Bar{\Lambda}(\theta,\theta_0)
=
    \sum_{k=1}^K \int_0^t L\Big( \frac{\bar{\lambda}_k(s,\theta,\theta_0)}{\bar{\lambda}_k(s,\theta_0,\theta_0)} \Big)
    \bar{\lambda}_k(s,\theta_0,\theta_0) \dif s,
\end{equation*}
where  
\begin{equation*}
    L \colon y \in (0,\infty) \mapsto \ln(y) +\frac{1}{y}  -1
\end{equation*} verifies $L(y) \leq 0$ for any $y>0$ with equality if and only if $y=1$. The function $\theta \mapsto \Bar{\Lambda}(\theta,\theta_0)$ thus attains a separable maximum at $\theta_0$ as long as $\bar{\lambda}_k(\cdot,\theta,\theta_0)$ and $\bar{\lambda}_k(\cdot,\theta_0,\theta_0)$ differ on an interval of non-null measure when $\theta \neq \theta_0$, which always holds true under assumptions~\ref{ass: stability} and~\ref{ass:identifiability}. When Lemma~\ref{lemma:sufficient_cv} holds, the \textsc{m}-estimator master theorem (see Van der Vaart~\cite{VanDenVaartAsymptoticStatistics}[Theorem 5.7]) therefore applies, ensuring the consistency of the \textsc{mle} as $n \to \infty$.

\begin{lemma}
\label{lemma:sufficient_cv}
    Under Assumptions~\ref{ass:baseline} to~\ref{ass:identifiability} and Assumption~\ref{ass: stability}
    \begin{equation}\label{equ:sufficient_cv}
    \sup_{\theta \in \Theta}
    \big\lVert
        \frac{1}{n}
        \Lambda_n(\theta_0,\theta) 
        -
        \Bar{\Lambda}(\theta,\theta_0)
    \big\rVert_1 
    \xrightarrow[n \to \infty]{\Prob(\theta_0)} 0.
    \end{equation}
\end{lemma}

\begin{proof}[Proof of Lemma~\ref{lemma:sufficient_cv}]

    We begin by showing convergence~\eqref{equ:sufficient_cv} holds pointwise at any $\theta \in \Theta$. Let $\theta \in \Theta$. Firstly,
    \begin{equation*}
    \sum_{k=1}^K
        \int_0^T
        \frac{1}{n}\{
        \lambda^{(n)}_{k,s}(\theta_0)
        -
        \lambda^{(n)}_{k,s} (\theta)
        \}
        -
        \{
        \Bar{\lambda}_k (s,\theta_0,\theta_0)
        -
        \Bar{\lambda}_k (s,\theta,\theta_0)
        \}
        \dif s
    \end{equation*}
     
    vanishes in probability under $\Prob(\theta_0)$ as $n \to \infty$ as a direct application of the uniform convergence  from Lemma~\ref{lemma:clinet_like}. Secondly, recalling that $h(\cdot,\theta_0)=\Bar{\lambda}(\cdot,\theta_0,\theta_0)$,
    \begin{align} 
    \sum_{k=1}^K \int_0^T
        \ln 
        \Big(
            \frac
            {\lambda^{(n)}_{k,s}(\theta_0)}
            {\lambda^{(n)}_{k,s}(\theta)}
        \Big) 
        \frac{1}{n}\dif N^{(n)}_{k,s}
    &-
        \ln 
        \Big(
            \frac
            {\Bar{\lambda}_k(s,\theta_0,\theta_0)}
            {\Bar{\lambda}_k(s,\theta,\theta_0)}
        \Big) h_k(s,\theta_0) \dif s
    \nonumber
    \\
    &=
    \sum_{k=1}^K \int_0^T \ln  
    \Big( 
    \frac
            {\Bar{\lambda}_k(s,\theta_0,\theta_0)}
            {\Bar{\lambda}_k(s,\theta,\theta_0)}\Big)
     \frac{1}{n}  \dif M_{k,s}^{(n)}
    \nonumber 
    \\
    &+
    \sum_{k=1}^K \int_0^T
    \ln 
        \Big(
            \frac
            {\Bar{\lambda}_k(s,\theta_0,\theta_0)}
            {\Bar{\lambda}_k(s,\theta,\theta_0)}
        \Big) 
    \{ \frac{1}{n}  \lambda^{(n)}_{k,s}(\theta_0)- \Bar{\lambda}_k(s,\theta_0,\theta_0)  \}\dif s
    \label{equ:LR_split}
    \\
    &-
    \sum_{k=1}^K \int_0^T
    \big\{
    \ln
        \Big(
            \frac
            { \frac{1}{n} \lambda^{(n)}_{k,s}(\theta_0)}
            { \frac{1}{n} \lambda^{(n)}_{k,s}(\theta)}
        \Big) 
    -
    \ln 
        \Big(
            \frac
            {\Bar{\lambda}_k(s,\theta_0,\theta_0)}
            {\Bar{\lambda}_k(s,\theta,\theta_0)}
        \Big) 
    \big\}
    \frac{1}{n} \dif N_{k,s}^{(n)}.
    \nonumber
    \end{align}
    
     Now, the third and last term on the right-hand side of~\eqref{equ:LR_split} is bounded in absolute value by
\begin{align}
    \sup_{t \in [0,T]} 
    \sum_{k=1}^K\Big\lvert 
    \ln
        \Big(
            \frac
            { \frac{1}{n} \lambda^{(n)}_{k,s}(\theta_0)}
            { \frac{1}{n} \lambda^{(n)}_{k,s}(\theta)}
        \Big) 
    -
    \ln 
        \Big(
            \frac
            {\Bar{\lambda}_k(s,\theta_0,\theta_0)}
            {\Bar{\lambda}_k(s,\theta,\theta_0)}
        \Big) 
    \Big\rvert
    \big\lVert 
        \frac{1}{n} \boldsymbol{N}^{(n)}_T
    \big\rVert_1,
    \label{equ:last_term_to_control_in_score}
\end{align}

where the two multiplicative factors in~\eqref{equ:last_term_to_control_in_score} are respectively $o_{\Prob(\theta_0)}(1)$ and $O_{\Prob(\theta_0)}(1)$ from Lemma~\ref{lemma:clinet_like} and the classical law of large numbers. As for the second term of~\eqref{equ:LR_split}, we have the bound
\begin{equation*}
    \sum_{k=1}^K
    \int_0^T \Big\lvert 
        \ln 
        \Big(
            \frac
            {\Bar{\lambda}_k(s,\theta_0,\theta_0)}
            {\Bar{\lambda}_k(s,\theta,\theta_0)}
        \Big) 
    \Big\rvert
    \dif s
    \sup_{t \in [0,T]}
    \big\lVert 
    \frac{1}{n}\lambda_{k,s}^{(n)}(\theta_0)
    -
    \bar{\lambda}_k(s,\theta,\theta_0)
    \big\rVert,
\end{equation*}

where  $s \mapsto \Bar{\lambda}(s , \theta, \theta_0) $ is locally bounded for any $\theta \in \Theta$ and Corollary~\ref{coro:N_goes_to_H} guarantees the second factor of the preceding decomposition vanishes as $n \to \infty$. Finally, a straightforward adaptation of the proof for Lemma~\ref{lemma:Doob} shows the first and remaining term on the right-hand side of~\eqref{equ:LR_split} is $o_{\Prob(\theta_0)}(1)$ too as a result of Doob's maximal inequality. Consequently, for every $\theta \in \Theta$
\begin{equation}
    \frac{1}{n} \Lambda_n(\theta_0,\theta)
    \xrightarrow[n \to \infty]{\Prob(\theta_0)}
    \Bar{\Lambda}(\theta,\theta_0) .
    \label{equ:score_pointwise_convergence}
\end{equation}

There remains to extend~\eqref{equ:score_pointwise_convergence} uniformly over $\Theta$. By virtue of Lemma~\ref{lemma:pointwise_to_uniform} it suffices to show the stochastic equicontinuity of $n^{-1} \Lambda_n - \Bar{\Lambda}$, which reduces to that of $\Lambda_n$ since $\Bar{\Lambda}$ is deterministic and continuous over the compact $\Theta$. Let $\varepsilon>0$. The baseline $\mu$ and the kernel $\varphi$ are  uniformly equi-continuous in $\theta$.   For any sufficiently close $\vartheta_1,\vartheta_2$ in $\Theta$ thus, one has for any $t \in [0,T]$ and $n \in \mathbb{N}^\star$ 
\begin{align*}
   \lVert 
       \frac{1}{n} \lambda^{(n)}_t(\vartheta1)
       &-
       \frac{1}{n} \lambda^{(n)}_t(\vartheta2)
    \rVert_1
    \\
    &\leq 
    \lVert 
    \mu(t,\vartheta_1)
    -
    \mu(t,\vartheta_2)
    \rVert_1
    +
    \int_0^T
    \lVert 
    \varphi(t-s,x,\vartheta_1)
    -
    \varphi(t-s,x,\vartheta_2)
    \rVert_1
    \frac{1}{n}
    N_{k,s}^{(n)}(\dif s, \dif x)
    \\
    &\leq 
    \varepsilon
    +
    \varepsilon
    \lVert \frac{1}{n} \boldsymbol{N}^{(n)}_T \rVert_1
    =
    \varepsilon( 1 + \lVert H(T,\theta_0) \rVert_1 + o_{\Prob(\theta_0)}(1))
    =
    \varepsilon O_{\Prob(\theta_0)}(1), 
\end{align*}
where $H(\cdot,\cdot)$ is defined in~\eqref{equ:H_function}. Hence $n^{-1} \lambda^{(n)}$ is equi-continuous in the sense of Lemma~\ref{lemma:pointwise_to_uniform}. The equicontinuity of $n^{-1}\Lambda(n,\cdot)$, 
 deduces directly from that of the normalized intensity by applying Heine's Lemma in the same fashion as in the proof of Lemma~\ref{lemma:clinet_like} and remarking again that $\frac{1}{n} N^{(n)} = O_{\Prob(\theta_0)}(1)$. The proof of Lemma~\ref{lemma:sufficient_cv} is then complete, yielding the consistency of the \textsc{mle} as a consequence.
      
\end{proof}

\section{Proof of the weak convergence of the likelihood ratio}
This section contains the proofs for Theorem~\ref{thm:distribution_of_lambda_as_a_projection} and Proposition~\ref{thmTCL}. A classical argument for the weak convergence of the \textsc{mle} relies on the second order expansion
\begin{equation}\label{equ:classic_expansion}
    n^{- \frac{1}{2}} \boldsymbol{\mathcal{S}}(n,\theta_0)
     =
    n^{- \frac{1}{2}} \boldsymbol{\mathcal{S}}(n,\hat{\theta}_n)
    +
    n^{-1}
    \partial_\theta \boldsymbol{\mathcal{S}}(n, \theta_0)
    \sqrt{n} (\hat{\theta}-\theta_0)
    +
    o_{\Prob} \big( \sqrt{n} (\hat{\theta} - \theta_0) \big).
\end{equation}

Under standard conditions, after cancelling the score at $\hat{\theta}_n$, the asymptotic normality of $n^{- \frac{1}{2}} \boldsymbol{\mathcal{S}}(n,\theta_0)$ together with Slutsky's Lemma yields a central limit theorem for the \textsc{mle}. The present setting differs in that $\boldsymbol{\mathcal{S}}(n,\hat{\theta}_n)$ may now take non-null values, pulling apart $\sqrt{n}(\hat{\theta}_n - \theta_0)$ and the score at $\theta_0$.  This leads to $\sqrt{n}(\hat{\theta}_n - \theta_0)$ behaving as the minimizer in $h$ over $\Theta - \{ \theta_0 \}$ of the quadratic form
\begin{equation}\label{equ:quadractic_approximation}
    q_Z \colon h \mapsto (Z-h)^\textsc{t} I(\theta_0)(Z-h)
\end{equation}
with $Z \sim \mathcal{N}(0,I(\theta_0)^{-1})$.  Approximation~\eqref{equ:quadractic_approximation} was introduced by Self \& Liang~\cite{Self-Liang}[Theorem 3 and equ. (3.1) p. 607]. The precise conditions for its validity are detailed in the master theorem for constrained \textsc{m}-estimators of Geyer~\cite{Geyer}[Theorem 4.4 and Remark p. 2004], the verification of which assumptions' constitutes the main objective of the next few results. Notably, we prove in proposition~\ref{prop:score_is_normal_at_true_parameter} the asymptotic normality of the score at the true parameter, extending a result of Chen \& Hall~\cite{FengInferenceForNonStationarySEPP}[equ. (12) p. 1021] to the multivariate case.  \\
\subsection{Technical results}

The limit function $\Bar{\Lambda}$ plays the role of $\mathbb{E} [ n^{-1} \mathcal{L}(n,\cdot) ]$ in Geyer's~\cite{Geyer} setting. Lemma~\ref{lemma:convergence_of_likelihood_derivatives} shows the derivatives of $\Bar{\Lambda}$ express as uniform limits of derivatives of $ n^{-1}  \mathcal{L}(n,\cdot)$. This will serve in the sequel so as to clarify some properties of $\Bar{\Lambda}$. Additionally, we recover the consistency of $\boldsymbol{\mathcal{I}}(n,\hat{\theta}_n)$ as an estimator of $I(\theta_0)$.

\begin{lemma}\label{lemma:convergence_of_likelihood_derivatives}
    Under Assumptions~\ref{ass:baseline} to ~\ref{ass: stability},
    \begin{align}
    \sup_{\theta \in \Theta}
    \lVert 
        \frac{1}{n}
        \boldsymbol{\mathcal{S}}(n,\theta)
        -
         \partial_\theta\Bar{\Lambda}(\theta,\theta_0)
    \lVert_2
    =o_{\Prob(\theta_0)}(1),
    \label{equ:lemma11_1}
    \\
    \sup_{\theta \in \Theta}
    \lVert 
        \frac{1}{n} \partial_\theta
        \boldsymbol{\mathcal{S}}(n,\theta)
        -
        \partial^{\otimes 2}_\theta \Bar{\Lambda}(\theta,\theta_0)
    \lVert_2
    =o_{\Prob(\theta_0)}(1).
    \label{equ:lemma11_2}
    \end{align}

Moreover, using notation~\eqref{equ:empirical_information}  for the empirical information, 
\begin{equation}
        \boldsymbol{\mathcal{I}}(n,\hat{\theta}_n)
        \xrightarrow[n \to \infty]{\Prob(\theta_0)}
        I(\theta_0)
        \label{equ:lemma11_3}
\end{equation}
\end{lemma}
\begin{proof}
    The normalised score $ n^{-1}\boldsymbol{\mathcal{S}}(n,\cdot)$ and its derivative write coordinate-wise as functionals of the type
    \begin{equation*}
    \sum_{k=1}^K \int_0^T
    U\big(\frac{1}{n}Y^{(n)}_{k,ij}(s,\theta)\big)  \frac{1}{n}\dif N^{(n)}_{k,s}
    +
    \int_0^T
    V\big(\frac{1}{n}Y_{k,ij}^{(n)}(s,\theta)\big) \dif s,
    \end{equation*}
    where $U,V$ are continuous functions and $Y^{(n)}$ is defined in Lemma~\ref{lemma:clinet_like_derivative}. Convergences~\eqref{equ:lemma11_1} and~\eqref{equ:lemma11_2}  then proceed from Lemma~\ref{lemma:clinet_like_derivative} in the exact same fashion Lemma~\ref{lemma:sufficient_cv} proceeds from Lemma~\ref{lemma:clinet_like} and we do not detail their proof. From the same arguments, $\boldsymbol{\mathcal{I}}(n,\cdot)$ converges uniformly in probability towards the asymptotic information function $I(\cdot)$. Together with the consistency of the \textsc{mle} this yields~\eqref{equ:lemma11_3}
\end{proof}

\begin{remark}
    Other possibilities exist for an estimator of $I(\theta_0)$. The one we retained does not require the computation of the limit function $h$ which is often non-explicit.  
\end{remark}

\begin{proposition}\label{prop:score_is_normal_at_true_parameter}
    Under Assumptions~\ref{ass:baseline} to ~\ref{ass: stability},
    \begin{equation}
        \frac{1}{\sqrt{n}} \boldsymbol{\mathcal{S}}(n,\theta_0)
        \xrightarrow[n \to \infty]{\mathcal{L}(\Prob(\theta_0))}
        \mathcal{N}(0,I(\theta_0))
        \label{equ:score_normality}
    \end{equation}
\end{proposition}

\begin{proof}
     Under $\Prob(\theta_0)$,  the score $\mathcal{S}(n,\theta_0)$ exhibits a martingale structure :
    \begin{align}
    \frac{1}{\sqrt{n}}
    \boldsymbol{\mathcal{S}}(n,\theta_0)
    &=
    \sum_{k=1}^K
    \int_0^T
        \frac{1}{\sqrt{n}}
        \frac
        {
            \partial^{}_\theta 
            \mu_k(s,\theta_0) 
            +
            \int_0^s \int_{\R^m}
                  \sum_{l}
                  \partial_{\theta}\varphi_{kl}(s-u,x,\theta_0) 
            \frac{1}{n}
            N^{(n)}_l(\dif u,\dif x)
        }
        {
            \mu_k(s,\theta_0) 
            +
            \int_0^s \int_{\R^m}
                \sum_{l}
                \varphi_{kl}(s-u,x,\theta_0)
            \frac{1}{n} N^{(n)}_l(\dif u,\dif x)
    }
    \dif M^{(n)}_{k,s}
    \nonumber \\
    &=
    \sum_{k=1}^K
    \int_0^T
        U^{(n)}_k(s) 
    \dif M^{(n)}_{k,s}, 
    \label{equ:score_at_0}
\end{align}

from which we define the process
\begin{equation*}
    \big( \Tilde{M}^{(n)}_t \big) =
    \big( \sum_{k=1}^K
    \int_0^t
        U^{(n)}_k(s) 
    \dif M^{(n)}_{k,s} \big).
\end{equation*}

 Using remark~\ref{remark:diagonal_covariation}, the predictable co-variation of $\Tilde{M}$ expresses as
    \begin{equation}
        \left\langle 
            \Tilde{M}^{(n)}
            ,
            \Tilde{M}^{(n)}
        \right\rangle
        (t)
        =
        \int_0^t
        \sum_{k=1}^K
            \frac
            {
            \left(
                \partial_{\theta}  \mu_k(s,\theta_0) 
                +
                \int_0^s \int_{\R^m}
                      \sum_{l}
                      \partial_{\theta}\varphi_{kl}(s-u,\theta_0) 
                \frac{1}{n}
                N^{(n)}_l(\dif u,\dif x)
            \right)^{\otimes 2}
            }
            {
                \mu_k(s,\theta_0) 
                +
                \int_0^s \int_{\R^m}
                    \sum_{l}
                    \varphi_{kl}(s-u,\theta_0)
                \frac{1}{n} N^{(n)}_l(\dif u,\dif x)
        }
    \dif s.
    \label{equ:covariation}
    \end{equation}
    
   Thanks to this simplification, expression~\eqref{equ:covariation} is quite similar to the one Chen \& Hall~\cite{FengInferenceForNonStationarySEPP} [equ. 12 p. 1021] or Ogata~\cite{OgataMLE}[Theorem 4 p. 255] obtain for the covariation of $\boldsymbol{\mathcal{S}}(n,\theta_0)$ in the univariate case, hence we may use the same arguments. Namely, it suffices to show that $\Tilde{M}$ verifies a martingale central limit theorem, as introduced by Rebolledo~\cite{Rebolledo} (see also Jacod \& Shiryaev~\cite{jacodbook}[Theorem VIII-3.11]). Define
    \begin{equation*}
        \Tilde{M}^{(n)}_\varepsilon (t)
        =
        \sum_{k=1}^K
        \int_0^t
        U^{(n)}_k(s)
        \mathbb{1}_{
        \lvert 
        U^{(n)}_k \lvert >\varepsilon} \dif M^{(n)}_{k,s},
    \end{equation*}

    where  $\lvert \cdot \lvert $  and  $>$ are to be understood component-wise. The necessary conditions of the martingale central limit theorem can be re-expressed as in Andersen \textit{et al.}  ~\cite{Andersen}[Th. II.5.1 and eq. (2.5.8) p. 83-84]) as
    \begin{align}
    \forall t \in [0,T], \hspace{0.2cm}
        \langle 
            \Tilde{M}^{(n)}
            ,
            \Tilde{M}^{(n)}
        \rangle
        (t)
        &\xrightarrow[n \to \infty]{\Prob(\theta_0)}
        V(t), \label{equ:covariation_convergence_condition}\\
    \forall \varepsilon>0,
    \forall t \in [0,T], \hspace{0.2cm}
        \langle 
        \Tilde{M}^{(n)}_\varepsilon 
        ,
        \Tilde{M}^{(n)}_\varepsilon 
        \rangle
        (t)
        &\xrightarrow[n \to \infty]{\Prob(\theta_0)}
        0. \label{equ:RAJ_condition}
    \end{align}

     Firstly, using Lemma~\ref{lemma:clinet_like} with equation~\eqref{equ:covariation}, condition~\eqref{equ:covariation_convergence_condition} is immediately satisfied with
     \begin{equation*}
     V(t)
     =
         \int_0^t
             \frac
            {
            \left(
            \partial_\theta \mu(s,\theta_0) 
            +
            \int_0^s 
            \partial_\theta 
                \langle F , \varphi \rangle (s-u,\theta_0) h(u,\theta_0) \dif u
            \right)^{\otimes 2}
            }
            {
            \mu(s,\theta_0) 
            +
            \int_0^s 
                \langle F,  \varphi \rangle (s-u,\theta_0) h(u,\theta_0)
            \dif u
            }
         \dif s.
     \end{equation*}

     Secondly, using remark~\ref{remark:diagonal_covariation} again, the predictable co-variation of  $\Tilde{M}^{(n)}_\varepsilon$ simplifies as in~\eqref{equ:covariation},
    \begin{equation*}
    \langle 
    \Tilde{M}^{(n)}_\varepsilon 
    ,
    \Tilde{M}^{(n)}_\varepsilon 
    \rangle (t)
    =
    \sum_{k=1}^K
    \int_0^t
    \frac
            {
            \left(
                \partial_{\theta}  \mu_k(s,\theta_0) 
                +
                \int_0^s \int_{\R^m}
                      \sum_{l}
                      \partial_{\theta}
                      \varphi_ {kl}(s-u,x,\theta_0) 
                \frac{1}{n}
                N^{(n)}_l(\dif u,\dif x)
            \right)^{\otimes 2}
            }
            {
                \mu_ k(s,\theta_0) 
                +
                \int_0^s \int_{\R^m}
                    \sum_{l}
                    \varphi_{kl}(s-u,x,\theta_0)
                \frac{1}{n} N^{(n)}_l(\dif u,\dif x)
        }
    \mathbb{1}_{
        \lvert 
        U^{(n)}_k \lvert >\varepsilon
        } \dif s.
    \end{equation*}

    where the integrand converges to $0$ uniformly in time since for every $k \in \segN{1,K}$,
    \begin{equation*}
        \sup_{t \in [0,T]}\lvert  U_k^{(n)}(t) \rvert  = O_{\Prob(\theta_0)} \big(  \frac{1}{\sqrt{n}} \big).
    \end{equation*}
   as a result of Lemma~\ref{lemma:clinet_like_derivative}. The Lindeberg condition~\eqref{equ:RAJ_condition} is thus satisfied too, and $n^{-\frac{1}{2}}\Tilde{M}^{(n)}$ converges in law for the Skorokhod topology to a Gaussian process with covariance $V$. Since convergence in Skorokhod topology implies convergence of finite-dimensional distributions, setting $t=T$,~\eqref{equ:score_normality} holds as a consequence.   
\end{proof}

We now introduce an extension of Lemma~\ref{lemma:Doob} so as to ensure later the  $\sqrt{n}$-consistency of certain functionals of the intensity.

\begin{lemma}[A uniform extension of Lemma~\ref{lemma:Doob}]\label{lemma:sqrt_consistency_preparation}
    Under assumptions~\ref{ass:baseline} to~\ref{ass: stability},,

        \begin{equation}\label{equ:bigO}
        \sup_{\theta \in \Theta} \sup_{t \in [0,T]}
        \Big\lVert
        \int_{[0,t] \times \R^m}
            g(x,\theta) 
        \frac{1}{\sqrt{n}}\boldsymbol{M}^{(n)}(\dif u, \dif x)
        \Big\rVert_1
        =
        O_{\Prob(\theta_0)}(1)
    \end{equation}

\end{lemma}

\begin{proof}[Proof of lemma~\ref{lemma:sqrt_consistency_preparation}] Recall from Lemma~\ref{lemma:Doob} that~\eqref{equ:bigO} already holds pointwise in $\theta$. Consider the family of processes $\{ Z^{(n)}(\theta) \}$ defined for any $t \in [0,T]$ and $ \theta \in \Theta$ by

\begin{equation}\label{equ:family_to_bound}
        Z^{(n)}_t(\theta)
        =
        \int_{[0,t] \times \R^m}
        g(x,\theta) \frac{1}{\sqrt{n}} \boldsymbol{M}^{(n)} (\dif s, \dif x).
\end{equation}

In view of Lemma~\ref{lemma:pointwise_to_uniform}, it thus suffices to show $\{ \sup_{t \in [0,T]} \lVert Z^{(n)}_t(\theta) \rVert_1 \}$  is stochastically equi-continuous in $\theta$ for Lemma~\ref{lemma:sqrt_consistency_preparation} to hold. This will derive from the chaining argument detailed in Lemma~\ref{lemma:chaining}, for which we require some concentration inequality on increments along $\theta$ of the family~\eqref{equ:family_to_bound}. We will work with the metrics $d_p : \Theta^2 \to \R^+ $ defined for any $p \in (0,1]$ by 
\begin{equation*}
    d_p(\theta,\nu) = \lVert \theta - \nu \lVert_2^p.
\end{equation*}

Now, from the reverse triangle inequality, for any $\theta,\nu \in \Theta$,
\begin{equation}\label{equ:reverse_triangular}
    \big\lvert
    \sup_{t \in [0,T]}
    \lVert 
    Z^{(n)}_t(\theta)
    \lVert_1 
    \hspace{0.1cm}
    -
    \sup_{t \in [0,T]}
    \lVert 
    Z^{(n)}_t(\nu)
    \rVert_1 
    \big\lvert
    \leq 
    \sup_{t \in [0,T]}
    \lVert 
    Z^{(n)}_t(\theta)
    -
    Z^{(n)}_t(\nu)
    \rVert_1,
\end{equation}

and, as far as tail bounds are concerned, we can look instead at the family of processes 
\begin{align*}
    Z^{(n)}_t(\theta,\nu)
    & =
    \int_{[0,t] \times \R^m}
        (g(x,\theta) - g(x,\nu))
    \frac{1}{\sqrt{n}}
    \boldsymbol{M}^{(n)}(\dif s, \dif x).
\end{align*}

 Note also that for any  $n \in \mathbb{N}$, $\theta,\nu \in \Theta$ and $\eta>0$,
\begin{equation}\label{equ:sum_bound}
    \Prob\big[ \sup_{t \in [0,T]} \lVert Z^{(n)}_t(\theta,\nu) \rVert_1 > \eta d_p(\theta,\nu) \big]
    \leq 
    \sum_{k=1}^K 
   \Prob\big[ \sup_{t \in [0,T]} \lvert Z^{(n)}_{k,t}(\theta,\nu) \rvert > \frac{\eta}{K} d_p(\theta,\nu) \big]
\end{equation}

and we can study each summand in the right-hand side of~\eqref{equ:sum_bound} separately to  alleviate the notation. For any $n\in \mathbb{N}^\star$, any $\theta,\nu \in \Theta$ and any $k \in \segN{1,K}$, from Remark~\ref{remark:diagonal_covariation}, the predictable variation of $(Z^{(n)}_{k,t}(\theta,\nu))$ writes
\begin{equation*}
     \sum_{l=1}^K \int_{\R^m} (g_{kl}(x,\theta) - g_{kl} (x,\nu))^2 F_l(\dif x) 
     \frac{1}{n} \int_0^t \lambda_{l,s}^{(n)} (\theta_0)  \dif s,
\end{equation*}

and is hence bounded uniformly in $t$ by
\begin{equation*}
    \sum_{i,j}
    \int_{\R^m}
        \lvert 
        g_{ij}(x, \theta) 
        -
        g_{ij}(x, \nu)
        \lvert^2
    F_j(\dif x)
    \frac{1}{n}\sum_{k=1}^K
    \int_0^T \lambda_{k,s}^{(n)} (\theta_0) \dif s.
\end{equation*}

Firstly,  under Assumption~\ref{ass:regularity_of_weight}, $x,\theta \mapsto g(x,\theta)$ is uniformly $p^*$-Hölder in $\theta$ for some $p^* \in (0,1]$, and there is some $C>0$ such that
\begin{equation}\label{equ:we_use_holder}
    \sum_{i,j}
    \int_{\R^m}
        \lvert 
        g_{ij}(x, \theta) 
        -
        g_{ij}(x, \nu)
        \lvert^2
    F_j(\dif x)
    \leq 
    C^2 \lVert \theta - \nu \lVert_2^{2p^*}
    =
    C^2 d_{p^*}(\theta,\nu)^2.
\end{equation}

Secondly, for any $\zeta>0$, from the continuity of the $\ell^1$ norm and the classical law of large numbers~\eqref{equ:law_of_large_numbers},  we have for any sufficiently large $n$ that, $\Prob(\theta_0)$-a.s,
\begin{equation}
\label{equ:LLN_in_use}
     \frac{1}{n}\sum_{k=1}^K
     \int_0^T \lambda_{k,s}^{(n)} (\theta_0)  \dif s
     \leq (1+\zeta) 
     \lVert H(T,\theta_0) \lVert_1.
\end{equation}

Setting $\zeta$ to some fixed value, say, $\zeta=1$, and then using~\eqref{equ:we_use_holder} and~\eqref{equ:LLN_in_use} together, one has some $N\in \mathbb{N}^\star$ such that, for any $n \geq N$,  any $t \in [0,T]$ and any $\theta,\nu \in \Theta$,
\begin{equation}\label{equ:covariation_bound} 
    \big\langle 
        Z_{k,\cdot}^{(n)} (\theta,\nu)
        ,
        Z_{k,\cdot}^{(n)}(\theta,\nu)
    \big\rangle_t
    \leq 
    2 C^2 
    \lVert 
    H(T,\theta_0)
    \lVert_1
    d_{p^*}(\theta,\nu)^2
    =
    \mathcal{K}
    d_{p^*}(\theta,\nu)^2
\end{equation}

$\Prob(\theta_0)$-a.s. Furthermore, the jumps of $(Z_{k,t}^{(n)}(\theta,\nu))$ are bounded by 
\begin{equation}\label{equ:jump_bound}
    \lVert g(x,\theta) - g(x,\theta) \lVert_1 
    \leq 
    C \lVert \theta - \nu \lVert_2^p 
    =
    C d_{p^*}(\theta,\nu).
\end{equation}

Applying the concentration inequality of Lemma~\ref{lemma:concentration} in light of~\eqref{equ:covariation_bound} and~\eqref{equ:jump_bound},  for any sufficiently large $n$, one has for any $\eta > 0$ and any $\theta,\nu \in \Theta$,
\begin{align*}
    \Prob
    \big[ \sup_{t \in [0,T]}
    \lvert 
        Z^{(n)}_{k,t}(\theta,\nu)
    \lvert 
    >
    \frac{1}{K}\eta d_{p^*}(\theta,\nu) 
    \big]
    &\leq 
    2 
    \exp 
    \Big[
        - 
        \frac{1}{2}
        \frac{\eta^2 d_{p^*}(\theta,\nu)^2}{ K^2( C d_{p^*}(\theta,\nu)^2 
        +
        \mathcal{K} d_{p^*}(\theta,\nu)^2) } 
        \Big]\\
        &\leq 
    2 \exp 
    \Big[
        - 
        \frac{1}{2}
        \frac{\eta^2 }
        { K^2(C 
        +
        \mathcal{K}) } 
        \Big].
\end{align*}

Taking $D^2  =K^2 \max ( C, \mathcal{K})$, summing over $k$ in~\eqref{equ:sum_bound}, and recalling~\eqref{equ:reverse_triangular} then,  for any large enough $n$, for any  $\eta>0$ one has for any $\theta, \nu \in \Theta$,
\begin{equation*}
    \Prob 
    \big[ 
    \lvert
    \sup_{t \in [0,T]}
    \lVert 
        Z^{(n)}_{t}(\theta)
    \lVert_1 
    -
    \sup_{t \in [0,T]}
    \lVert 
        Z^{(n)}_{t}(\nu)
    \lVert_1 
    \lvert 
    >
    \eta d_{p^*}(\theta,\nu) 
    \big]
    \leq 
    2 K 
    \exp 
    \Big[
        - 
        \frac{1}{2}
        \frac{\eta^2 }{ D^2  } 
        \Big].
\end{equation*}

Apart from the unimportant factor $K$, this is the sub-exponential bound we needed to apply the chaining Lemma~\ref{lemma:chaining}. Since we work with parametric families indexed on the compact $\Theta$, the finiteness of the covering integral is immediate in our case. Together with Lemma~\ref{lemma:pointwise_to_uniform}, the chaining Lemma~\ref{lemma:chaining} therefore yields Lemma~\ref{lemma:sqrt_consistency_preparation}.
\end{proof}

\begin{corollary}[A consequence of Lemma~\ref{lemma:sqrt_consistency_preparation}]\label{coro:sqrt_consistency_coro}
For any $p,m \in \{0,1\}$ and any $i,j \in \segN{1,K}$
\begin{equation*}
    \sup_{\theta \in \Theta }
    \sup_{t \in [0,T] }
    \big\lVert 
    \sqrt{n} \big\{
    \partial^p_{\theta_i}
    \partial^m_{\theta_j}
    \frac{1}{n} \lambda_{s}^{(n)}(\theta) 
    -
     \partial^p_{\theta_i}
    \partial^m_{\theta_j}
    \Bar{\lambda}(s,\theta,\theta_0)
    \big\}
    \big\rVert_2
    =
    O_{\Prob(\theta_0)}(1)
\end{equation*}
\end{corollary}

\begin{proof}
    We only consider the configuration $p=m=0$ since the proof for other cases follow from the same arguments.
    Consider again the bound~\eqref{equ:intensity_bound} from the proof of Corollary~\ref{coro:N_goes_to_H}, whence we have some $u \colon t, \vartheta,\theta \mapsto u(t,\vartheta,\theta) \in \mathcal{M}_K(\R)$ given by
    \begin{equation*}
        u(t,\theta,\theta_0)
        =
        f(t,\theta) +
        f(\cdot,\theta)
        \star 
        \langle F, \Psi \rangle(\cdot,\theta_0,\theta_0 ),
    \end{equation*}
    such that $u(\cdot,\theta,\theta_0) \in \mathcal{C}^1([0,T])$ for any $\theta \in \Theta$ and
    \begin{equation}\label{equ:intensity_bound_2}
    \sqrt{n}\big\lVert 
    \frac{1}{n}
    \lambda^{(n)}_t(\theta)
    -
    \Bar{\lambda}(t,\theta,\theta_0)
    \big\rVert_2
    \leq 
    C \sup_{\theta \in \Theta} \sup_{t \in [0,T]}
    \big\lVert
    \int_{[0,t] \times \R^m}
    g(x,\theta) \frac{1}{\sqrt{n}} \boldsymbol{M}^{(n)}(\dif u,\dif x)
    \big\rVert_2,
    \end{equation}
    where
    \begin{equation*}
        C = \sup_{\theta \in \Theta}  \Big( \lVert u(0,\theta,\theta_0) \rVert_2
        +
        \int_0^T \lVert u'(s,\theta,\theta_0) \rVert_2 \dif s
        \Big)
    \end{equation*}
    is finite under assumptions~\ref{ass:f_continuity} and~\ref{ass:regularity_of_kernel}.  Applying Lemma~\ref{lemma:sqrt_consistency_preparation} to~\eqref{equ:intensity_bound_2} yields Corollary~\ref{coro:sqrt_consistency_coro}. 
\end{proof}

\begin{corollary}[A consequence of Lemmata~\ref{lemma:sqrt_consistency_preparation} and~\ref{lemma:clinet_like_derivative}]\label{coro:more_stuff}
Let $ V,W \in \mathcal{C}^0(\R) $, $k \in \segN{1,K}$  $i,j \in \segN{1,d}$, and
\begin{equation*}
    U^{(n)}_{k,ij}(t,\theta) =     V \Big( \frac{1}{n} Y^{(n)}_{k,ij}(t,\theta)\Big)
    W \Big(  \Bar{Y}_{k,ij}(t,\theta,\theta_0)\Big),
\end{equation*}

where we have used the notation of Lemma~\ref{lemma:clinet_like_derivative}. Then, for any $q,r \in \{0,1\}$,
\begin{equation*}
    \sup_{\theta \in \Theta} \sup_{t \in [0,T]}
    \Big\lvert
    \int_0^t
    U_{k,ij}^{(n)}(s,\theta)
    \sqrt{n} 
        \big\{ 
        \frac{1}{n} \partial^q_{\theta_i}
    \partial^r_{\theta_j}
        \lambda_{k,t}^{(n)}(\theta) 
        -
        \partial^q_{\theta_i}
    \partial^r_{\theta_j} \Bar{\lambda}_{k}(t,\theta,\theta_0)
        \big\} 
        \frac{1}{n}  \dif N_{k,s}^{(n)}
    \Big\rvert
    =
    O_{\Prob(\theta_0)}(1)
\end{equation*}
\end{corollary}
\begin{proof}
 For any $t \in [0,T]$, $n \in \mathbb{N}^\star$ and $\theta \in \Theta$, the expression in Corollary~\ref{coro:more_stuff} is bounded by the product of the three terms
 \begin{align*}
     \sup_{\theta \in \Theta, t \in [0,T] }  \big\lvert U_{k,ij}^{(n)}(t,\theta) \big)  \big\rvert
    \sup_{\theta \in \Theta, t \in [0,T] }
    \sqrt{n}  \big\lVert 
        \frac{1}{n} \partial^p_{\theta_i}
    \partial^m_{\theta_j}
        \lambda_t^{(n)}(\theta) 
        -
        \partial^p_{\theta_i}
    \partial^m_{\theta_j} \Bar{\lambda}(t,\theta,\theta_0)
    \big\rVert_1
    \frac{1}{n} \lVert \boldsymbol{N}^{(n)}_T \lVert_1,
 \end{align*}
 which are all $O_{\Prob(\theta_0)}(1)$ as a result of Lemma~\ref{lemma:clinet_like}, Lemma~\ref{lemma:sqrt_consistency_preparation} and the classical law of large numbers, respectively.
\end{proof}

\subsection{Proof of Theorem~\ref{thm:distribution_of_lambda_as_a_projection}}

     We are now ready to prove Theorem~\ref{thm:distribution_of_lambda_as_a_projection}. As mentioned in the introduction, we plan to apply Theorem 4.4 of Geyer~\cite{Geyer}[p. 2003], which provides sufficient conditions for the epi-convergence in distribution of random functions.   More precisely we require a set of four conditions to be verified.\\

    \textit{First condition.} The limit function $\Bar{\Lambda}$  is twice differentiable and satisfies    \begin{equation}\label{equ:limit_gradient_is_null_at_true_parameter}
        \Bar{\Lambda}(\theta) = (\theta- \theta_0)^\textsc{t} V (\theta - \theta_0) +o ( \lVert \theta- \theta_0 \lVert^2)
    \end{equation}

    for some positive definite $V \in \mathcal{M}_d(\R)$. The regularity conditions above are immediately verified from $\Bar{\Lambda}$ being a uniform limit of twice continuously differentiable functions, with the non degeneracy of $V=I(\theta_0)$ guaranteed by Assumption~\ref{ass: Fisher_information_is_positive}. It remains to show the gradient of $\Lambda$ vanishes at $\Bar{\Lambda}$, which does not come naturally as $\theta$ may lie outside of the interior of $\Theta$. An immediate application of Proposition~\ref{prop:score_is_normal_at_true_parameter} does however yield, with probability tending to $1$,
    \begin{equation*}
        \boldsymbol{\theta_0}
        =
        \partial_\theta
        \Bar{\Lambda}(\theta_0) 
        =
        \lim_{n \to \infty} \frac{1}{\sqrt{n}} \Big( \frac{1}{\sqrt{n}}
        \boldsymbol{\mathcal{S}}(n,\theta_0) \Big)  = 0.
    \end{equation*}
    Condition~\eqref{equ:limit_gradient_is_null_at_true_parameter} hence holds true.\\

    \textit{Second condition.}  For any $n \in \mathbb{N}^\star$ and  $\theta$ in a neighbourhood of $\theta$, 
    \begin{equation}\label{equ:remainder_condition}
         \Lambda_n(\theta_0,\theta) 
        =
        \boldsymbol{D}_n^{\textsc{t}}
        (\theta - \theta_0) 
        + 
        \boldsymbol{R}_n(\theta)^\textsc{t}(\theta - \theta_0),
    \end{equation}   
    where the remainder $\boldsymbol{R}_n(\theta)$ is stochastically equicontinuous in the sense that, denoting by $\Bar{\boldsymbol{R}}$ the limit in probability under $\Prob(\theta_0)$ of $n^{-1} \boldsymbol{R}_n$, for any $\varepsilon>0$ and $\eta>0$ we may find a neighbourhood $K$ of $\theta_0$ such that
\begin{equation}\label{equ:equicontinuity_condition_for_tcl}
        \lim_{n \to \infty} 
        \Prob\Big[
            \sup_{\theta \in K} \sqrt{n}
            \big\lVert 
                \frac{1}{n}\boldsymbol{R}_n(\theta) 
                - 
                \boldsymbol{R}(\theta)
            \big\rVert
            >  
            \eta 
        \Big]
        \leq
        \varepsilon.
    \end{equation}

    The mean value theorem (or Lagrange form of the Taylor remainder) provides an expression for expansion~\eqref{equ:remainder_condition} with remainder
    \begin{equation}\label{equ:remainder_form}
        \boldsymbol{R}_n(\theta)^\textsc{t}
        =
        (\theta - \theta_0)^T
        \partial_\theta 
        \boldsymbol{\mathcal{S}}(n,\theta^\star)
    \end{equation}
    where $\theta^\star$ lies on the line between $\theta_0$ and $\theta$ and is correctly defined on account of the convexity of $\Theta$. A sufficient condition for the stochastic equicontinuity of the remainder is therefore that $\sqrt{n} \{ n^{-1} \partial_\theta \boldsymbol{\mathcal{S}}(n,\theta) - \partial_\theta \Bar{\boldsymbol{S}}(\theta,\theta_0)\}$ be $O_{\Prob(\theta_0)}(1)$ uniformly in $\theta$.  We use again the notation of Lemma~\ref{lemma:clinet_like_derivative}, for all $s,\theta \in [0,T] \times \Theta$,
    \begin{equation*}
        Y^{(n)}_{k,ij}(s,\theta) = ( 
        \lambda_{k,s}^{(n)},
        \partial_{\theta_i} \lambda_{k,s}^{(n)}(s,\theta), \partial_{\theta_j} \lambda_{k,s}^{(n)},
        \partial_{\theta_i} \partial_{\theta_j} \lambda_{k,s}^{(n)})(\theta).
    \end{equation*}
    
     From the definition of the score~\eqref{equ:definition_of_score} and its derivative~\eqref{equ:score_derivative}, we have some continuously differentiable function $\Phi$ such that,
     \begin{equation*}
         \partial_\theta \boldsymbol{\mathcal{S}}(n,\theta)
         =
         \Big\{ 
         \sum_{k=1}^K
         \int_0^T \Phi\big( \frac{1}{n} Y^{(n)}_{k,ij}(s,\theta) \big) \dif N_{k,s}^{(n)}
         -
         \int_0^T \partial_{\theta_i} \partial_{\theta_j} 
         \lambda^{(n)}_{k,s}(\theta) \dif s
         \Big\}_{1 \leq i,j \leq d},
     \end{equation*}

    so that following a re-arrangement of the terms reminiscent of the proof for Lemma~\ref{lemma:sufficient_cv}, one has  
    \begin{align}
        \sqrt{n} ( 
            \frac{1}{n} \partial_\theta \boldsymbol{\mathcal{S}}(n,\theta)
            &-
            \partial_\theta  \Bar{\boldsymbol{\mathcal{S}}}(\theta,\theta_0)
        )_{ij}
        =
        \sqrt{n} \sum_{k=1}^K \int_0^T
                \{                \partial_{\theta_i}\partial_{\theta_j}
                \bar{\lambda}_k(s,\theta,\theta_0)
                -
                \partial_{\theta_i}\partial_{\theta_j}
                \frac{1}{n} \lambda_{k,s}^{(n)}(\theta) 
                \}
            \dif s 
            \nonumber
        \\
        &+
        \sqrt{n}\sum_{k=1}^K
    \int_0^T
        \Phi\big(  \Bar{Y}_{k,ij}(s,\theta,\theta_0)\big)
        \{ 
        \frac{1}{n} \dif N^{(n)}_{k,s}
        -
        h_k(s,\theta_0) \dif s
        \}\label{equ:sqrt_expansion}
    \\
     &+
    \sqrt{n} \sum_{k=1}^K
    \int_0^T
        \Phi\big( \frac{1}{n} Y_{k,ij}^{(n)}(s,\theta)\big)
        -
        \Phi\big(  \Bar{Y}_{k,ij}(s,\theta,\theta_0)\big)
    \frac{1}{n} \dif N^{(n)}_{k,s}.
    \nonumber
\end{align}

Consider the last term of~\eqref{equ:sqrt_expansion}.Re-arranging the terms again, for any $s \in [0,T]$,
\begin{align}
    \sqrt{n} \big\{ \Phi\big( \frac{1}{n} Y_{k,ij}^{(n)}(s,\theta)\big) 
        &-
    \Phi\big(  \Bar{Y}_{k,ij}(s,\theta,\theta_0)\big)
    \big\}
    =
    \nonumber
    \\
    &
    \frac{
        \frac{1}{n} \partial_{\theta_j} \lambda_{k,s}^{(n)}(\theta)  
    }
    {
        \frac{1}{n}\lambda_{k,s}^{(n)}(\theta)
    }
    \sqrt{n}
    \Big\{ 
            \frac{1}{n}\partial_{\theta_i} \lambda_{k,s}^{(n)}(\theta) 
            - 
            \partial_{\theta_i} \Bar{\lambda}_k(s,\theta,\theta_0)
    \Big\}
    \nonumber
    \\
    +&
    \frac{
        \partial_{\theta_i} \Bar{\lambda}_k(s,\theta,\theta_0) 
    }
    {
        \frac{1}{n}\lambda_{k,s}^{(n)}(\theta)
    }
    \sqrt{n}  
    \Big\{ 
                \frac{1}{n}   \partial_{\theta_j} \lambda_{k,s}^{(n)}(\theta) 
                - 
              \partial_{\theta_j} \Bar{\lambda}_k(s,\theta,\theta_0)
            \Big\}
    \nonumber
    \\
    +&
    \frac{
        \frac{1}{n}\partial_{\theta_i}\Bar{\lambda}_k(s,\theta,\theta_0) \partial_{\theta_j} \Bar{\lambda}_k(s,\theta,\theta_0)
    }
    {
        \frac{1}{n}\lambda_{k,s}^{(n)}(\theta)\Bar{\lambda}_k(s,\theta,\theta_0)
    }
    \sqrt{n}
    \Big\{ 
        \frac{1}{n}\lambda_{k,s}^{(n)}(\theta) 
        - 
         \Bar{\lambda}_k(s,\theta,\theta_0)
    \Big\} \label{equ:this_term_OP1} \\
    +&
    \frac{
        1
    }
    {
        \frac{1}{n}\lambda_{k,s}^{(n)}(\theta)
    }
    \sqrt{n}
    \Big\{ 
        \frac{1}{n} \partial_{\theta_i}
        \partial_{\theta_j} \lambda_k^{(n)}(s,\theta) 
        - 
        \partial_{\theta_i}
        \partial_{\theta_j}\Bar{\lambda}_k(s,\theta,\theta_0)
    \Big\} \nonumber
    \\
    +&
    \frac{ \partial_{\theta_i}\partial_{\theta_j}\Bar{\lambda}_k(s,\theta,\theta_0)
    }
    {
        \frac{1}{n}\lambda_{k,s}^{(n)}(\theta)
        \Bar{\lambda}_k(s,\theta,\theta_0)
    }
    \sqrt{n}
    \Big\{ 
        \frac{1}{n}  \lambda_{k,s}^{(n)}(\theta) 
        - 
        \Bar{\lambda}_k(s,\theta,\theta_0)
    \Big\}.
    \nonumber
\end{align}

 All five terms on the right-hand side of equation~\eqref{equ:this_term_OP1} write as 
\begin{equation*}
    U \Big( \frac{1}{n} Y_{k,ij}^{(n)} (t,\theta) \Big) 
    V \Big( \Bar{Y}_{k,ij}^{(n)} (t,\theta) \Big) 
    \sqrt{n} 
    \big\{
        \partial^p_{\theta_i} \partial^m_{\theta_j} \frac{1}{n} \lambda^{(n)}_{k,t}(\theta) 
        -
        \partial^p_{\theta_i} 
        \partial^m_{\theta_j}\Bar{\lambda}_k(t,\theta) 
    \big\},
\end{equation*}

where $U,V$ are some continuous functions, $p,m \in \{0,1\}$, and $i,j \in \segN{1,d}$.  Applying Corollary~\ref{coro:more_stuff} hence, one has that the third and last term on the right-hand side of equation~\eqref{equ:sqrt_expansion} is uniformly $O_{\Prob}(1)$. Now, the second term on the right-hand side of~\eqref{equ:sqrt_expansion} splits into 
\begin{align}
\sum_{k=1}^K
    \frac{1}{\sqrt{n}}
    \int_0^T
    \Phi( \Bar{Y}_{ij,k}(s, & \theta,\theta_0))
    \dif M^{(n)}_{k,s} 
    \nonumber
    \\ &+
    \int_0^T
    \Phi( \Bar{Y}_{ij,k}(s,\theta,\theta_0))
    \sqrt{n}
    (
    \frac{1}{n}\lambda^{(n)}_{k,s}(\theta_0) 
    -
    \Bar{\lambda}_{k}(s,\theta_0,\theta_0)
    )
    \dif s. \label{equ:split}
\end{align}

The last term of the sum~\eqref{equ:split} is uniformly $O_\Prob(1)$ as a direct consequence of Corollary~\ref{coro:more_stuff}. The deterministic function $\Bar{Y}$ being differentiable in $t$ by virtue of Assumption~\ref{ass:regularity_of_mu} and~\eqref{equ:separability_psi}, we may once again integrate by parts as in the proof of Corollary~\ref{coro:N_goes_to_H}. Hence the first term of~\eqref{equ:split} is uniformly bounded by
\begin{equation*}
\sup_{\theta \in \Theta} 
\sum_{k=1}^K
    \lvert \Phi( \Bar{Y}_{ij,k}(0,\theta,\theta_0) \rvert 
    +
    \int_0^T
     \lvert \partial_s  \Bar{Y}_{ij,k}(s,\theta,\theta_0)
     \Phi' ( \Bar{Y}_{ij,k}(s,\theta,\theta_0)) \rvert  \dif s
     \frac{1}{\sqrt{n}}\sup_{t \in [0,T]} \lVert M^{(n)}_{t} \rVert_1, 
\end{equation*}
which is also $O_\Prob(1)$ as per Lemma~\ref{lemma:Doob}. Finally, the first and remaining integral appearing in the right-hand side of~\eqref{equ:sqrt_expansion} is uniformly $O_\Prob(1)$ as a consequence of Corollary~\ref{coro:sqrt_consistency_coro}. Condition~\eqref{equ:equicontinuity_condition_for_tcl} is therefore verified.\\

    \textit{Third condition} The random vector $\boldsymbol{D}_n$ in~\eqref{equ:remainder_condition} verifies a central limit theorem
    \begin{equation}\label{equ:TCL_condition}
        \frac{1}{\sqrt{n}}\boldsymbol{D}_n \xrightarrow[n \to \infty]{\mathcal{L}(\Prob(\theta_0))} \mathcal{N}(0,A)
    \end{equation}
    Condition~\eqref{equ:TCL_condition} is Proposition~\ref{prop:score_is_normal_at_true_parameter}.\\

    \textit{Fourth condition} The \textsc{mle} is consistent
    \begin{equation}\label{equ:consistency_condition}
        \hat{\theta}_n 
        =
        \theta_0    
        +
        o_{\Prob(\theta_0)}(1)
    \end{equation}
    Condition~\eqref{equ:consistency_condition} is Proposition~\ref{prop:large_sample_consitency}.\\
    
    The conditions for the constrained \textsc{m}-estimator master theorem are therefore satisfied. Writing
    \begin{align}
    \Lambda_n
    &=
        2 \big( 
            \mathcal{L}(n,\hat{\theta}_n)
            -
            \mathcal{L}(n,\hat{\theta}_{n}^0
        \big)
        \nonumber
        \\
        &=2 \big( 
            \mathcal{L}(n,\hat{\theta}_n)
            -
            \mathcal{L}(n,\theta_{0})
        \big)
        -
        2 \big( 
        \mathcal{L}(n,\hat{\theta}_{n}^0)
            -
         \mathcal{L}(n,\theta_0) 
        \big),
        \label{equ:lambda_as_two_terms}
    \end{align}

    and applying Geyer's results~\cite{Geyer}[Theorem 4.4 and Remark p. 2004 to theorem 4.4], we obtain the weak convergence of  $\Lambda_n$ towards
    \begin{equation}\label{equ:result_of_geyer}
        \inf_{h \in \text{H}} \big\{ 
            2h^\textsc{t} \boldsymbol{Z} 
            -
            h^\textsc{t} I(\theta_0) h
        \big\} 
        -
        \inf_{h \in \text{H}_0} \big\{ 
           2h^\textsc{t} \boldsymbol{Z} 
            -
            h^\textsc{t} I(\theta_0) h
        \big\},
    \end{equation}

    where $\boldsymbol{Z} \sim \mathcal{N}(0,I(\theta_0))$. Re-arranging the terms as in Self \& Liang~\cite{Self-Liang}[equ. (2.1) p. 606], one has
    \begin{align}
        2h^\textsc{t} \boldsymbol{Z} 
        -
        h^\textsc{t} I(\theta_0)h
        &=
        2h^\textsc{t} \boldsymbol{Z} 
        -
        h^\textsc{t} \boldsymbol{Z} 
        -
        \boldsymbol{Z}^\textsc{t} h  
        +
        \boldsymbol{Z}^\textsc{t} 
        I(\theta_0)^{-1} \boldsymbol{Z}
        -
        \big(I(\theta_0)^{-1} \boldsymbol{Z}-h\big)^\textsc{t} I(\theta_0) \big(I(\theta_0)^{-1} \boldsymbol{Z}-h\big), \nonumber \\
        &=
        \boldsymbol{Z}^\textsc{t} I(\theta_0)^{-1} \boldsymbol{Z}
        -
        \big(I(\theta_0)^{-1} \boldsymbol{Z}-h\big)^\textsc{t} I(\theta_0) \big(I(\theta_0)^{-1} \boldsymbol{Z}-h\big).\label{equ:self-liang_simplification}
    \end{align}

    Inserting~\eqref{equ:self-liang_simplification} into~\eqref{equ:result_of_geyer} and simplifying redundant terms, the asymptotic distribution of $\Lambda_n$ is the one of
    \begin{equation}
        \inf_{h \in \text{H}_0} \big\{ (\Tilde{\boldsymbol{Z}} - h)^\textsc{t} I(\theta_0) ( \Tilde{\boldsymbol{Z}} -h)
        \big\}
        -
        \inf_{h \in \text{H}} \big\{ (\Tilde{\boldsymbol{Z}} - h)^\textsc{t} I(\theta_0) ( \Tilde{\boldsymbol{Z}}-h)
        \big\},
        \label{equ:result_of_geyer_as_two_terms}
    \end{equation}

    where $\Tilde{\boldsymbol{Z}} \sim \mathcal{N}(0,I(\theta_0)^{-1})$. Up to a matrix decomposition, this is the distribution of Theorem~\ref{thm:distribution_of_lambda_as_a_projection}. Under the alternative, the weak convergence of the left-hand side of~\eqref{equ:lambda_as_two_terms} towards the right-hand side of~\eqref{equ:result_of_geyer_as_two_terms} is preserved, while the right-hand side of~\eqref{equ:lambda_as_two_terms} is bounded from below by
    \begin{equation*}
        n \inf_{\theta \in \Theta_0} \frac{1}{n}
        \big(\mathcal{L}(n,\theta) 
        -
        \mathcal{L}(n,\theta_0) \big)
        = 
        n \big(\inf_{\theta \in \Theta_0}
        \Bar{\Lambda}(\theta,\theta_0)
        +
        o_{\Prob(\theta_0)}(1) \big)
    \end{equation*}
    as a result of Lemma~\ref{lemma:sufficient_cv}. As $\inf_{\theta \in \Theta_0}
        \Bar{\Lambda}(\theta,\theta_0)<0$ since $\theta_0$ is out of $\Theta_0$ in the alternative, this ensures the consistency of the test.

    \subsection{Proof of proposition~\ref{thmTCL}}\label{section:proof_of_TCL}
    From the previous discussion and Geyer~\cite{Geyer}[Theorem 4.4] still, 
    \begin{equation}
        \sqrt{n}( \hat{\theta}_n - \theta_0)
        \xrightarrow[n  \to \infty]{\mathcal{L}(\Prob(\theta_0))}
        \argmin_{\theta \in \text{H}+\theta_0}
        \big\{ \big( \Tilde{\boldsymbol{Z}} - (\theta-\theta_0))^\textsc{t} I(\theta_0) ( \Tilde{\boldsymbol{Z}}-(\theta-\theta_0)\big)
        \big\} ,\label{equ:large_MLE}\end{equation} 
        which is the first assertion of Proposition~\ref{thmTCL}, and
        \begin{equation}
        \sqrt{n}( \hat{\theta}_{n}^0 - \theta_0)
        \xrightarrow[n  \to \infty]{\mathcal{L}(\Prob(\theta_0))}
         \argmin_{\theta \in \text{H}_0+\theta_0}
        \big\{ \big(\Tilde{\boldsymbol{Z}} - (\theta-\theta_0))^\textsc{t} I(\theta_0) ( \Tilde{\boldsymbol{Z}}-(\theta-\theta_0)\big)
        \big\}, \label{equ:sub_MLE}
    \end{equation}

     where $\text{H}$ is the Painlevé-Kuratowski limit of $\sqrt{n}(\Theta - \{\theta_0\})$ and $\text{H}_0$ that of $\sqrt{n}(\Theta_0 - \{\theta_0\})$. Now, for any $u$ in some $\R^k$ and $l \in \segN{1,k-1}$, we denote by
    $u_{\leq l}$ the sub-vector made of the first $l$ coordinates of $u$. Similarly, for any $M \in \mathcal{M}_k(\R)$ we denote by $M_{\leq l, \leq l}$ the first $l \times l$ principal submatrix of $M$, by $M_{\leq l, >l}$ the $l \times (k-l)$ submatrix on its right, and so on. The sparse \textsc{mle} writes
    \begin{equation*}
        \hat{\theta}_{n}^0
        =
        \begin{bmatrix}
            \hat{\theta}_{1,n}^0
             \cdots 
             \hat{\theta}_{d-p,n}^0
            , 0 \cdots 0
        \end{bmatrix}^\textsc{t}
        =
        \begin{bmatrix}
            (\hat{\theta}_{n,\leq d-p}^0)^\textsc{t}
            , 0 \cdots 0
        \end{bmatrix}^\textsc{t}.   
    \end{equation*}
    Since $\theta_{0,\leq d-p}$ lies in the in the interior of its sub-parameter space, $\text{H}_0$ is a linear space and the $\argmin$ in~\eqref{equ:sub_MLE} is attained at the critical point $h^\star \in \R^{d-p}$ of 
    \begin{equation*}
    \begin{bmatrix}
        h_1, \cdots h_{d-p}
    \end{bmatrix}^\textsc{t}
        \mapsto
        ( \boldsymbol{Z}^\textsc{t} - 
        \begin{bmatrix}
            h_1 \cdots h_{d-p}, 0 \cdots 0 
        \end{bmatrix}
        )
        I(\theta_0)
        (\boldsymbol{Z}
        -
        \begin{bmatrix}
            h_1 \cdots h_{d-p}, 0 \cdots 0 
        \end{bmatrix}^\textsc{t}),
    \end{equation*}

    which expresses as
    \begin{equation}\label{equ:variance_to_re_express}
         h^\star = Z_{\leq d - p}
        -
        (I(\theta_0)_{\leq d-p, \leq d-p})^{-1}
        I(\theta_0)_{ \leq d-p,>d-p}
        Z_{>d-p}.
    \end{equation}

     After some elementary manipulation of bloc-matrix inverses, the variance of the Gaussian variable on the right-hand side of~\eqref{equ:variance_to_re_express} simplifies as the inverse of $I(\theta_0)_{\leq d-p, \leq d-p}$ (see Appendix~\ref{appendix:variance_simplification}). One finds
    \begin{equation*}
        \sqrt{n} (\hat{\theta}_{n,\leq d-p}^0 - \theta_{0,\leq d-p})
        \xrightarrow[n \to \infty]{\mathcal{L}(\Prob(\theta_0)}
        \mathcal{N} \big(0,(I(\theta_0)_{\leq d-p, \leq d-p})^{-1} \big).
    \end{equation*}
    This is the second and final assertion of Proposition~\eqref{thmTCL}.

\section{Acknowledgements}
The author would like to thank T. Deschatre, P. Gruet, and M. Hoffmann for helpful discussions
and their fruitful input during the preparation of this work; and E. Cognéville and F. Trieu for the
pre-processing of the German intraday power market order book.

\begin{appendices}
\section{Appendix}

\subsection{On the variance of the sparse \textsc{mle}}\label{appendix:variance_simplification}Lemma~\ref{Lemma:Schur} is part of the routine technicalities of likelihood-based model selection  (see for instance  Van der Vaart~\cite{VanDenVaartAsymptoticStatistics}[Problem 5 p. 241]). Since its short proof is rarely stated in full detail we recall it here for the sake of completeness.  The sub-vector notation is the one of section~\ref{section:proof_of_TCL} and we suppose, without loss of generality, that the null adjacency coefficients constitute the last $p$ coordinates of the true parameter $\theta_0$ . 
\begin{lemma}\label{Lemma:Schur}
    The asymptotic variance under $\Prob(\theta_0)$ of the first $d-p$ coordinates of $\sqrt{n}(\hat{\theta}_n^0 - \theta_0)$ is the inverse of  $I(\theta_0)_{\leq d-p, \leq d-p}$, where $I(\theta_0)$ is the $d \times d$ asymptotic information matrix.
\end{lemma}
\begin{proof}
    Write
    \begin{equation*}
        I(\theta_0)
        =
        \begin{bmatrix}
            A & B \\
            B^\textsc{t} & C
        \end{bmatrix}.
    \end{equation*}
    where $A$ is a $(d-p) \times (d-p)$ submatrix. Then, using Schur's  notation, one has
    \begin{equation*}
        I(\theta_0)^{-1}
        =
        \begin{bmatrix}
            (\nicefrac{I}{D})^{-1} & (\nicefrac{I}{D})^{-1} B D^{-1} \\
            D^{-1} B^\textsc{t} (\nicefrac{I}{D})^{-1} & (\nicefrac{I}{A})^{-1}
        \end{bmatrix},
    \end{equation*}
    where the Schur complements
    \begin{align*}
        (\nicefrac{I}{D})^{-1} &= A - B D^{-1} B^\textsc{t}, \\
        (\nicefrac{I}{A})^{-1} &= D - B^\textsc{t} A^{-1} B,
    \end{align*}
    verify 
    \begin{align}
        (\nicefrac{I}{D})^{-1} 
        &= A^{-1} + A^{-1} B  
        (\nicefrac{I}{A})^{-1}
        B^\textsc{t} A^{-1},
        \label{equ:schur_formula1}
        \\
        (\nicefrac{I}{A})^{-1}
        &=
        D^{-1}
        +
        D^{-1} B^\textsc{t}
        (\nicefrac{I}{D})^{-1}
        B D^{-1}.
        \label{equ:schur_formula2}
    \end{align}
      From the proof of proposition~\ref{thmTCL} the asymptotic variance of 
    \begin{equation*}
        \sqrt{n}
        (\hat{\theta}_{n}^0
        -
        \theta_{0,\leq d-p}
        )
    \end{equation*}
    under $\Prob(\theta_0)$ is the variance of 
    \begin{equation*}
        \boldsymbol{Z}_{\leq d-p}
        -
        A^{-1} B^\textsc{t} \boldsymbol{Z}_{> d-p}.
    \end{equation*}
    whre $\boldsymbol{Z} \sim \mathcal{N}(0,I(\theta_0)^{-1})$. Following a straightforward simplification,
    \begin{equation*}
        \mathbb{V}ar(\boldsymbol{Z}_{\leq d-p}
        -
        A^{-1} B^\textsc{t} \boldsymbol{Z}_{> d-p})
        =
        (\nicefrac{I}{D})^{-1}
        -
        A^{-1} B
        (\nicefrac{I}{A})^{-1} B^\textsc{t} A^{-1}
    \end{equation*}
    which reformulates thanks to~\eqref{equ:schur_formula1} as
    \begin{equation*}
        A^{-1}
        =
        (I(\theta_0)_{\leq d-p, \leq d-p})^{-1}
    \end{equation*}
\end{proof}
\subsection{On the Chi-Bar Distribution}\label{annex:chi_bar_def}

The existence of the $\Bar{\chi}^2$ distribution as a mixture of $\chi^2$ distributions for regular enough models is a traditional result in non-standard likelihood theory. Proposition~\ref{Prop:chi_bar_mixture} recalls the correspondence between the Fisher information and the mixture weights of the chi-bar distribution for models parameterised on a closed orthant. This result was originally proposed by Kudo~\cite{Kudo}[Theorem 3.1 p. 414] and can also be regarded as a special case of Shapiro~\cite{Shapiro}[Theorem 3.1 p. 138]. We provide a short alternative proof to Kudo's~\cite{Kudo} to ensure the present paper remains self-contained.  Let $d \in \mathbb{N}^\star$ and define the orthant
\begin{equation}\label{equ:cone}
    C = \{ \theta=(\theta_i) \in \R^d \lvert \hspace{0.2cm}\theta_i \geq 0 \hspace{0.2cm} \text{for all i } \in \segN{1,d} \}.
\end{equation}

For any $p \in \segN{1,d}$ we denote by $\mathcal{P}_k^p$ the set of all combinations of size $k$ of $\segN{1,p}$, or, up to a bijection,
\begin{equation*}
    \mathcal{P}_k^p
    =
    \{ (i_1, \cdots, i_k) \lvert 1 \leq i_1 < \cdots < i_k \leq p \},
\end{equation*}

and for any $I_k \in \mathcal{P}_k^p$ we abuse the notation to denote 
\begin{equation*}
    ^c I_k = (j_1, \cdots, j_{p-k}), \hspace{0.2cm} j_1 < \cdots <j_{p-k}, \hspace{0.2cm} j_l \in \segN{1,p} \setminus I_k .
\end{equation*}

For any $p \in \segN{0,p}$, define also the sparse sub-spaces 
\begin{equation*}
    C_p = \{ \theta \in C \lvert \text{ there are exactly } p \hspace{0.1cm} \text{null coefficients in } \theta \}, 
\end{equation*}

and
\begin{equation*}
    C^{I_k} = \{ \theta \in C \: \lvert \: \theta_i = 0 \iff i \in I_k\},
\end{equation*}

so that for any $p\geq 1$, $C_p$ is the union of the disjoints sets $C^{I_k}$ with $I_k \in \mathcal{P}^p_k$,  and $C_0=C$. Finally we let $\mathcal{A}$ be a positive definite $d \times d$ real matrix. Picking $p \in \segN{1,d}$, we let 
\begin{equation*}
    m = \begin{bmatrix}
    0 & \cdots & 0 & m_{p+1} & \cdots & m_d
\end{bmatrix}^\textsc{t} \in C^{(1, 2, \cdots p)} \subset C_p,
\end{equation*} 

and $X \sim \mathcal{N}(m,\mathcal{A}^{-1})$, and  we define the quadratic form 
\begin{equation*}
    q^\mathcal{A}_X \colon z \in \R^d \mapsto (X-z)^\textsc{t} \mathcal{A}(X-z).
\end{equation*}

\begin{defi}[Chi-bar distribution]\label{definition:chi_bar}
     The $\Bar{\chi}_\mathcal{A}^2(p)$ distribution is the law of 
\begin{equation}\label{equ:chi_bar_definition}
    \ell^\mathcal{A}(X) = \min_{z \in C_p} q^\mathcal{A}_X(z) -  \min_{z \in C} q^\mathcal{A}_X(z).
\end{equation}

\end{defi}

In the rest of this article we have omitted the variance subscript from the $\Bar{\chi}^2$ notation to match conventions in force in the literature. Do note that regardless of this practice, a $\Bar{\chi}^2$ distribution cannot be characterised by its degrees of freedom alone. 
\begin{Proposition}\label{Prop:chi_bar_mixture}
    The Chi-bar law writes as a mixture of $\chi^2$ distributions
    \begin{equation*}
        \Bar{\chi}^2_\mathcal{A} (p)
        =
        \sum_{i=1}^p \omega^\mathcal{A}_i \chi^2(p-i),
    \end{equation*}
    
    where,  defining $F_{I_k} \sim N(0,(\mathcal{A}^{-1}_{I_k})^{-1})$ and $G_{I_k} \sim N(0,(\mathcal{A}_{^c I_k})^{-1})$  the mixing probabilities are given by
    \begin{equation*}
        \omega^\mathcal{A}_k 
        =
        \sum_{I_k \in \mathcal{P}_k^p } \Prob \big[ F_{I_k} \in {\R^{-}}^k \big] 
        \Prob \big[ G_{I_k} \in {\R^{+}}^{p-k} \times \R^{K-p} \big].
    \end{equation*}
\end{Proposition}

\begin{proof} In the rest of the proof, denote 
    
    \begin{align*}
        \Tilde{z} =\argmin_{z \in C} ( z - X)^T \mathcal{A} (z-X)
        \hspace{0.5cm}
        \text{and}
        \hspace{0.5cm}
        \Tilde{z}^p =\argmin_{z \in C_p} ( z - X)^T \mathcal{A} (z-X).
    \end{align*}
    Let us first clarify the distribution of the first term on the right of~\eqref{equ:chi_bar_definition}. By Karush-Kuhn-Tucker (\textsc{kkt}), there exist some random coefficients $\lambda^T \in {R^+}^p$  such that $\lambda_i  \Tilde{z}_i^p = 0$ for  $i = 1\hdots k$, and

    \begin{equation}\label{equ:KKT}
        ( \Tilde{z}^p - X)^T \mathcal{A}
        =
        \begin{bmatrix}
            \lambda_1 & \cdots & \lambda_p  & 0 & \cdots & 0
        \end{bmatrix}
        =
        \begin{bmatrix}
            \lambda & 0 & \cdots & 0
        \end{bmatrix}
    \end{equation}
    which given that $\Tilde{z}^p _1 = \Tilde{z}^p_p=0$ as $\Tilde{z}^p \in C_p$ yields 

    \begin{equation}\label{equ:dist_of_lambda}
        \lambda   = (( \mathcal{A}^{-1})_{\leq p, \leq p})^{-1} \:  X_{\leq p} \sim \mathcal{N}(0, (\mathcal{A}^{-1}_{\leq p, \leq p})^{-1}).
    \end{equation}

    Plugging~\eqref{equ:KKT} back into $q^\mathcal{A}_X$,

    \begin{align*}
        \min_{z \in C_p} q^\mathcal{A}_X(z) 
        =
        ( \Tilde{z}^p- X)^T \mathcal{A} (\Tilde{z}^p - X)
        =
        \begin{bmatrix}
            \lambda & 0 & \cdots & 0
        \end{bmatrix}
        ( \Tilde{z}^p - X)
        =
        \lambda^T 
        (\mathcal{A}^{-1})_{\leq p, \leq p} \lambda \\
    \end{align*}

    which together with~\eqref{equ:dist_of_lambda} shows that $\min_{z \in C_p} q^\mathcal{A}_X(z) $ follows a $ \chi^2(p)$ distribution and writes as

    \begin{equation*}
        \min_{z \in C_p} q^\mathcal{A}_X(z) = \sum_{k=1}^p \frac{X_k^2}{\mathcal{A}^{-1}_{kk}}.
    \end{equation*}

    We may actually proceed in a similar fashion for the second term of~\eqref{equ:dist_of_lambda}. Let $k \in \segN{1,p}$. The event $\{ \Tilde{z} \in C_k \}$ decomposes into the mutually exclusive $\{ \Tilde{z} \in C^{I_k} \}$, $I_k \in \mathcal{P}^p_k$,  upon each of which the behaviour of $\ell^\mathcal{A}_X$ is the same. We may assume $I_k = \segN{1,k}$ without loss of generality.  On the event $\{ \Tilde{z} \in C^{\segN{1,k}} \}$ then, $\Tilde{z}$ coincides with the solution of the minimisation program of $q^\mathcal{A}_X$ over $C^{\segN{1,k}}$, which is strongly dual, and verifies the \textsc{kkt} conditions~\eqref{equ:KKT} with $p$ replaced by $k$. In particular,

    \begin{equation}\label{equ:KKT_above_k}
        Y^{(k)} = ((\mathcal{A}^{-1})_{\leq k \leq k})^{-1} X_{\leq k} = - 
        \begin{bmatrix}
            \lambda_1 \cdots  \lambda_k
        \end{bmatrix}^T
        \in {\R^{-}}^k,
    \end{equation}

    and, re-inserting the preceding expression for $\lambda$ back into relation~\eqref{equ:KKT},

    \begin{equation}\label{equ:kkt_below_k}
                 Z^{(k)} = X_{> k} - (\mathcal{A}^{-1})_{> k \leq k} Y^{(k)}= \Tilde{z}_{>k} \in {\mathbb{R}^{+\star}}^{p-k} \times \R^{K-p}.
    \end{equation}

 From the same argument as for the first term, it follows from~\eqref{equ:KKT_above_k} that conditionally on $\{\Tilde{z} \in C^{I_k} \}$,

    \begin{equation*}
        \min_{z \in C}  q^\mathcal{A}_X(z) 
        =
        \lambda^T (\mathcal{A}^{-1})_{\leq k \leq k } \lambda
        = \sum_{i \in I_k } \frac{X_i^2}{\mathcal{A}^{-1}_{ii}}
        \hspace{0.5cm}
        \text{and}
        \hspace{0.5cm}
        \ell^\mathcal{A}(X)
        =
        \sum_{i \notin I_k } \frac{X_i^2}{\mathcal{A}^{-1}_{ii}}
        \sim \chi^2(p-i).
    \end{equation*}
    
    Note furthermore that relation~\eqref{equ:KKT_above_k} implies $\lambda_i$ taking value $0$ is a negligible event and the $\lambda$'s can be taken almost surely strictly positive.   Reciprocally, when conditions~\eqref{equ:KKT_above_k} and~\eqref{equ:kkt_below_k} hold, that is, $Y^{(k)} < 0$ and $Z^{(k)}_{\leq p-k}>0$ coordinate-wise, the pair $(x,\lambda) \in C^{I_k} \times {R^{+\star}}^k$ defined by $\lambda=- (\mathcal{A}^{-1}_{\leq k \leq k})^{-1} X_{\leq k}$ and $\Tilde{z}= \begin{bmatrix}
        0 & \cdots & 0 & Z^{(k)} 
    \end{bmatrix}$ satisfy the \textsc{kkt} conditions for the minimisation problem of $q^\mathcal{A}_X$ over $C$. Its solution $\Tilde{z}$ then lies in $C^{I_k}$ and therefore in $C_k$. We have shown the event $\{ \Tilde{z} \in C^{\segN{1,k}}\}$ can be rewritten as $\{ Y^{(k)} <0, Z^{(k)}_{\leq p-k}>0\}$. Remarking furthermore from simple matrix manipulations and Schur's formulas that

\begin{align*}
    \mathbb{C}\text{ov}(Y^{(k)},Y^{(k)}) 
    &=  ((\mathcal{A}^{-1})_{\leq k\leq k})^{-1}\\
    \mathbb{C}\text{ov}(Z^{(k)},Y^{(k)}) 
    &= 
    (\mathcal{A}^{-1})_{>k\leq k} ((\mathcal{A}^{-1})_{\leq k\leq k})^{-1}
    -
    (\mathcal{A}^{-1})_{>k\leq k} \mathbb{C}\text{ov}(Y^{(k)},Y^{(k)}) = 0\\
    \mathbb{C}\text{ov}(Z^{(k)},Z^{(k)}) 
    &= 
    (\mathcal{A}^{-1})_{>k > k}
    -
     (\mathcal{A}^{-1})_{>k\leq k}  ((\mathcal{A}^{-1})_{\leq k\leq k})^{-1} (\mathcal{A}^{-1})_{\leq k > k} 
     =
     (\mathcal{A}_{>k,>k})^{-1},
\end{align*}

one observes $Y^{(k)}$ and $Z^{(k)}$ to be independent. Up to a re-arrangement of the coordinate of $X$, the same arguments hold for any other $I_k \in \mathcal{P}^n_k$  Summing over all possible configurations for $I_k$,  this gives the announced expression for the mixture weight $\omega^\mathcal{A}_k$. There remains only the situation of $\Tilde{z}$ falling in the interior of $C$, in which case the second term of~\eqref{equ:chi_bar_definition} vanishes and the distribution of $\ell^\mathcal{A}_X$ is a  $\chi^2(p)$ law, as would traditionally be observed in absence of positivity constraints. This happens with probability $\Prob[ X_{\leq p} \in {\R^+}^p ]$.

\end{proof}

\begin{remark}
    In all generality, the expression for the weights have no closed form. In the case $K=p=2$ however,
    $
        \Prob[ F \in \R^- \times \R^-]
        =
        \Prob[ H \in \mathcal{A}^{- \nicefrac{1}{2}}\R^{-} \times \R^- ]
    $
    where $H \sim \mathcal{N}(0,I_2)$ and re-expresses as the portion of the unit ball intersecting with $\mathcal{A}^{- \nicefrac{1}{2}}\R^{-} \times \R^-$, which is given by the angle between $\mathcal{A}^{- \nicefrac{1}{2}} \begin{bmatrix}
        -1 &  0
    \end{bmatrix}$ and $\mathcal{A}^{- \nicefrac{1}{2}} \begin{bmatrix}
        0 &  -1    \end{bmatrix}$, that is, foregoing a few matrix manipulation,

    \begin{align*}
        \omega_2 
        =
        \frac{1}{2\pi}
        \cos^{-1} \big( \nicefrac{\mathcal{A}_{12}}{\sqrt{\mathcal{A}_{11}\mathcal{A}_{22}}}\big)
        =
        \frac{1}{2 \pi}
        \big\{
            \pi
            - 
            \cos^{-1} \big( \nicefrac{\mathcal{A}^{-1}_{12}}{\sqrt{\mathcal{A}^{-1}_{11}\mathcal{A}^{-1}_{22}}}\big)
        \big\}
        =
        \frac{1}{2}
        -
        \omega_0.
    \end{align*}

    This is example 7 from Self \& Liang~\cite{Self-Liang}.
\end{remark}

\end{appendices}

\bibliographystyle{plain}
\bibliography{bibli}

\end{document}